\documentclass[english,12pt,dvips]{book}
\usepackage[applemac]{inputenc} 
\usepackage[T1]{fontenc} 

\usepackage[english]{babel}
\usepackage{color}
\usepackage[dvips]{geometry}
\usepackage{titlesec} 
\usepackage[dvips,ps2pdf,hypertex]{hyperref} 
\usepackage{helvet}
\usepackage{fix-cm}
\usepackage{fancyhdr}

\usepackage{amssymb, latexsym, amsmath, amscd, oldlfont} 
\usepackage[all]{xy} 
\usepackage{stmaryrd}
\usepackage{array}
\usepackage{amsfonts}  
\usepackage{amsthm}
\usepackage{euscript}
\usepackage{dsfont}

\usepackage[dvips]{graphicx}
\usepackage{psfrag}
\usepackage{epsf, subfigure, verbatim}
\usepackage{epic}

\usepackage[active]{srcltx}

\newcommand{\Dcirc}{\overset%
{\raisebox{-.3ex}[0ex][-.3ex]{\mbox{$\scriptscriptstyle\circ$}}\mskip-5mu}D}

\def\Hess{{\text{Hess }}}

\def\al{{\alpha}}
\def\be{{\beta}}
\def\ep{{\varepsilon}}
\def\eps{{\epsilon}}
\def\de{{\delta}}
\def\ga{{\gamma}}
\def\Ga{{\Gamma}}
\def\io{\iota}
\def\la{{\lambda}}
\def\La{{\Lambda}}
\def\om{{\omega}}
\def\Om{{\Omega}}

\def\ve{\varepsilon}
\def\vp{\varphi}
\def\si{\sigma}
\def\Si{\Sigma}
\def\rh{\rho}

\def\ra{\rightarrow}

\def\lg{\langle}
\def\rg{\rangle}

\def\A{{\mathcal A}}
\def\cc{{\mathcal C}}
\def\E{{\mathcal E}}
\def\fc{{\mathcal F}}
\def\H{{\mathcal H}}
\def\Js{{\mathcal J}}
\def\Jr{{\mathcal J}_{reg}}
\def\L{{\mathcal L}}
\def\M{{\mathcal M}}
\def\Mod{\M(x^-,x^+,H;\J)}
\def\Modti{\widetilde{\M}(x^-,x^+,H;\J)}

\def\Od{{\mathcal O}}
\def\P{{\mathcal P}}
\def\Sp{{\mathcal S}}

\def\C{{\mathbb C}}
\def\F{{\mathbb F}}
\def\N{{\mathbb N}}
\def\Pr{{\mathbb P}}

\def\R{{\mathbb R}}

\def\Z{{\mathbb Z}}

\def\aa{{\mathbf a}}
\def\J{{\mathbf J}}

\def\Tm{T^*M}

\def\demi{\frac{1}{2}}

\def\gl{G^-_n}
\def\gg{{G^+_n}}
\def\gs{G_{n,s}}
\def\hl{H^-}
\def\hg{H^+}
\def\hs{H_s}

\def\Sg{{S_\ga}}
\def\fg{{f_\ga}}
\def\lga{{l_\ga}}
\def\Fd{{F_\delta}}
\def\Fnd{{F_{n,\delta}}}
\def\gam{{\ga_{\min}}}
\def\gaM{{\ga_{\max}}}
\def\oga{{\overline{\ga}}}
\def\uga{{\underline{\ga}}}
\def\Modgg{\M(S_\oga, S_\uga; H, \J)}

\def\Modpq{\M(p, q; H, \{f_\ga \}, \J)}
\def\Modmpq{\M_m (p, q; H, \{f_\ga \}, \J)}

\def\evo{{\overline{\rm ev}}}
\def\evu{{\underline{\rm ev}}}
\def\Fr{{\mathcal F}_{reg}}
\DeclareMathOperator{\ind}{ind}

\def\ms{\medskip}
\def\ni{\noindent}

\def\llb{\llbracket}
\def\rrb{\rrbracket}

\newcommand{\re}{{\rm I\!R}}

\newcommand{\qe}{{\rm I\hspace{-6pt}Q}}

\newcommand{\id}{{\rm Id}}

\DeclareMathOperator{\im}{im}
\DeclareMathOperator{\sign}{sign}

\DeclareMathOperator{\crit}{crit}
\DeclareMathOperator{\rank}{rank}
\DeclareMathOperator{\diag}{diag}
\DeclareMathOperator{\Mat}{Mat}

\newtheorem{thm}{Theorem}
\newtheorem{cor}{Corollary}
\newtheorem{lem}{Lemma}[section]
\newtheorem{prop}{Proposition}[section]
\newtheorem{definition}{Definition}[section]

\newtheorem*{thm1}{Theorem A}
\newtheorem*{thm2}{Theorem B}
\newtheorem*{thm3}{Theorem C}
\newtheorem*{conj}{Weinstein conjecture}
\newtheorem*{thm*}{Theorem}
\newtheorem*{thmG}{Theorem (Gromov)}

\newtheorem*{definition*}{Definition}

\theoremstyle{remark}
\newtheorem{rmk}{Remark}
\newtheorem*{rmk*}{Remark}

\theoremstyle{remark}

\theoremstyle{remark}

\renewcommand{\headrulewidth}{0.5pt}
\renewcommand{\footrulewidth}{0pt}
\fancypagestyle{plain}{ 
\fancyhead{} 
\renewcommand{\headrulewidth}{0pt} 
}

\newcommand{\clearemptydoublepage}{%
\newpage{\pagestyle{empty}\cleardoublepage}}

\begin{document}
\pagenumbering{roman}


\begin{titlepage}
\begin{center}
\begin{Large}\textbf{Universit\'e Libre de Bruxelles} \end{Large} \\
\vspace{0.5cm}
\begin{Large}\textbf{Universit\'e de Neuch\^atel}\end{Large}\\
\vspace{5cm}
\begin{huge}\textbf{The Weinstein conjecture with multiplicities on spherizations}\end{huge}\\
\vspace{2cm}
Th\`ese pr\'esent\'ee en vue de l'obtention\\
du grade de Docteur en Sciences\\
\vspace{0.25cm}
par\\
\vspace{0.25cm}
\begin{large} Muriel HEISTERCAMP \end{large}

\vspace{3cm}
\end{center}
\begin{flushleft}
\textbf{Directeurs de th\`ese}: F. Schlenk, Universit\'e de Neuch\^atel\\
	\hspace{4.31cm}F. Bourgeois, Universit\'e Libre de Bruxelles
			
\end{flushleft}
\begin{center}
\begin{large}Ann\'ee acad\'emique 2010-2011\end{large}
\end{center}
\fancyhf{}
\newpage
$\mbox{}$
\newpage

\vspace{10cm}
$\mbox{}$\\
$\mbox{}$\\
$\mbox{}$\\
$\mbox{}$\\
$\mbox{}$\\
$\mbox{}$\\
$\mbox{}$\\
$\mbox{}$\\
$\mbox{}$\\
$\mbox{}$\\
$\mbox{}$\\
$\mbox{}$\\
$\mbox{}$\\
$\mbox{}$\\
$\mbox{}$\\
$\mbox{}$\\
$\mbox{}$\\
$\mbox{}$\\
\begin{center}
\textit{\`A Nelle.}
\end{center}

\end{titlepage}
\clearemptydoublepage

\fancyhf{} 
\fancyhead[LE,RO]{\bfseries\thepage} \fancyhead[LO]{\bfseries
Contents} \fancyhead[RE]{\bfseries Contents}
\renewcommand{\headrulewidth}{0.5pt}
\renewcommand{\footrulewidth}{0pt}
\addtolength{\headheight}{0.5pt} 

\tableofcontents

\clearemptydoublepage

\pagenumbering{roman}%
\titleformat{\chapter}[display]
{\normalfont\filcenter\huge\bf}{\chaptertitlename\ \thechapter}
{1ex} {\titlerule
  \vspace{1ex}%
   \filcenter\LARGE\bf}
  [\vspace{2ex}%
  ]

\chapter*{Abstract}
\addcontentsline{toc}{chapter}{\protect\numberline{}{Abstract}}

\fancyhf{} 
\fancyhead[LE,RO]{\bfseries\thepage} \fancyhead[LO]{\bfseries
Introduction}
\fancyhead[RE]{\bfseries Abstract}
\renewcommand{\headrulewidth}{0.5pt}
\renewcommand{\footrulewidth}{0pt}
\addtolength{\headheight}{0.5pt} 
\fancypagestyle{plain}{ 
\fancyhead{} 
\renewcommand{\headrulewidth}{0pt}}

Let $M$ be a smooth closed manifold and $\Tm$ its cotangent bundle endowed with the usual symplectic structure $\om=d\la$, where $\la$ is the Liouville form. A hypersurface $\Si\subset T^*M$ is said to be {\it fiberwise starshaped} if for each point $q\in M$ the intersection $\Si_q := \Si \cap T^*_q M$ of $\Si$ with the fiber at $q$ is starshaped with respect to the origin $0_q \in T^*_q M$. \\

\ni
In this thesis we give lower bounds of the growth rate of the number of closed Reeb orbits on a {\it fiberwise starshaped hypersurface} in terms of the topology of the free loop space of $M$. We distinguish the two cases that the fundamental group of the base space $M$ has an exponential growth of conjugacy classes or not. If the base space $M$ is simply connected we generalize the theorem of Ballmann and Ziller on the growth of closed geodesics to Reeb flows.

\clearemptydoublepage

\chapter*{Acknowlegments}
\addcontentsline{toc}{chapter}{\protect\numberline{}{Acknowlegments}}

\fancyhf{} 
\fancyhead[LE,RO]{\bfseries\thepage} \fancyhead[LO]{\bfseries
Introduction}
\fancyhead[RE]{\bfseries Acknowlegments}
\renewcommand{\headrulewidth}{0.5pt}
\renewcommand{\footrulewidth}{0pt}
\addtolength{\headheight}{0.5pt} 
\fancypagestyle{plain}{ 
\fancyhead{} 
\renewcommand{\headrulewidth}{0pt}}

I would like to thank my advisor Felix Schlenk for his support troughout this work and the opportunity to complete this project. Furthermore I thank Fr\'ed\'eric Bourgeois, Urs Frauenfelder, Agn\`es Gadbled and Alexandru Oancea for valuable discussions. I also thank Will Merry and Gabriel Paternain for having pointed out to me that the methods used in the proof of Theorems A and B can also be used to prove Theorem C. 

\clearemptydoublepage

\chapter*{Introduction}
\addcontentsline{toc}{chapter}{\protect\numberline{}{Introduction}}

\fancyhf{} 
\fancyhead[LE,RO]{\bfseries\thepage} \fancyhead[LO]{\bfseries
Introduction}
\fancyhead[RE]{\bfseries Introduction}
\renewcommand{\headrulewidth}{0.5pt}
\renewcommand{\footrulewidth}{0pt}
\addtolength{\headheight}{0.5pt} 
\fancypagestyle{plain}{ 
\fancyhead{} 
\renewcommand{\headrulewidth}{0pt}}

Let $M$ be a smooth closed manifold and denote by $T^*M$ the cotangent bundle over $M$ endowed with its usual symplectic structure $\om=d\la$ where $\la=p\,dq=\sum_{i=1}^m p_i \, dq_i$ is the Liouville form. A hypersurface $\Si\subset T^*M$ is said to be {\it fiberwise starshaped} if for each point $q\in M$ the intersection $\Si_q := \Si \cap T^*_q M$ of $\Si$ with the fiber at $q$ is starshaped with respect to the origin $0_q \in T^*_q M$. There is a flow naturally associated to $\Si$, generated by the unique vector field $R$ along $\Si$ defined by 
$$ d\la (R,\cdot)=0, \quad \la(R) = 1.$$ 
The vector field $R$ is called the {\it Reeb vector field} on $\Si$ and its flow is called the {\it Reeb flow}. The main result of this thesis is to prove that the topological structure of $M$ forces, for all fiberwise starshaped hypersurfaces $\Si$, the existence of many closed orbits of the Reeb flow on $\Si$. More precisely, we shall give a lower bound of the growth rate of the number of closed Reeb-orbits in terms of their periods.\\

\ni 
The existence of one closed orbit was conjectured by Weinstein in 1978 in a more general setting.
\begin{conj}
A hypersurface $\Si$ of contact type and satisfying $H^1 (\Si) = 0$ carries a closed characteristic. 
\end{conj}
\ni
Independently, Weinstein \cite{Wein78} and Rabinowitz \cite{Rab78} established the existence of a closed orbit on star-like hypersurfaces in $\re^{2n}$. In our setting the Weinstein conjecture without the assumption $H^1 (\Si) = 0$ was proved in 1988 by Hofer and Viterbo, \cite{HV88}. The existence of many closed orbits has already been well studied in the special case of the geodesic flow, for example by Gromov \cite{Gr78}, Paternain \cite{Pa97,Pa99} and Paternain--Petean \cite{PP06}. In this thesis we will generalize their results.\\

\ni
The problem at hand can be considered in two equivalent ways. First, let $H\colon T^*M \ra \re$ be a smooth Hamiltonian function such that $\Si$ is a regular level of $H$. Then the Hamiltonian flow $\vp_H$ of $H$ is orbit-equivalent to the Reeb flow. Therefore, the growth of closed orbits of $\vp_H$ equals the growth of closed orbits of the Reeb flow. Secondly, let $S M$ be the cosphere bundle over $M$ endowed with its canonical contact structure $\xi = \ker \la$. The contact manifold $(S M, \xi) $ is called the spherization of $M$. Our main results are equivalent to saying that for {\it any} contact form $\al$ for $\xi$, i.e. $\xi=\ker \al$, the growth rate of the number of closed orbits of the Reeb flow of $\al$ in terms of their period depends only on $M$ and is bounded from below by homological data of $M$.

\subsubsection{The free loop space}
The complexity of the Reeb flow on $\Si\subset T^*M$ comes from the complexity of the free loop space of the base manifold $M$. Let $(M,g)$ be a $C^\infty$-smooth, closed, connected Riemannian manifold. Let $\La M$ be the free loop space of $M$, i.e. the set of loops $q:S^1 \ra M$ of Sobolev class $W^{1,2}$. This space has a canonical Hilbert manifold structure, see \cite{Kli95}. The energy functional $\E= \E_g : \La M \ra \re$ is defined by 
$$
\E (q) := \demi \int_0^1 |\dot{q}(t)|^2 \,dt
$$
where $|q(t)|^2 = g_{q(t)}(\dot{q}(t),\dot{q}(t))$. For $a>0$ we consider the sublevel sets
$$
\La^a := \{q\in \La M \mid \E(q) \leq a \}.
$$
Now let $\Pr$ be the set of prime numbers and write $\Pr_0 := \Pr \cup \{0\}$. For each prime number $p$ denote by $\F_p$ the field $\Z / p \Z$, and write $\F_0 := \qe$. Throughout, $H_*$ will denote singular homology and
$$
\io_k: H_k (\La^a ;\F_p) \ra H_k(\La M;\F_p)
$$
the homomorphism induced by the inclusion $\La^a M \hookrightarrow \La M$. Following \cite{FS06} we make the  

\begin{definition*}
The Riemannian manifold $(M,g)$ is energy hyperbolic if 
$$
C(M,g) := \sup_{p\in \Pr} \liminf_{n\ra\infty} \frac{1}{n} \log \sum_{k\geq 0} \dim \io_k H_k \bigl(\La^{\demi n^2} ;\F_p \bigr) >0.
$$
\end{definition*}

\ni
Since $M$ is closed, the property {\it energy hyperbolic} does not depend on $g$ while, of course, $C(M,g)$ does depend on $g$. We say that the closed manifold $M$ is {\it energy hyperbolic} if $(M,g)$ is energy hyperbolic for some and hence for any Riemannian metric $g$ on $M$.\\

\ni
We also consider the {\it slow growth} of the homology given by
$$
c(M,g) := \sup_{p\in \P} \liminf_{n\ra\infty} \frac{1}{\log n} \log \sum_{k\geq 0} \dim \io_k H_k \bigl(\La^{\demi n^2} ;\F_p \bigr).
$$

\ni
Denote by $\La_\al M$ the component of a loop $\al$ in $\La M$ and by $\La_0 M$ the component of contractible loops. The components of the loop space $\La M$ are in bijection with the set $\cc (M)$ of conjugacy classes in the fundamental group $\pi_1(M)$, i.e. 
$$
\La M \,=\, \coprod_{c\, \in \, \cc (M)} \La_c M .
$$

\ni
For each element $c\in \cc(M)$ denote by $e(c)$ the infimum of the energy of a closed curve representing $c$. Let $\cc^a (M) := \left\{ c \in \cc(M) \mid e(c) \le a \right\}$, and define
\begin{eqnarray*}
E(M) &:=& 
\liminf_{a \to \infty} \frac 1 a    \log \# \cc^a(M) , \\
e(M) &:=& 
\liminf_{a \to \infty} \frac 1{\log a} \log \# \cc^a(M) .
\end{eqnarray*}
Note that $E(M)$ and $e(M)$ do not depend on the metric~$g$ and that $C(M,g) \geq E(M)$.

\subsubsection{Fiberwise starshaped hypersurfaces in $T^*M$}
Let $\Si$ be a smooth connected hypersurface in $T^*M$. We say
that $\Si$ is {\it fiberwise starshaped}\, if for each point $q
\in M$ the set $\Si_q := \Si \cap T_q^*M$ is the smooth
boundary of a domain in $T_q^*M$ which is strictly starshaped with
respect to the origin $0_q \in T^*M$. This means that the radial
vector field $\sum_i p_i\;\partial p_i$ is transverse to each
$\Si_q$. We assume throughout that $\dim M \geq 2$. 
Then $T^*M \setminus \Sigma$ has two components, 
the bounded inner part $\Dcirc (\Sigma)$ containing the zero section 
and the unbounded outer part $D^c(\Sigma) = T^*M \setminus D(\Sigma)$,
where $D(\Sigma)$ denotes the closure of $\Dcirc (\Sigma)$.

\subsubsection{Formultation of the results}
Let $\Si \subset T^*M$ be as above and denote by $\vp_R$ the Reeb flow on $\Si$. For $\tau>0$ let $\Od_R (\tau)$ be the set of closed orbits of $\vp_R$ with period $\leq \tau$. We measure the growth of the number of elements in $\Od_R (\tau)$ by

\begin{eqnarray*}
N_R &:=& \liminf_{\tau \to \infty} \frac 1 \tau \log \left( \# \Od_R (\tau) \right) ,\\
n_R &:=& \liminf_{\tau \to \infty} \frac 1{\log \tau} \log \left( \# \Od_R (\tau) \right) . 
\end{eqnarray*}

\ni
The number $N_R$ is the {\it exponential growth rate} of closed orbits, while $n_R$ is the {\it polynomial growth rate}. The following three theorems are the main result of this thesis.

\begin{thm1}
Let $M$ be a closed, connected, orientable, smooth manifold and let $\Si \subset T^*M$ be a fiberwise starshaped hypersurface. Let $\vp_R$ be the Reeb flow on $\Si$, and let $N_R$, $n_R$, $E(M)$ and $e(M)$ be defined as above. Then

\begin{itemize}
\item[(i)]
$N_R \geq E(M)$;
\item[(ii)]
$n_R \geq e(M) - 1$.
\end{itemize}

\end{thm1}

\ni
We will say that $\Si$ is generic if each closed Reeb orbit is transversally nondegenerate, i.e.
$$\det(1-d\vp_R^\tau (\ga(0))|_\xi) \neq 0.$$

\begin{thm2}
Let $M$ be a closed, connected, orientable smooth manifold and let $\Si \subset \Tm$ be a generic fiberwise starshaped hypersurface. Let $\vp_R$ be the Reeb flow on $\Si$, and let $N_R$, $n_R$, $C(M,g)$ and $c(M,g)$ be defined as in section \ref{section:free loop space}. Then

\begin{itemize}
\item[(i)]
$N_R \geq C(M,g)$.

\item[(ii)]
$n_R \geq c(M,g)-1$.
\end{itemize}
\end{thm2}

\ni
The hypothesis of genericity of $\Si$ will be used to achieve a Morse-Bott situation for the action functional that we will introduce.\\

\ni
The idea of the proofs is as follows. Let $\Si\subset \Tm$ be a fiberwise starshaped hypersurface. If $\Si$ is the level set of a Hamiltonian function $F: \Tm \ra \re$, then the Reeb flow of $\la$ is a reparametrization of the Hamiltonian flow. We can define such a Hamiltonian by the two conditions
\begin{equation}
F|_\Si \equiv 1, \quad F(q,sp)=s^2F(q,p), \quad s\geq 0 \mbox{ and }(q,p)\in \Tm.
\end{equation}
This Hamiltonian is not smooth near the zero section, we thus define a cut-off function $f$ to obtain a smooth function $f\circ F$. We then use the idea of sandwiching develloped in Frauenfelder--Schlenk \cite{FS05} and Macarini--Schlenk \cite{MS10}. By sandwiching the set $\Si$ between the level sets of a geodesic Hamiltonian, and by using the Hamiltonian Floer homology and its isomorphism to the homology of the free loop space of $M$, we shall show that the number of 1-periodic orbits of $F$ of action $\leq a$ is bounded below by the rank of the homomorphism
$$
\io_k: H_k (\La^{a^2} ;\F_p) \ra H_k(\La M;\F_p)
$$
induced by the inclusion $\La^{a^2} M \hookrightarrow \La M$.\\

\ni
The Hamiltonian Floer homology for $F$ is not defined, since all the periodic orbits are degenerate. We thus need to consider small pertubations of $F$. In the proof of Theorem~A, we will add to $F$ small potentials of the form $V_l (t,q)$. Assuming $\| V_{l}(t,q) \|_{C^\infty} \, \ra \, 0$ for $l\ra \infty$, we will show the existence of a periodic orbit of $F$ in every non-trivial conjugacy class as the limit of periodic orbits of $F +V_l$. This strategy cannot be applied for Theorem~B. We thus use the assumption of genericity to achieve a Morse--Bott situation following Frauenfelder \cite[Appendix A]{Frauen04} and Bourgeois--Oancea \cite{BO09} and use the {\it Correspondence Theorem} between Morse homology and Floer homology due to Bourgeois--Oancea, \cite{BO09}, to obtain our result.

\begin{rmk*}
A proof of rough versions of Theorems A and B is outlined in Section 4a of Seidel's survey \cite{Sei08}.
Meanwhile, a different (and difficult) proof of these theorems, with coefficients in $\Z_2$ only, was given by Macarini--Merry--Paternain in \cite{MMP11}, where a version of Rabinowitz--Floer homology is contructed to give lower bounds for the growth rate of leaf-wise intersections. 
\end{rmk*}


\subsubsection{Spherization of a cotangent bundle}
The hyperplane field $\xi|_\Si = \ker \la|_\Si \subset T\Si$ is a contact structure on $\Si$. If $\Si'$ is another fiberwise starshaped hypersurface, then $(\Si, \xi_\Si)$ and $(\Si', \xi_{\Si'})$ are contactomorphic. In fact the differential of the diffeomorphism obtained by the radial projection maps $\xi_\Si$ to $\xi_{\Si'}$. The identification of these contact manifolds is called the {\it spherization} $(S M, \xi)$ of the cotangent bundle $(T^*M,\om)$.  Theorem~A and Theorem~B gives lower bounds of the growth rate of closed orbits for {\it any Reeb flow on the spherization $S M$ of $\Tm$}.\\

\ni 
Special examples of fiberwise starshaped hypersurfaces are unit cosphere bundles $S_1 M(g)$ associated to a Riemmanian metric $g$,
$$
S_1 M(g) := \{(q,p) \in \Tm \mid |p| = 1\}. 
$$
The Reeb flow is then the geodesic flow. In this case, Theorem~A is a direct consequence of the existence of one closed geodesic in every conjugacy classes. If $M$ is simply connected, Theorem~B for geodesic flows follows from the following result by Gromov~\cite{Gr78} 
\begin{thmG}
Let $M$ be a compact and simply connected manifold. Let $g$ be a bumpy Riemannian metric on $M$. Then there exist constants $\al = \al(g) >0$ and $\be =\be(g)>0$ such that there are at least 
$$
 \frac{\al \displaystyle\sum_{i=1}^{\be t} b_i(M)}{t}
$$
periodic geodesics of length less than $t$, for all $t$ sufficiently large. 
\end{thmG}

\ni
The assumption on the Riemannian metric to be {\it bumpy} corresponds to our genericity assumption. Generalizations to geodesic flows of larger classes of Riemannian manifolds were proved in Paternain \cite{Pa97,Pa99} and Paternain--Petean \cite{PP06}.\\

\ni 
The exponential growth of the number of Reeb chords in spherizations is studied in \cite{MS10}. Results on exponential growth rate of the number of closed orbits for certain Reeb flows on a large class of closed contact 3-manifolds are proved in \cite{CH08}.

\subsubsection{The simply connected case}
In \cite{BZ82} Ballman and Ziller improved Gromov's theorem in the case of simply connected Riemannian manifolds with bumpy metrics. They showed that the number $N_g(T)$ of closed geodesics of length less than or equal to $T$ is bounded below by the maximum of the $kth$ betti number of the free loop space $k\leq T$, up to some constant depending only on the metric. Following their idea we shall prove the following

\begin{thm3}
\label{prop:ballmann-ziller}
Suppose that $M$ is a compact and simply connected $m$-dimensional manifold. Let $\Si$ be a generic fiberwise starshaped hypersurface of $T^*M$ and $R$ its associated Reeb vector field. Then there exist constants $\al = \al(R) >0$ and $\be = \be(R)>0$ such that
$$
\# \Od_R (\tau) \geq \al \max_{1 \leq i\leq \be \tau} b_i (\La M)
$$
for all $\tau$ sufficiently large.
\end{thm3}

\subsubsection{Two questions}

\begin{itemize}
\item[I.]We assume the hypersurface $\Si$ to be fiberwise starshaped {\it with respect to the origin}. Could this assumption be omitted? In the case of Reeb chords it cannot, see \cite{MS10}. 

\item[II.]The assumption on $\Si$ to be fiberwise starshaped is equivalent to the assumption that $\Si$ is of restricted contact type with respect to the Liouville vector field $Y=p\partial p$. Are Theorem~A and Theorem~B true for any hypersurface $\Si \subset T^*M$ of restricted contact type? 
\end{itemize}

\medskip
\ni
The thesis is organized as follows: In Chapter~1 we introduce the definitions and tools that we will use throughout this work. Chapter~2 provides the tool of sandwiching used here to compare the growth of closed Reeb orbits with the growth of closed geodesics. In Chapter~3 we recall the definition of Morse--Bott homology which is used in the proof of Theorem~B. In Chapter~4 we prove Theorem~A, Theorem~B and Theorem~C. In Chapter~5 we shall evaluate our results on several examples introduced in Chapter~1. 

In Appendix~A we review some tools to prove the compactness of moduli spaces introduced in section \ref{section:definition of hf}. In Appendix~B we recall the definition of the Legendre transform. In Appendix~C we give a proof of the existence of Gromov's constant, see Theorem~\ref{thm:gromov}.

\clearemptydoublepage

\pagenumbering{arabic}
 
 \setcounter{page}{1}
 
\fancyhf{} 
\fancyhead[LE,RO]{\bfseries\thepage} \fancyhead[LO]{\bfseries
\leftmark} \fancyhead[RE]{\bfseries \rightmark}

\chapter{Definitions and Tools}

In this chapter we introduce the definitions and tools that we will use throughout this work. In section \ref{section:free loop space} we describe the free loop space $\La M$ of a manifold $M$ and introduce topological invariant measuring the topological complexity of the free loop space. Section \ref{section:cotangent bundles} gives an overview of Hamiltonian dynamic on cotangent bundles and fiberwise starshaped hypersurfaces. We discuss the relation between Reeb orbits on a fiberwise starshaped hypersurface and the 1-periodic orbits of a Hamiltonian flow for which the hypersurface is an energy level. In section \ref{maslov index} we recall the definitions and properties of Maslov type indexes introduced by Conley and Zehnder in \cite{CZ84} and Robbin and Salamon in \cite{RS93}.

\section{The free loop space}
\label{section:free loop space}

Let $(M,g)$ be a connected, $C^\infty$-smooth Riemannian manifold. Let $\La M$ be the set of loops $q:S^1 \ra M$ of Sobolev class $W^{1,2}$. $\La M$ is called the {\it free loop space of $M$}. This space carries a canonical structure of  Hilbert manifold, see \cite{Kli95}. \\

\ni
The energy functional $\E= \E_g : \La M \ra \re$ is defined as 
$$
\E (q) := \demi \int_0^1 |\dot{q}(t)|^2 \,dt
$$
where $|q(t)|^2 = g_{q(t)}(\dot{q}(t),\dot{q}(t))$. It induces a filtration on $\La M$. For $a>0$, consider the sublevel sets $\La^a \subset \La M$ of loops whose energy is less or equal to $a$,
$$
\La^a := \{q\in \La M \mid \E(q) \leq a \}.
$$

\ni
The length functional $\L := \L_g: \La M \ra\re$ is defined by 
$$
\L(q) = \int_0^1 |\dot{q}(t)| dt.
$$
Similarly, for $a>0$ we can consider the sublevel sets 
$$
\L^a := \{q\in \La M \mid \L(q) \leq a\}.
$$

\ni
Applying Schwarz's inequality
$$
\left( \int_0^1 fg \, dt \right)^2 \leq \left( \int_0^1 f^2 \, dt \right) \left( \int_0^1 g^2 \, dt \right)
$$
with $f(t)=1$ and $g(t) = |\dot{q}(t)|$ we see that
$$
\demi \L^2 (q) \leq \E(q),
$$
where equality holds of an only if $q$ is parametrized by arc-length.\\

\ni
Denote by $\La_\al M$ the connected component of a loop $\al$ in $\La M$. The components of the loop space $\La M$ are in bijection with the set $\cc (M)$ of conjugacy classes of the fundamental group $\pi_1(M)$,  
$$
\La M \,=\, \coprod_{c\, \in \, \cc (M)} \La_c M .
$$

\subsubsection{Counting by counting conjugacy classes in $\pi_1$}

Let $X$ be a path-connected topological space. Denote by $\cc(X)$ the set of conjugacy classes in $\pi_1(X)$ and by $\fc(X)$ the set of free homotopy classes in $\La X$. Given a loop $\al: (S^1,0) \ra (X,x_0)$ we will denote its based homotopy class in $\pi_1(X)$ by $[\al]$ and its free homotopy class in $\fc(X)$ by $\llb \al \rrb$.

\begin{prop}
Let $X$ be a path-connected topological space and $x_0$ a base point. Then 
$$\Phi: \cc(X)\ra \fc(X): [\al] \mapsto \llb \al \rrb$$
is a bijection between the set of conjugacy classes in $\pi_1(X)$ and the set of free homotopy classes in $\La X$. 

Furthermore, if $f: (X,x_0) \ra (Y,y_0)$ is a continuous map between based topological spaces, we have
$$\Phi \circ f_* = f_* \circ \Phi.$$ 
\end{prop}
 
\begin{proof}
Let $f,g : (S^1,0) \ra (X,x_0)$ be two continuous maps. If $f$ is homotopic to $g$ then $f$ is also freely homotopic to $g$. Thus we get a well defined map 
$$
\overline{\Phi}: \pi_1(X) \ra \fc(X)
$$
sending a based homotopy class $[\ga]$ to its free homotopy class $\llb \ga \rrb$. Let $\ga_0, \ga_1,\al : (S^1,0) \ra (X,x_0)$ such that 
$$
[\al] [\ga_0] [\al]^{-1} = [\ga_1]
$$
which is equivalent to
$$
[\al \cdot \ga_0 \cdot \al^{-1}] = [\ga_1].
$$
Consider the homotopy $F:[0,1] \times [0,1] \ra X$ defined by $F(s,t) = \al ( 1-(1-s)(1-t))$. Then $F(0,t) = \al(t)$ and $F(1,t) = x_0$, meaning that $F$ is a free homotopy of curves from $\al$ to $x_0$. Using $F$, one can construct a free homotopy of loops between $ \al \cdot \ga_0 \cdot \al^{-1}$ and $\ga_0$.
As $\al \cdot \ga_0 \cdot \al^{-1}$ is based homotopic to $\ga_1$ it follows that $\ga_0$ is free homotopic to $\ga_1$ and thus $\overline{\Phi}([\ga_0]) = \overline{\Phi} ([\ga_1])$. Thus $\overline{\Phi}$ descends to a map 
$$
\Phi : \cc(X) \ra \fc(X).
$$
\\

\ni
Now consider a loop $\ga : S^1 \ra X$ and take a continuous path $\al: [0,1] \ra X$ with $\al(0) = \ga (0)$ and $\al (1) = x_0$. Then $\al \cdot \ga \cdot \al^{-1}$ is a continuous loop with base point $x_0$ which is freely homotopic to $\ga$. This implies that $ \Phi ([\al \cdot \ga \cdot \al^{-1}]) = \llb \ga \rrb$ which yields the surjectivity of $\Phi$. \\

\ni
Let $[f_0]$ and $[f_1]$ be two elements of $\pi_1(X,x_0)$ with $\overline{\Phi} ([f_0]) = \overline{\Phi} ([f_1])$ and $H: [0,1] \times S^1 \ra X$ a free homotopy from $f_0$ to $f_1$. Define $g: S^1 \ra X$ by $g(s) := H(s,0)$. Then $g \cdot f_0 \cdot g^{-1}$ is homotopic to $f_1$ and thus $[f_0]$ and $[f_1]$ are conjugate. This proves the injectivity of $\Phi$. \\

\ni
The naturality follows from the definition of $\Phi$ as

$$\begin{aligned}
\Phi \circ f_* \,([\ga]) & = \Phi ([f\circ \ga]) \\
 & = \llb f \circ \ga \rrb \\
 & = f_*\, \llb \ga \rrb \\
 & = f_* \, \Phi ([\ga]).
\end{aligned}$$  

\end{proof}

\subsection{Growth coming from $\cc(M)$} 
\label{section:growth from cc}

Consider the set $\cc(M)$ of conjugacy classes of the fundamental group $\pi_1(M)$. For each element $c\in \cc(M)$ denote by $e(c)$ the infimum of the energy of a closed curve representing $c$. We denote by $\cc^a(M)$ the set of conjugacy classes whose elements can be represented by a loop of energy at most $a$, 
$$
\cc^a (M) := \left\{ c \in \cc(M) \mid e(c) \le a \right\}.
$$
The exponential and polynomial growth of the number of conjugacy classes as a function of the energy are measured by  
\begin{eqnarray*}
E(M) &:=& 
\liminf_{a \to \infty} \frac 1 a    \log \# \cc^a(M) ,\quad \mbox{and} \\
e(M) &:=& 
\liminf_{a \to \infty} \frac 1{\log a} \log \# \cc^a(M) .
\end{eqnarray*}

\subsection{Energy hyperbolic manifolds}
\label{section:energy hyperbolic manifold}

Recall that for $a>0$, $\La^a$ denotes the subset of loops whose energy is less or equal to $a$,
$$
\La^a := \{q\in \La M \mid \E(q) \leq a \}.
$$
Let $\Pr$ be the set of prime numbers, and write $\Pr_0 := \Pr \cup \{0\}$. For each prime number $p$ denote by $\F_p$ the field $\Z / p \Z$, and abbreviate $\F_0 := \qe$. Throughout, $H_*$ denotes singular homology. Let
$$
\io_k: H_k (\La^a ;\F_p) \ra H_k(\La M;\F_p)
$$
be the homomorphism induced by the inclusion $\La^a M \hookrightarrow \La M$. It is well-known that for each $a$ the homology groups $H_k (\La^a M; \F_p)$ vanish for all large enough $k$, see \cite{Benci86}. Therefore, the sums in the following definition are finite. Following \cite{FS06} we make the

\begin{definition}
The Riemannian manifold $(M,g)$ is energy hyperbolic if 
$$
C(M,g) := \sup_{p\in \P} \liminf_{n\ra\infty} \frac{1}{n} \log \sum_{k\geq 0} \dim \io_k H_k \bigl(\La^{\demi n^2} ;\F_p \bigr) >0
$$
\end{definition}

\ni
Since $M$ is closed, the property {\it energy hyperbolic} does not depend on $g$ while, of course, $C(M,g)$ does depend on $g$. We say that the closed manifold $M$ is {\it energy hyperbolic} if $(M,g)$ is energy hyperbolic for some and hence for any Riemannian metric $g$ on $M$.\\

\ni
We will also consider the {\it rational growth} of the homology given by
$$
c(M,g) := \sup_{p\in \P} \liminf_{n\ra\infty} \frac{1}{\log n} \log \sum_{k\geq 0} \dim \io_k H_k \bigl(\La^{\demi n^2} ;\F_p \bigr).
$$

\ni
Fix a Riemannian metric $g$ and $p\in\Pr$. It holds that
$$
\dim \io_0 H_0 \bigl(\La^{\demi a} ;\F_p \bigr) = \# \cc^a (M).
$$
Thus $E(M)$, respectively $e(M)$, is a lower bound for $C(M,g)$, respectively $c(M,g)$.


\subsection{Examples}
\label{section:examples}

\subsubsection{Negative curvature manifolds}
\label{section:negative curvature}

Suppose our manifold $M$ carries a Riemannian metric of negative curvature. 

\begin{prop}
\label{prop:negative curvature}
If $M$ posses a Riemannian metric $g$ of negative curvature, then the component of contractible loops $\La_0 M$ is homotopy equivalent to $M$, and all other components are homotopy equivalent to $S^1$. 
\end{prop}

\ni
Using the result of the previous section, this yields

\begin{cor}
$\La M \,\simeq M \coprod_{[\al]\in\cc (M)} S^1 $
\end{cor}

\begin{proof}
Consider the energy functional $\E:= \E_g: \La M \ra \re$ with respect to the metric $g$. It's a Morse--Bott, i.e. 
$$
\crit(\E) := \{ q \in \La M \mid d\E(q) = 0 \}
$$
is a submanifold of $\La M$ and 
$$
T_q \crit(\E) = \ker (\Hess (\E) (q)).
$$
Moreover its critical points are closed geodesics. Let $c$ be a closed geodesic on $M$. Then $c$ gives rise to a whole circle of geodesics whose parametrization differ by a shift $t\in S^1$. We denote by $S_c$ the set of such geodesics. Consider the following result of Cartan~\cite[Section 3.8]{Kli95}.

\begin{thm}{\bf (Cartan)}
Let $M$ be a compact manifold with strictly negative curvature. Then there exists, up to parametrization, exactly one closed geodesic $c$ in every free homotopy class which is not the class of the constant loop. $c$ is the element of minimal length in its free homotopy clas. All closed geodesics on $M$ are of this type.  
\end{thm}

\ni
Thus $\E$ has a unique critical manifold $S_c$ in every component which is not the component of the constant loops. While the component of the constant loop has as critical manifold $S_0$ the subspace of constant loops. Moreover all the Morse indices are equal to zero. Following \cite{GM169}, one can resolve every critical submanifolds $S_c$ into finitely many non-degenerate critical points $c_1, \ldots,c_l$ corresponding to critical points of a Morse function $h: S_c \ra\re$. The index of a non-degenerate critical point $c_i$ is then given by the sum $\la+\la_i$ where $\la$ is the Morse index of $c$ with respect to $\E$ and $\la_i$ the Morse index of $c_i$ with respect to the Morse function $h$. \\

\ni
Let $a<b$ be regular values of $\E$ and $c_1, \cdot, c_k$ of $\E$ in $\E{-1}[a,b]$. Let $c_{i1}, \ldots c_{i{l_i}}$ be the corresponding non-degenerate critical points of indices $\la_{i1}, \ldots , \la_{i{l_i}}$. Then Lemma 2 of \cite{GM169} tells us that $\La^a$ is diffeomorphic to $\La^b$ with a handle of index $\la_{ij}$ attached for each non-degenerate critical point $c_{ij}$, $1 \leq i \leq k$, $1 \leq j \leq k_i$. The diffeomorphism can be chosen to keep $\La^a$ fixed. Using the methods of Milnor in \cite[Section 3]{Mi63}, we obtain that the component of the contractible loop has the homotopy type of the space of constant loops while every other component has the homotopy type of $S^1$.
\end{proof}

\ni
Consider the counting function $CF(L)$ for periodic geodesics, where
$$
CF(L) = \#\{\mbox{periodic geodesics of length smaller than or equal to $L$}\}.
$$
Proposition \ref{prop:negative curvature} tells us that in the negative curvature case, every periodic geodesic correspond to an element of $\cc(M)$. Setting $a=\demi L^2$, we have the following equality
$$\#\cc^a(M) = CF(L).$$

\ni
A lower bound for $E(M)$ can be deduced from a result of Margulis.

\begin{thm}{\bf(Margulis 1969 \cite{Ma69})}
\label{thm:margulis}
On a compact Riemannian manifold of negative curvature it holds that
$$h_{top}(g) = \lim_{L\ra\infty} \frac{\log CF(L)}{L},$$
where $h_{top}(g)$ is the topological entropy of the geodesic flow.
\end{thm}

\ni
For a definition see \cite{Gr07, Be03}. Theorem~\ref{thm:margulis} implies that for $L$ large enough, 
$$
\#\cc^a(M) = CF (L) > \frac{e^{h_{top}(g)L}}{2L}
$$

\ni
For example if $M = \Si_\ga$ is an orientable surface of genus $\ga$ and constant curvature $-1$, then $h_{top}(g) = 1$, see \cite[Section 10.2.4.1]{Be03}, and thus $$\#\cc^{\demi L^2}(\Si_\ga) > \frac{e^L}{2L}.$$

\subsubsection{Products}

\begin{lem}
Let $M, N$ be two manifolds. Then $$\La (M\times N) \cong \La M \times \La N.$$
\end{lem}

\begin{proof}
Consider the map $\phi : \La (M\times N) \ra \La M \times \La N$, sending the loop $\al: S^1 \ra M\times N : t \mapsto (\al_1 (t), \al_2(t))$ onto $(\al_1 (t), \al_2 (t))$.
\end{proof}

\ni
We will show in section \ref{section:p1 finite} that the product of two spheres $S^l\times S^n$ has $c(M,g)>0$.



\subsubsection{Lie groups}
\label{section:lie groups}

Let $G$ be a compact connected Lie group, i.e. a compact, connected smooth manifold with a group structure in which the multiplication and inversion maps $G\times G \ra G$ and $G\ra G$ are smooth. \\

\ni
The fundamental group $\pi_1(G)$ of a connected Lie group $G$ is abelian. In fact, considering the universal cover $\widetilde{G}$, the kernel of the projection $p:\widetilde{G} \ra G$ is then isomorphic to $\pi_1(G)$. This is a discrete normal subgroup of $\widetilde{G}$. Let $\ga \in \pi_1(G)$. Then $\tilde{g}\ra \tilde{g}\ga \tilde{g}^{-1}$ is a continuous map $\widetilde{G} \ra \pi_1(G)$. Since $\widetilde{G}$ is connected and $\pi_1(G)$ discrete, it is constant, so $\tilde{g}\ga \tilde{g}^{-1}=\ga$ for all $\tilde{g}$. Hence $\pi_1(G)$ is central in $\widetilde{G}$ and in particular, it is abelian. This yields
$$
\La G = \coprod_{[\al] \in \pi_1(G)} \La_\al G.
$$

\ni
Denote by $\La_1$ the component of the constant loop. For $a>0$ we consider the sublevel sets
$$
\La_1^a := \{q\in \La_1 \mid \E(q) \leq a \}.
$$
Choose a constant loop $\ga_1$ representing $\La_1$ and consider another component $\La_i$ of $\La G$ represented by a loop $\ga_i$. Fix $\ep \in (0,1)$. For $\ga,\ga' \in \La G$ we define
$$
(\ga *_\ep \ga') (t) = 
\begin{cases}
h\ga(\frac{t}{\ep}), & 0\leq t \leq \ep, \\
\ga' (\frac{t-\ep}{1-\ep}), & \ep\leq t \leq \ep,
\end{cases} 
$$
where $h\in G$ is such that $h\ga(0) = \ga'(0)$. Notice that $\tilde{\ga}(t) := h\ga(t)$ is homotopic to $\ga$. Then the map $\Ga_i: \La_1 \ra \La_i: \ga \mapsto \ga_i *_\ep \ga$, is a homotopy equivalence with homotopy inverse $\La_i \ra \La_1: \ga \mapsto \ga_i^{-1} *_\ep \ga$.  
We have
$$
\E(\ga *_\ep \ga') = \frac{1}{\ep} \E(h\ga) + \frac{1}{1-\ep} \E(\ga'), \quad \mbox{for all } \ga,\ga' \in \La G.
$$
Abbreviating $\La_i^a = \La_i \cap \La^a$ and $E_i = \max \{\E(h\ga) \mid h\in G\}$, we therefore have
$$
\Ga_i (\La_1^a) \subset \La_i^{\frac{a}{1-\ep} + \frac{E_i}{\ep}}.
$$ 
For $c \in \pi_1(G)$, recall that $e(c)$ denotes the infimum of the energy of a closed curve representing $c$ and $\cc^a(G) := \{c \in \pi_1(G) \mid e(c) \leq a\}$. Set $\ep = \demi$. Since $\Ga_c : \La_1 \ra \La_c$ is a homotopy equivalence, it follows that 
$$
\dim \io_k H_k (\La_1^a; \F_p) \leq \dim \io_k H_k (\La_c^{2a+2e(c)}; \F_p)
$$
for all $a>0$. We can estimate
$$
\begin{aligned}
\dim \io_k H_k (\La^{4a}; \F_p)  & = \sum_{c \in \pi_1(G)} \dim \io_k H_k (\La_c^{4a}; \F_p) \\
	 					& \geq \sum_{c \in \pi_1(G)} \dim \io_k H_k (\La_1^{2a-e(c)}; \F_p) \\
						& \geq \sum_{c \in \cc^a(G)} \dim \io_k H_k (\La_1^{a}; \F_p).
\end{aligned}
$$

\ni
We conclude that 
$$
\sum_{k\geq 0} \dim \io_k H_k (\La^{2 n^2}; \F_p) \geq \#\cc^{\demi n^2}(G) \cdot \sum_{k\geq 0} \dim \io_k H_k (\La_1^{\demi n^2}; \F_p).
$$
Considering the homomorphism
$$
\io_k: H_k (\La_1^a ;\F_p) \ra H_k(\La M;\F_p)
$$
induced by the inclusion $\La_1^a M \hookrightarrow \La M$, we can look at the growth rates 
$$
C_1(M,g) := \sup_{p\in \Pr} \liminf_{n\ra\infty} \frac{1}{n} \log \sum_{k\geq 0} \dim \io_k H_k \bigl(\La^{\demi n^2} ;\F_p \bigr) 
$$
and
$$
c_1(M,g) := \sup_{p\in \Pr} \liminf_{n\ra\infty} \frac{1}{\log n} \log \sum_{k\geq 0} \dim \io_k H_k \bigl(\La^{\demi n^2} ;\F_p \bigr).
$$
It follows from the definitions of $E(M)$ and $e(M)$ that for a Lie group,
$$
C(M,g) \geq E(M) + C_1(M,g)
$$
and 
$$
c(M,g) \geq e(M) + c_1(M,g).
$$






\section{Cotangent bundles}
\label{section:cotangent bundles}

Let $M$ be a smooth, closed manifold of dimension $m$. Let $T^*M$ be the corresponding cotangent bundle, and $\pi:T^*M\ra M$ the usual projection. We will denote local coordinates on $M$ by $q=(q_1,\ldots,q_m)$, and on $T^*M$ by $x=(q,p)=(q_1,\ldots,q_m,p_1,\ldots,p_m)$. 
\\ \\
We endowed $T^*M$ with the {\it standard symplectic form} $\om=d\la$ where $\la=p\,dq=\sum_{i=1}^m p_i \, dq_i$ is the {\it Liouville form}. The definition of $\la$ does not depend on the choice of local coordinates. It has also a global interpretation on $T^*M$ as
$$\la(x)(\xi)=p\,(d\pi\, \xi)$$ 
for $x\in T^*M$ and $\xi \in T_x T^*M$.
\\ \\
A {\it symplectomorphism} $\phi : (T^*M,\om) \ra (T^*M,\om)$ is a diffeomorphism such that the pullback of the symplectic form $\om$ is $\om$, i.e. $\phi^* \om = \om$. 
\\ \\
An {\it Hamiltonian} function $H$ is a smooth function $H: T^*M \ra\re$. Any Hamiltonian function $H$ determines a vector field, the {\it Hamiltonian vector field} $X_H$ defined by 
\begin{equation}
\label{e:hamiltonian vector field}
\om (X_H(x) , \cdot ) = -dH(x).
\end{equation}
Let $H:S^1\times T^*M \ra \re$ be a $C^\infty$-smooth time dependent $1$-periodic family of Hamiltonian functions. Consider the Hamiltonian equation 
\begin{equation}
\dot{x}(t) = X_{H} (x(t)), 
\label{hameq}
\end{equation}
In local coordinates it takes the physical form
\begin{equation}
\begin{cases}
\dot{q}=\phantom{-} \partial_p H(t,q,p) \\
\dot{p}= -\partial_q H(t,q,p). 
\end{cases}
\label{ham syst}
\end{equation}

\ni
The solutions of \eqref{hameq} generate a family of symplectomorphisms $\vp^t_H$ via
$$\frac{d}{dt} \vp^t_H = X_H \circ \vp^t_H, \quad \vp^0_H = \id.$$
The 1-periodic solutions of \eqref{hameq} are in one-to-one correspondence with the fixed points of the time-1-map $\vp_H = \vp^1_H$. We denote the set of such solutions by
$$\P (H) = \{ x:S^1 \ra \re \mid \dot{x}(t) = X_H (x(t)) \}.$$

\ni
A $1$-periodic solution of \eqref{hameq} is called {\it non-degenerate} if $1$ is not an eigenvalue of the differential of the Hamiltonian $d\vp_H (x(0)): T_{x(0)}\Tm \ra T_{x(0)}\Tm$, i.e.
$$
\det (I - d\vp_H (x(0)) \neq 1.
$$
If $H$ is time-independent and $x$ a non-constant $1$-periodic orbit, then it is necessarily degenerate because $x(a+t)$, $a \in \re$, are also $1$-periodic solutions. We will call such an $x$ {\it transversally nondegenerate} if the eigenspace to the eigenvalue $1$ of the map $d\vp^1_H (x(0)): T_{x(0)}\Tm \ra T_{x(0)}\Tm$ is two-dimensional. \\

\ni
An {\it almost complex structure} is a complex structure $J$ on the tangent bundle $TT^*M$ i.e. an automorphism $J: TT^*M \ra TT^*M$ such that $J^2 = -\id$. It is said to be {\it $\om$-compatible} if the bilinear form
$$\langle \cdot , \cdot \rangle \equiv g_J (\cdot,\cdot) := \om( \cdot , J \cdot)$$
defines a Riemannian metric on $T^*M$. Given such an $\om$-compatible almost complex structure, the Hamiltonian system \eqref{ham syst} becomes
$$X_H (x) = J(x)\nabla H(x).$$ 

\ni
The {\it action functional} of classical mechanics $\A_H: \La(\Tm) \ra \re$ associated $H$ is defined as 
$$
\A_H(x(t)) := \int_0^1(\la(\dot{x}(t)) - H(t,x(t)) \,dt.
$$
This functional is $C^\infty$-smooth and its critical points are precisely the elements of the space $\P(H)$.

\subsection{Fiberwise starshaped hypersurfaces}
Let $\Si$ be a smooth, connected hypersurface in $T^*M$. We say that $\Si$ is {\it fiberwise starshaped} if for each point $q\in M$ the set $\Si_q := \Si \cap T^*_q M$ is the smooth boundary of a domain in $T^*M$ which is strictly starshaped with respect to the origin $0_q \in T^*_q M$. \\

\ni
A hypersurface $\Si\subset \Tm$ is said to be of {\it restricted contact type} if there exists a vector field $Y$ on $\Tm$ such that
$$
\L_Y\om = d\io_Y\om = \om
$$
and such that $Y$ is everywhere transverse to $\Si$ pointing outwards. Equivalently, there exists a 1-form $\al$ on $\Tm$ such that $d\al=\om$ and such that $\al\wedge (d\al)^{m-1}$ is a volume form on $\Si$. Our assumption that $\Si$ is a fiberwise starshaped hypersurface thus translates to the assumption that $\Si$ is of restricted contact type with respect to the Liouville vector field 
$$
Y=\sum_{i=1}^m p_i \,\partial p_i
$$
or, equivalently: the Liouville form $\la$ defines a contact form on $\Si$. \\ 

\begin{figure}[h]
  \centering
  \psfrag{a}[][][1]{$T^*_qM$}
  \psfrag{b}[][][1]{$\Sigma_q$}
  \psfrag{c}[][][1]{$0_q$}
  \includegraphics[width=10cm]{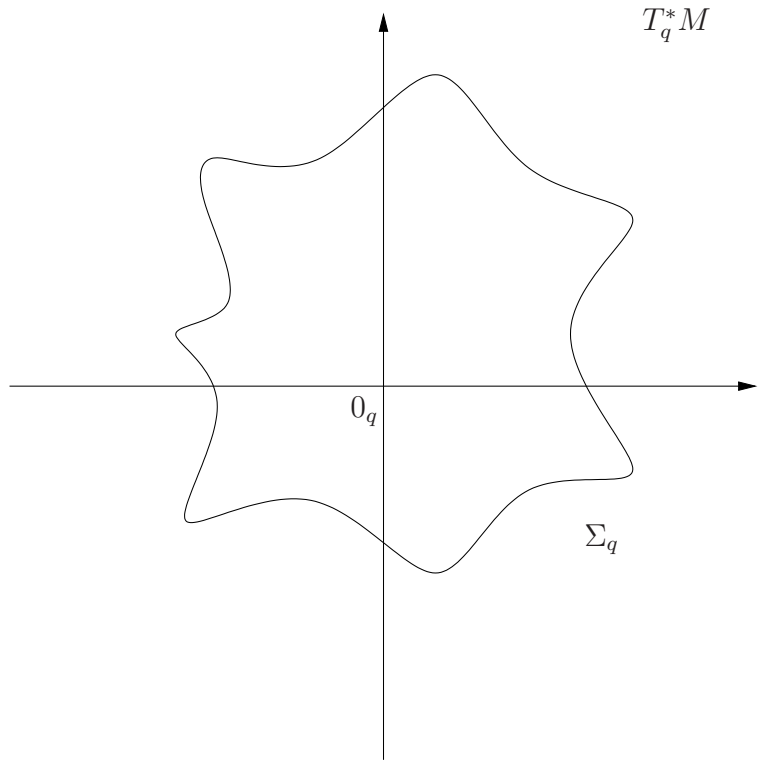}
  \caption{Fiberwise starshaped hypersurface.}
\end{figure}

\ni
There is a flow naturally associated with $\Si$, generated by the unique vector field $R$ along $\Si$ defined by
$$
d\la (R, \cdot) = 0, \quad \la(R) = 1.
$$
The vector field $R$ is called the {\it Reeb vector field} on $\Si$, and its flow $\vp^t_R$ is called the {\it Reeb flow}. \\

\ni
For $\tau>0$ let $\Od_R (\tau)$ be the set of closed orbits of $\vp^t_R$ with period $\leq \tau$. We measure the growth of the number of elements in $\Od_R (\tau)$ by

\begin{eqnarray*}
N_R &:=& \liminf_{\tau \to \infty} \frac 1 \tau \log \left( \# \Od_R (\tau) \right) ,\\
n_R &:=& \liminf_{\tau \to \infty} \frac 1{\log \tau} \log \left( \# \Od_R (\tau) \right) . 
\end{eqnarray*}

\ni
The number $N_R$ is the {\it exponential growth rate} of closed orbits, while $n_R$ is the {\it polynomial growth rate}.

\subsection{Dynamics on fiberwise starshaped hypersurfaces}

Given a fiberwise starshaped hypersurface $\Si \subset\Tm$, we can define an Hamiltonian function $F:\Tm\ra\re$ by the two conditions
\begin{equation}
\label{e:homogeneous hamiltonian}
F|_\Si \equiv 1, \quad F(q,sp)=s^2F(q,p), \quad s\geq 0 \mbox{ and }(q,p)\in \Tm.
\end{equation}
This function is of class $C^1$, fiberwise homogeneous of degree 2 and smooth off the zero-section.

\begin{lem}
Let $\Si\subset \Tm$ be a fiberwise starshaped hypersurface. If $\Si$ is the level set of a Hamiltonian function $H: \Tm \ra \re$, then the Reeb flow of $\la$ is a reparametrization of the Hamiltonian flow.
\end{lem}

\begin{proof}
The restriction of the 2-form $\om=d\la$ to $T\Si$ is degenerate and of rank $(2m-2)$. Its kernel is therefore 1-dimensional. By definition of both Reeb and the Hamiltonian vector fields, they also define this kernel since
$$
\io_R d\la|_{T\Si} = 0
$$
and
$$
\io_{X_H}\om|_{T\Si} = -dH|_{T\Si} = 0.
$$
Therefore 
$$
X_H (x)= a(x)R(x)
$$
for every $x\in\Si$, with a nowhere vanishing smooth function $a$.
\end{proof}

\ni
The first condition in \eqref{e:homogeneous hamiltonian} thus implies that the Hamiltonian vector field $X_F$ restricted to $\Si$ is a positive reparametrization of the Reeb vector field. Consequently 
$$
\vp_F^t(x) = \vp_R^{\si(t,x)} (x)
$$
for every $x\in\Si$ and for a smooth positive function $\si$ on $\re\times\Si$. In particular, we have that 
$$
\vp_F^1(x) = \vp_R^{s(x)} (x).
$$
where $s$ is smooth and positive. Since $\Si$ is compact, $s$ is bounded from above, say $s(x)\leq b$. Thus for a fixed metric, the image of any segment of length one is an orbit segment of length at most $b$. Hence the image of a segment of length at most $t$ has length less than $(t+1)b$ which for $t\geq 1$ is less than  $2bt$. This yields that the image of any periodic orbit of $X_F$ of period $\leq t$ has period $\leq 2bt$. Thus, if we denote by $\Od_F (t)$ the set of closed orbits of $\vp_F$ with period $\leq t$, we have that
$$
\Od_R(t) \geq \Od_F(t)/2b.
$$
The growth of the function $t\mapsto \#\Od_F(t)$ is thus equal to the growth of $t \mapsto \#\Od_R(t)$\\

\ni
Now we want to establish a correspondence between $1$-periodic solutions of $X_H$ and closed orbit of $X_H$ on $\Si$.
Consider the map $c_s : \Tm \ra \Tm$,
$$
c_s (x):= sx := (q,sp)
$$
and define the vector field $Y$ on $\Tm$ by
$$
Y(x) := \frac{d}{ds}\Big\vert_{s=1} c_s(x).
$$
In local coordinates $Y(q,p) = \sum_i p_i\partial p_i$. Differentiating $c_s^* \la = \sum_i sp_i dq_i = s\la$ with respect to $s$, we obtain
$$
\la = \L_Y \la = i_Y \om.
$$
Differentiating $F(sx) = s^2 F(x)$ with respect to $s$ at $s=1$ we get Euler's identity
$$
2F(x) = dF(x) (Y(x)) = - \om (X_F(x), Y(x)) = \la (X_H(x)).
$$
In view of $c_s^* F = s^2 F$ and $c_s^*\, \om = s\, \om$, we get
$$
\begin{aligned}
c_s^*\, i_{((c_s)_* X_F)} \om &= i_{X_F} (c_s^*\, \om ) \\
	& = s \, i_{X_F} \om \\
	& = - s \, dF \\
	& = -\frac{1}{s} \, d(c_s^*F) \\
	& = \frac{1}{s} \, c_s^*\, i_{X_F}\om.
\end{aligned}
$$
Thus $(c_s)_* X_F = \frac{1}{s} X_F$ or equivalently $dc_s (x) (X_F)(x) = \frac{1}{s} X_F (sx)$. \\

\ni
Fix a $1$-periodic solution $x$ of $X_F$ and define $x_s$ by $x_s(t) := sx(t)$. Then 
$$
\dot{x}_s = dc_s (x) (X_F) (x) = \frac{1}{s} X_F (x_s).
$$
Suppose $\A_F (x) = a > 0$. Euler's identity yields
$$
\A_F(x) = \int_0^1 \la(X_F(x))-F(x) = F(x).
$$
So if we set $s := 1/\sqrt{a}$, then $x_s$ satisfies $F(x_s) = s^2 F(x) = 1$. Thus $x_s$ is a closed orbit of period $\sqrt{a}$.\\

\ni
Conversely if $y:S^1 \ra \Si$ is a closed characteristic, then $\dot{y} = \frac{1}{s} X_F(y)$ for some $s>0$. Thus $x(t) := \frac{1}{s} y(t)$ is a $1$-periodic solution of $X_F$ and $x_s = y$. \\

\ni
We obtain that the map $x\mapsto x_{1/\sqrt{a}}$ is a bijection between $1$-periodic solutions of $X_F$ with action $a$ and the closed orbits with period $\sqrt{a}$.

\subsection{Spherization of a cotangent bundle}

Let $M$ be a closed connected submanifold and $\Si\subset \Tm$ a fiberwise starshaped hypersurface in $\Tm$. The hyperplane field
$$
\xi_\Si := \ker (\la|_\Si) \; \subset T\Si
$$
is a contact structure on $\Si$. Consider another fiberwise starshaped hypersurface $\Si'$. The radial projection in each fiber induces a map $\psi_q: \Si_q \ra \re$ such that  for every $p \in \Si_q$, $\psi_q(p) p \in \Si'_q$.
Then the differential of the diffeomorphism
$$
\Psi_{\Si\Si'} : \Si \ra \Si' : \quad (q,p)\mapsto (q,\psi_q(p)p)
$$
satisfies
$$
\Psi_{\Si\Si'}^*(\la|_\Si) = \psi \la|_\Si
$$
where $\psi(q,p) = \psi_q(p)$. Thus $\Psi_{\Si\Si'}$ maps $\xi_\Si$ to $\xi_\Si'$ and hence is a contactomorphism 
$$(\Si, \xi_\Si)\ra(\Si', \xi_\Si').$$

\ni
The induced identification of those contact manifold is called the {\it spherization} $(SM,\xi)$ of the cotangent bundle. The unit cosphere bundle $(S_1M(g), \ker \la)$,
$$
S_1M(g) := \{(q,p) \in \Tm \mid |p| = 1\}
$$
associated to a Riemannian metric $g$ on $M$ is a particular representative.\\

\ni
Fix a representative $(\Si, \xi_\Si)$. For every smooth positive function $f:\Si \ra \re$, $\xi_\Si = \ker(f\la|_\Si)$ holds true. We can thus consider the associated Reeb vector field $R_f$ on $T\Si$ defined as the unique vector field such that
$$
d(f\la) (R_f, \cdot) = 0, \quad f\la(R_f) = 1.
$$
We have seen in the previous section that for $f\equiv 1$ any Hamiltonian function $H:\Tm \ra \re$ with $H^{-1}(1)=\Si$, where $1$ is a regular value, the Reeb flow of $\R_f$ is a time change of of the Hamiltonian flow $\vp^t_H$ restricted to $\Si$. Note that for a different function $f$ the Reeb flows on $\Si$ can be completely different.\\ 

\ni
Given another representative $\Si'$, the associated contactomorphism $\Psi_{\Si\Si'}$conjugates the Reeb flows on $(\Si, \la)$ and $(\Si',\psi^{-1}\la)$, in fact
$$
d\Psi_{\Si\Si'}(R_\la) = R_{\psi^{-1}\la}. 
$$
This yields that the set of Reeb flows on $(SM,\xi)$ is in bijection with Hamiltonian flows on fiberwise starshaped hypersurfaces, up to time change. \\

\ni
Theorem A and Theorem B thus give lower bounds for the growth rate of closed orbits {\it for any Reeb flow on the spherization $SM$ of $\Tm$}.

\section{Maslov index}
\label{maslov index}

Consider $\re^{2m}$ endowed with its standard symplectic structure
$$\om_0 = dp\wedge dq, \quad (q,p)\in\re^m \times\re^m ,$$
and it's standard complex structure
$$
J_0 = \left( \begin{array}{cc} 0 & I \\
						  -I & 0 \end{array} \right).
$$ 
Denote by $Sp(2m)$ the set of symplectic automorphisms of $(\re^{2m}, \om_0)$, i.e.
$$
Sp(2m) := \{ \Psi \in \Mat(2m\times 2m, \re) \mid \Psi^t J_0 \Psi = J_0 \},
$$
by $\L(m)$ the space of Lagrangian subspaces of $(\re^{2m}, \om_0)$, i.e. 
$$
\L(m) := \{ L \subset \re^{2m} \mid \om_0(v,w)=0 \mbox{ for all } v,w \in L \mbox{ and }\dim L = m \},
$$
and by $\la_0$ the vertical Lagrangian subspace $\la_0 = \{0\} \times \re^m$.

\subsection{Maslov index for symplectic path}

 In \cite{CZ84} Conley and Zehnder introduced a Maslov type index that associates an integer $\mu_{CZ} (\Psi)$ to every path of symplectic automorphisms belonging to the space
$$
\Sp P :=\{\Psi  :[0,\tau] \ra Sp(2m) \mid \Psi(0)= I, \det(I-\Psi(1)) \neq 0 \}.
$$
The following description is presented in \cite{SZ92}.\\ 

\ni 
It is well known that the quotient $Sp(2m)/U(m)$ is contractible and so the fundamental group of $Sp(2m)$ is isomorphic to $\Z$. This isomorphism can be represented by a natural continuous map 
$$
\rho: Sp(2m) \ra S^1
$$
which restricts to the determinant map on $Sp(2m) \cap O(2m) \simeq U(m)$. Consider the set
$$
Sp(2m)^* := \{\Psi \in Sp(2m) \mid \det(I -\Psi) \neq 0\}.
$$
It holds that $Sp(2m)^*$ has two connected components
$$
Sp(2m)^\pm := \{\Psi \in Sp(2m) \mid \pm \det(I -\Psi) > 0\}.
$$
Moreover, every loop in $Sp(2m)^*$ is contractible in $Sp(2m)$. \\

\ni
For any path $\Psi: [0,\tau] \ra Sp(2m)$ choose a function $\al: [0,\tau] \raÊ\re$ such that $\rho(\Psi(t)) = e^{i\al(t)}$ and define
$$
\Delta_\tau (\Psi) = \frac{\al(\tau) - \al(0)}{\pi}.
$$
For $A \in Sp(2m)^*$ choose a path $\Psi_A (t) \in Sp(2m)^*$ such that $\Psi_A (0) = A$ and $\Psi_A (1) \in \{-I , \diag(2,-1,\ldots,-1,\demi,-1,\ldots,-1)\}$. Then $\Delta_1 (\Psi_A)$ is independent of the choice of this path. Define
$$
r(A) = \Delta_1 (\Psi_A), \quad A\in Sp(2m)^*.
$$
\\

\ni
The {\it Conley-Zenhder index} of a path $\Psi \in \Sp P$ is define as the integer
$$
\mu_{CZ}(\Psi) = \Delta_\tau(\Psi) + r(\Psi(\tau)).
$$
The following index iteration formula, proved in \cite{SZ92}, will be used in section \ref{section:the simply connected case}.
\begin{lem}
\label{lemma:iteration formula}
Let $\Psi(t) \in Sp(2m)$ be any path such that
$$\Psi(k\tau+t) = \Psi(t) \Psi(\tau)^k$$
for $t\geq 0$ and $k\in \N$. Then
$$\Delta_{k\tau}(\Psi) = k \Delta_1(\Psi)$$
for every $k\in \N$. Moreover, $|r(A)| <n$ for every $A\in Sp(2m)^*$.
\end{lem}

\subsection{The Maslov index for Lagrangian paths}
A generalization of the Conley--Zehnder index is due to Robbin and Salamon in \cite{RS93}. They associated a Maslov type index to any path regardless of where its endpoints lie.\\ 

\ni
Let $\La(t) :[0,\tau] \ra \L(m)$ be a smooth path of Lagrangian subspaces. For each $t_0$ we define a quadratic form $Q_{t_0}$ on $\La(t_0)$ as follows. Take a Lagrangian complement $W$ of $\La(t_0)$. For $v\in \La(t_0)$ and $t$ near $t_0$ define $w(t)\in W$ by $v+w(t) \in \La(t)$. Then $Q_{t_0} (v) := \frac{d}{dt}|_{t=t_0} \om_0(v,w(t))$ is independent of the choice of $W$.
Now fix a Lagrangian subspace $V$ and define
$$
\Si_k(V) := \{L \in\L(m) \mid \dim(L\cap V) = k\} \quad \mbox{and} \quad \Si(V) := \bigcup_{k=1}^m \Si_k (V).
$$
Without loss of generality we can assume that $\La$ has only regular crossings with $\Si(V)$ (this can be achieved by homotopy). Then the crossing form $\Ga_t := Q_t |_{\La(t)\cap V}$ is a nonsingular quadratic form whenever $\La(t) \in \Si(V)$. \\

\ni
The {\it Robbin--Salamon index} of a path $\La$ is defined as the half integer
$$
\mu_{RS} (\La, V) := \demi \sign \Ga_0 + \sum_{\substack{0<t<\tau \\ \La(t) \in \Si(V)}} \sign \Ga_t + \demi \sign \Ga_t,
$$ 
where the signature $\sign$ is the number of positive minus the number of negative eigenvalues of the quadratic form. \\

\ni
If  $\Psi:[0,\tau] \ra Sp(2n)$ is a path of symplectic automorphisms, then the graphs of $\Psi(t)$, $gr(\Psi(t))$, form a path of Lagrangian subspaces in $(\re^{2m} \oplus \re^{2m}, (-\om_0)\oplus\om_0)$. The {\it Robbin--Salamon index} of a path $\Psi$ is defined as the half integer
$$
\mu_{RS} (\Psi) = \mu_{RS} (gr(\Psi), \Delta),
$$
where $\Delta$ is the diagonal of $\re^{2m} \oplus \re^{2m}$. If $\Psi \in \Sp P$, this index is equal to the Conley--Zehnder index $\mu_{CZ} (\Psi)$.

\subsection{Maslov index for periodic orbits}
Let $H:[0,1]\times \Tm \ra\re$ be a time dependent Hamiltonian.
In order to define the Maslov index of a $1$-periodic solution of \eqref{hameq} we can proceed as follows. \\

\ni
Let $x \in \P(H)$. Then the symplectic vector bundle $x^*(T\Tm)$ admits a symplectic trivialization $\Phi: S^1\times \re^{2m} \ra x^*(T\Tm)$ such that
\begin{equation}
\label{e:trivialization}
\Phi (t) \la_0 = T_{x(t)}^v \Tm \quad \mbox{for all } t\in S^1,
\end{equation}
see \cite[Lemma 1.1]{AS06}. By acting on the differential of the Hamiltonian flow along $x$, the trivialization $\Phi$ provides a path in $Sp(2m)$,
$$
\Psi_x(t) = \Phi^{-1}(t) d\vp^t_H (x(0)) \Phi(0).
$$ 
If $x \in \P(H)$, neither the Conley--Zehnder index $\mu_{CZ} (\Psi_x)$, if well defined, nor the Robbin--Salamon index $\mu_{RS} (\Psi_x)$ depends on the symplectic trivialization $\Phi$ satisfying \eqref{e:trivialization}, for two such trivializations are then homotopic.\\

\ni
We define the {\it Maslov index} of the $1$-periodic orbit $x$ by setting
$$
\mu (x) := \mu_{RS} (\Psi_x).
$$
If $x$ is non-degenerate, $\Psi_x$ belongs to the set $\Sp P$ and thus $\mu(x)$ is an integer equal to $\mu_{CZ} (\Psi_x)$. If $x\in \P(H)$ is transversally nondegenerate, $gr(\Psi_x (1)) \in \Si_1 (\Delta)$ and thus $\mu(x)$ is not an integer.

 \clearemptydoublepage

\chapter{Convex to Starshaped} 
\label{chapter:convex to starshaped}
In this chapter we follow the idea of sandwiching develloped in Frauenfelder--Schlenk, \cite{FS05}, and Macarini--Schlenk, \cite{MS10}. By sandwiching the set $\Si$ between the level sets of a geodesic Hamiltonian, and by using the Hamiltonian Floer homology and its isomorphism to the homology of the free loop space of $M$, we shall show that the number of 1-periodic orbits of $K$ of action $\leq a$ is bounded below by the rank of the homomorphism
$$
\io_k: H_k (\La^{a^2} ;\F_p) \ra H_k(\La M;\F_p)
$$
induced by the inclusion $\La^{a^2} M \hookrightarrow \La M$.

\section{Relevant Hamiltonians}
\label{section:relevant hamiltonians}
Let $\Si$ be a fiberwise starshaped hypersurface in $T^*M$. We can define a Hamiltonian function $F:\Tm\ra\re$ by the two conditions
$$F|_\Si \equiv 1, \quad F(q,sp)=s^2F(q,p), \quad s\geq 0 \mbox{ and }(q,p)\in \Tm.$$
This function is of class $C^1$, fiberwise homogeneous of degree 2 and smooth off the zero-section. To smoothen $F$ choose a smooth function $f:\re\ra\re$ such that 
$$\begin{cases}
f(r) = 0 & \mbox{if } r\leq \ep^2,\\
f(r) = r  & \mbox{if } r\geq \ep  \\
f'(r) >0 & \mbox{if } r > \ep^2, \\
0\leq f'(r)\leq2 & \mbox{for all } r,
\end{cases}$$
where $\ep\in(0,\frac{1}{4})$ will be fixed in section \ref{section:non-crossing lemma}. Then $f\circ F$ is smooth. \\

\begin{figure}[h]
  \centering
  \includegraphics[width=5cm]{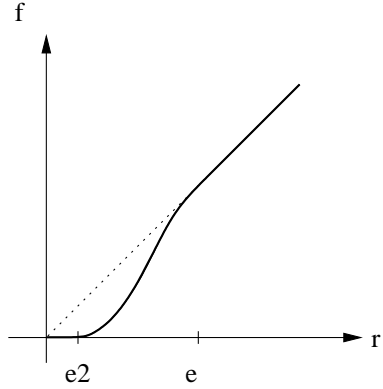}
  \caption{The cut off function $f$.}
\end{figure}

\ni
Recall that $(q, p)$ denotes canonical coordinates on $T^*M$. Fix a Riemannian metric $g$ on $M$, let $g^*$ be the Riemannian metric on $T^*M$ induced by $g$, and define 
$$
G(q,p) := \demi g^*(q)(p,p).
$$
In the following we will often write $G(q,p)=\demi|p|^2$. Our sign convention in the definitions of the symplectic form $\om$ and the Hamiltonian vector field \eqref{e:hamiltonian vector field} is such that the flow $\vp^t_G$ is the geodesic flow.\\

\ni
For $r>0$ we abbreviate
$$D(r) = \{(q,p) \in T^*M \mid |p| \leq r\}.$$

\ni
Recall that $M$ is compact and that $\Si$ is fiberwise starshaped. After multiplying $g$ by a constant, we can assume that $G \leq F$. Choose $\si >0$ such that $\si G \geq F$. Moreover we ask that $G < F < \si G $ on $D^c(\ep^2)$.\\

\ni
The Hamiltonian Floer homology for $F$ or $G$ is not defined since all the periodic orbits are degenerate. As we shall consider multiples $nF$ and $nG$ of $F$ and $G$, $n \in \N$, we associate to them the followings perturbations. \\

\ni
Fix $n\in \N$. Choose $c \in (0,\frac{1}{4})$. \\
We shall add to $nF$ a perturbation $V_n:S^1 \times T^*M \ra \re$ such that 
\begin{itemize}
\item[(V1)] All 1-periodic solutions of $\dot{x}(t)=X_{n(f\circ F)+V_n}(x(t))$ are nondegenerate, and
\item[(V2)] $\|V_n (t,q,p)\|_{C^1} < \min \{c,\displaystyle\frac{c}{\phantom{-}\| X_{n(f\circ F)} \|_{C^0}}, \frac{c}{\phantom{-}\| p\|_{C^0}} \}  $,
\end{itemize}
where $\|V(t,x)\|_{C^1} := \sup\{|V(t,x)| +|dV(t,x)| \mid (t,x) \in S^1\times \Tm\}$.\\

\ni
We shall add to $nG$ and $n\sigma G$ a perturbation $W_n:S^1 \times M \ra \re$ such that
\begin{itemize}
\item[(W1)] All 1-periodic solutions of $\dot{x}(t)=X_{n(f\circ G)+W_n}(x(t))$ and  $\dot{x}(t)=X_{n\si G+W_n}(x(t))$ are nondegenerate, 
\item[(W2)] $\| W_n (t,q) \|_{C^1} < \min\{c_n, \displaystyle\frac{c_n}{\phantom{-}\| X_{n\si G}\|_{C^0}} \}$ for some constant $0< c_n < \min\{c, \frac{d_n}{4} \} $, where $d_n < 1$ will be fixed in section \ref{section:continuation homomorphisms}, 
\item[(W3)] $nG+W_n \leq nF+V_n \leq n \si G+W_n$,
\item[(W4)] $nF - V_n \leq n \si G - W_n$ on $D^c(\Si)$.
\end{itemize}

\begin{rmk}
We could add to $nF$, $nG$ and $n\si G$ the same perturbation $W_n(t,q)$ such that all 1-periodic solutions are nondegenerate and $\| W_n (t,q) \|_{C^1} <c_n <c$. However in the proof of Theorem B, the perturbations $V_n$ that we will consider will not be potentials. Therefore we made this more general choice of perturbations.
\end{rmk}

\ni
We alter $n (f\circ F)+V_n$ near infinity to a perturbed Riemannian Hamiltonian.

\ni
Choose smooth functions $\tau_n:\re\ra\re$ such that 
$$\begin{array}{ll}
\tau_n(r)=0 &\mbox{if } r\leq 3, \\
\tau_n(r)=1 &\mbox{if } r\geq 6 \mbox{ and} \\
\tau'_n(r)\geq 0 & \mbox{for all } r\in\re.
\end{array}$$
Set 
$$\begin{aligned}
\gg(t,q,p) & := n\si G(q,p) +W_n (t,q), \\
K_n(t,q,p) & := (1-\tau_n(|p|) \left(n (f\circ F) + V_n\right) (t,q,p)+\tau_n(|p|)\;  \gg(t,q,p), \\
\gl(t,q,p) & := (1-\tau_n(|p|) \left(n (f\circ G) + W_n \right)(t,q,p)+\tau_n(|p|)\;  \gg(t, q,p) \\
 	      & \phantom{i}= n \bigl( (1-\tau_n(|p|)  (f\circ G )(q,p)+\tau_n(|p|)\;  \si G(q,p) \bigr) + W_n(t,q).
\end{aligned}$$
Then 
$$\gl\leq K_n \leq \gg, \quad \mbox{for all } n \in \N.$$ 
where 
$$K_n=n( f\circ F)+V_n \; \mbox{ and } \; \gl = n( f\circ G) + W_n \mbox{ on } D(3).$$
Since $D(\Si) = \{F \leq 1\} \subset \{nF+V \leq 2n\} \subset \{nG+W_n \leq 2n\} \subset D(3)$, we thus have in particular that
$$ K_n = n (f\circ F) +V_n  \mbox{ on } D(\Si) .$$
Moreover,
$$ \gl = K_n = \gg  \mbox{ on } \Tm\backslash D(6).$$ 

\begin{figure}[h]
  \centering
  \psfrag{a}[][][1]{$G^{-1}(1)$}
  \psfrag{b}[][][1]{$K^{-1}(1)$}
  \psfrag{c}[][][1]{$\sigma G^{-1}(1)$}
  \includegraphics[width=10cm]{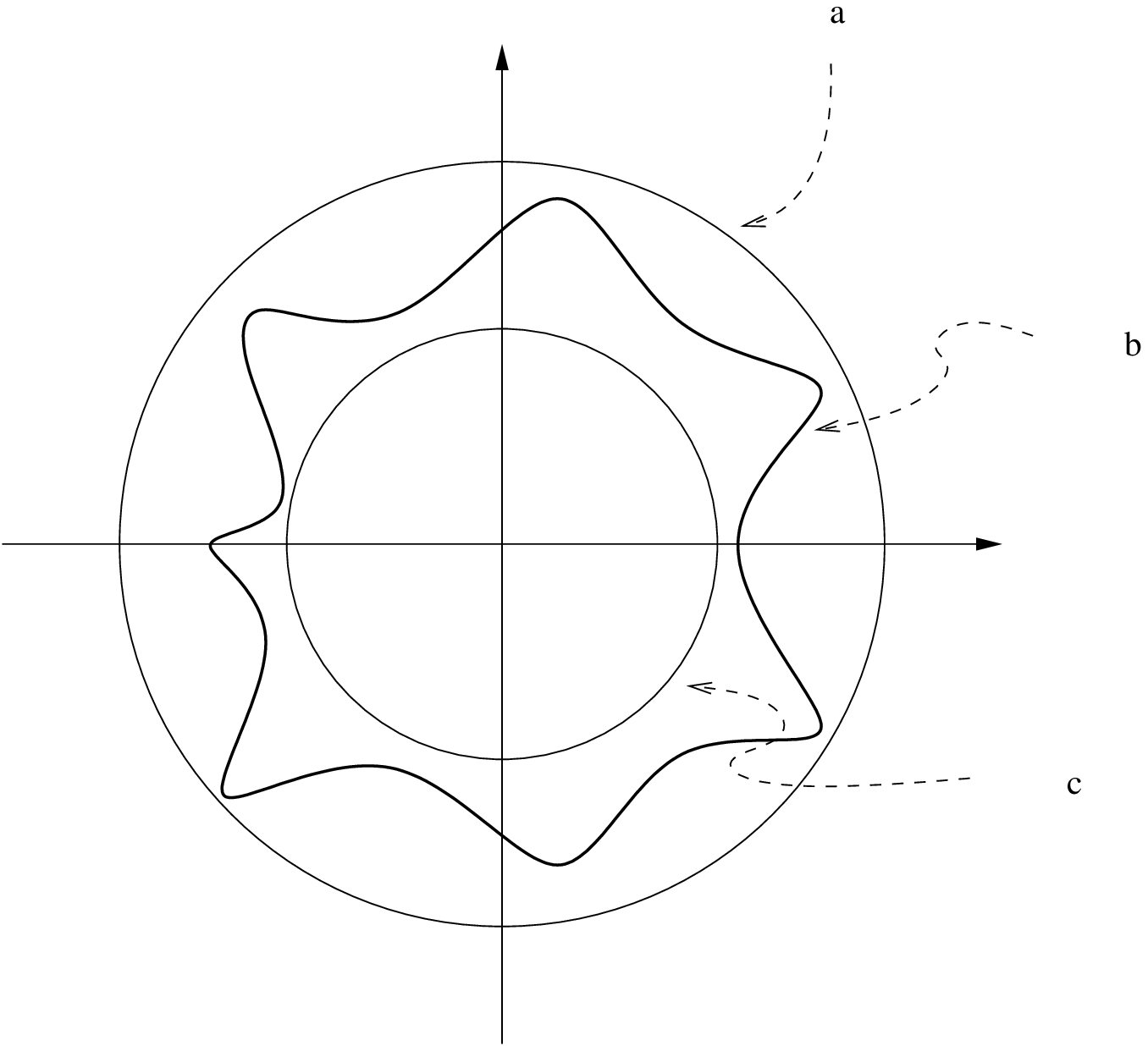}
  \caption{Sandwiching.}
\end{figure}

\section{Action spectra}
\label{section:action spectra}
The action spectrum $\Sp (H)$ of a proper Hamiltonian $H:S^1\times T^*M\ra\re$ is the set of critical values of the action functional $\A_H:\La (T^*M)\ra\re$, that is
$$\Sp(H):= \{\A_H(x)\mid x\in \P(H)\}.$$ 
The following proposition follows from \cite[Proposition 3.7]{Schw00}.
\begin{prop}
The action spectrum $\Sp (H)$ is the union of countably many nowhere dense subsets of $\re$.
\label{spectrum}
\end{prop}


\ni
We now look at the action spectra of perturbed homogeneous Hamiltonians of degree~2. 

\begin{lem}
\label{lemma:savior}
Let $H$ be fiberwise homogeneous of degree~2 and let $V(t,q,p)$ be a perturbation such that $\|V\|_{C_1} <\frac{c}{\phantom{-}\|X_H\|_{C^0}}$. Let $x\in \P(H+V)$ and assume that $H(x(t_0)) = a$ for some $t_0 \in S^1$. Then it holds that
$$
a-c < H(x(t)) < a+c
$$
for all $t\in S^1$.
\end{lem}

\begin{proof}
By definition of the Hamiltonian vector field we have
$$\begin{aligned}
\Bigl|\frac{d}{dt} H(x(t))\Bigr| & = \,\mid dH (x(t)) \; \dot{x}(t) \mid \\
			      & = \,\mid\om(X_H(x(t)), X_{H+V}(x(t))) \mid \\
			      & = \,\mid\om(X_H(x(t)), X_H (x(t)) +X_V (x(t))) \mid \\
			      & =\, \mid\om(X_H(x(t)), X_{V}(x(t))) \mid \\
			      &= \,\mid dV(x(t)) \; X_H(x(t))\mid \\
			      &\leq \| V \|_{C^1} \; \| X_H \|_{C^0} \\
			      & <c.
\end{aligned}$$
Assume there exist $t_1\in S^1$ such that $|H(x(t_1)) - H(x(t_0))| \geq c$. Then there exists $t \in S^1$ such that 
$$
\Bigl|\frac{d}{dt} H(x(t))\Bigr|  \geq c
$$
which is a contradiction.
\end{proof}

\ni
In the sequel, Lemma~\ref{lemma:savior} will always be applied under assumption (V2) or (W2).

\begin{lem}
Let $H: \Tm \ra\re$ be fiberwise homogeneous of degree 2 and let $W(t,q)$ be a perturbation such that $\|W\|_{C^1} <c$. Let $h:\re\ra\re$ be a smooth function and let $r>0$. Then it holds that 
\begin{itemize}
\item[(i)] for $x\in\P(h\circ H + W)$ we have $$\A_{h\circ H + W}(x)= \int_0^1 2h'(H(x)) H(x)-h(H(x)) - W(t,x)\, dt,$$
and that
\item[(ii)] $\Sp(rH + W)\subset \frac{1}{r}\Sp(H + W) + [-c, c]$.
\end{itemize}
\label{action}
\end{lem}

\begin{proof}
(i) Set $Y=\sum_i p_i\partial p_i$. For each $t\in[0,1]$ we have $\dot{x}(t)=X_{h\circ H + W}(x(t))$ and thus
$$\begin{aligned}
\la(\dot{x}(t))& =\om(Y,X_{h\circ H + W}(x(t)))\\
  &= \om (Y, X_{h\circ H}(x(t))) + \om (Y,X_W(x(t))) \\
  &=d(h\circ H)(x(t)) (Y) + dW(x(t)) (Y) \\
  &=h'(H(\ga(t)))\,dH(x(t))(Y).
\end{aligned}$$
Since $H$ is fiberwise homogeneous of degree 2, Euler's identity yields
$$dH( x(t))(Y)=2H(x(t))$$
whence 
$$\la(\dot{x}(t))=2h'(H(x(t)))H(x(t)).$$
Therefore 
$$\begin{aligned}
\A_{h\circ H +W }(x(t)) &=\int_0^1 (\la(\dot{x}(t))-h(H(x(t)))-W(x(t))dt \\
 & = \int_0^1 2h'(H(x(t)))H(x(t))-h(H(x(t))) - W(x(t)) \, dt
\end{aligned}$$
as claimed.\\

(ii)By definition of the Hamiltonian vector field, we have that $X_{cH}(q,p)=cX_H (q,p)$ and thus $X_H (q,\frac{1}{r}p)= \frac{1}{r}X_H(q,p)$ since $H$ is homogeneous of degree 2. Also, since $W$ does not depend on $p$, $X_W (t,q,\frac{1}{r}p) = X_W (t,q,p)$. Hence
$$\begin{aligned}
X_{rH + W}(t,q,\frac{1}{r}p) &=rX_H(q,\frac{1}{r}p)+ X_W (t,q,\frac{1}{r}p) \\
 &=X_H (q,p) + X_W (t,q,p) \\
 &=X_{H+W} (t,q,p).
\end{aligned}$$
To the orbit $x(t)=(q(t),p(t))$ in $\P(H + W)$ therefore corresponds to $x_r (t)=(q(t),\frac{1}{r}p(t))$ in $\P(rH + W)$. Claim (i) then yields 
$$\begin{aligned}
\A_{rH+W}(x_r) &= \int_0^1 rH(x_r) - W(t,x_r) \,dt \\
 &=\int_0^1  rH(q(t), \frac{1}{r}p(t)) - W(t,x_r) \,dt \\ 
 &=\int_0^1  \frac{1}{r}H(q(t),p(t))- W(t,x_r) \,dt \\ 
 &=\frac{1}{r}\A_{H+W} (x) + \frac{r-1}{r} \int_0^1 W(t,x_r) \,dt ,
\end{aligned}$$
and thus $\Sp(rH+W) \subset \frac{1}{r}\Sp(H+ W)+[-c,c]$.
\end{proof} 

\ni
We next have a look at the action spectrum of $K_n$.
\begin{prop}
Let $\ga\in \P(K_n)$. If $F(\ga(t_0)) < 1$ for some $t_0 \in S^1$, then $\A_{K_n}(\ga)< n+ 2c$. If $F(\ga(t_0)) > 1$ for some $t_0 \in S^1$, then $\A_{K_n}(\ga)>n- 2c $.  
\label{action2}
\end{prop}

\begin{proof}
Assume first that $F(\ga(t_0)) < 1$ for some $t_0 \in S^1$. Then 
$$
(K_n - V_n) (\ga(t_0)) = n(f\circ F)(\ga(t_0)) < n
$$
and by Lemma~\ref{lemma:savior} 
$$
(K_n - V_n) (\ga(t)) = n(f\circ F)(\ga(t)) < n + c .
$$
Thus $\ga\in\P( n (f\circ F )+V_n)$. Hence our choice of $f$ and (V2) yield that
$$
\begin{aligned}
\A_{K_n} (\ga) &= \A_{n(f\circ F) + V_n} (\ga) \\
 			&=  \int_0^1 2nf'(F(\ga)) F(\ga) + \la(X_{V_n} (x))- nf(F(\ga)) - V_n(\ga)\,dt \\
			&= \int_0^1 2nf'(F(\ga)) F(\ga) - nf(F(\ga)) \,dt + \int_0^1 dV_n(x) (Y) - V_n(\ga)\,dt \\
 			&< n+2c.
\end{aligned}
$$
Assume now that $F(\ga(t_0)) > 1$ for some $t_0 \in S^1$. Then 
$$
\bigl( K_n - ( (1-\tau_n(|p|)) V_n + \tau_n(|p|) W_n ) \bigr) (\ga(t_0)) > n
$$
and by Lemma~\ref{lemma:savior}
$$
\bigl( K_n - ( (1-\tau_n(|p|)) V_n + \tau_n(|p|) W_n ) \bigr) (\ga(t)) > n -c.
$$
Then $nF(\ga)>n-c$ and $n(f \circ F) (\ga) = nF(\ga)$. Let again $Y=\sum_i p_i\partial p_i$. Using the definition of $K_n$ we compute at $\ga(t) = (q,p)$, for $t\in[0,1]$ that
$$\begin{aligned}
d(K_n)(\ga(t)) (Y) = & -\tau'_n(|p|)(nF+V_n) (t,q,p) |p| + (1-\tau_n (|p|)) d(nF+V_n) (q,p) (Y) \\
 & + \tau'_n (|p|) (n\si G+W_n) (t,q,p) |p| + \tau_n (|p|) d(n\si G) (q,p) (Y). 
\end{aligned}$$
Since $\tau'_n (|p|) \geq 0$ and $n\si G +W_n \geq nF+V_n $, the sum of the first and third summands is $\geq 0$. Together with Euler's identity applied to the functions $nF$ and $n\si G$ we obtain that
$$\begin{aligned}
d(K_n) (\ga(t)) (Y) - K_n(\ga(t))
 & \geq \bigl(1-\tau_n(|p|)\bigr)  \bigl(d(nF+V_n) (\ga(t)) (Y) - (nF+V_n)(\ga(t)\bigr) \\
 &\phantom{\geq} \; + \tau_n(|p|) \bigl( d(n\si G) (\ga(t)) (Y)  - (n\si G+W_n) (\ga(t)) \bigr) \\
 & = \bigl(1- \tau_n(|p|) \bigr) (nF-V_n)(\ga(t)) + \tau_n(|p|) (n\si G-W_n)(\ga(t)) \\
 & \phantom{\geq} + \bigl(1- \tau_n(|p|) \bigr) dV_n (\ga(t)) (Y)\\ 
 & \geq (nF-V_n)(\ga(t)) + \bigl(1- \tau_n(|p|) \bigr) dV_n (\ga(t)) (Y).
\end{aligned}$$ 
where in the inequality we use the fact that $\tau_n(|p|) \geq 0$ and $n\si G -W_n \geq nF-V_n $. Recalling (V2) and $0 \leq 1- \tau_n(|p|) \leq 1$, this yields 
$$
\begin{aligned}
\A_{K_n} (\ga) &=\int_0^1 \bigl( d(K_n)(\ga(t)) (Y) - K_n(\ga(t)) \bigr) \;dt \\
			& \geq \int_0^1 nF(\ga(t)) \,dt -\int_0^1 V_n (\ga(t)) - \bigl(1- \tau_n(|p|) \bigr) dV_n (\ga(t)) (Y) \, dt\\
			& > n- 2c.
\end{aligned}
$$
as claimed.
\end{proof}

\subsection{The Non-crossing lemma}
\label{section:non-crossing lemma}

Consider the space of Hamiltonian functions 
$$
\H_6 (\gg) = \{ H:S^1\times \Tm\ra\re \mid H=\gg \mbox{ on } S^1\times \Tm \backslash D(6)\}.
$$ 
Note that $\gl$ and $K_n$ belong to $\H_6 (\gg)$. For $a\in\re$ set
$$
\H^a_6 (\gg) = \{H\in \H_6(\gg) \mid a \not\in \Sp (H)\}.
$$
Now fix a smooth function $\be:\re\ra[0,1]$ such that
\begin{equation}
\begin{array}{ll}
\be(s)=0 & \mbox{for } s\leq0,\\
\be(s)=1 & \mbox{for } s\geq 1\mbox{ and} \\ 
\be'(s) \geq 0 & \mbox{for all } s\in \re.
\end{array}
\label{beta}
\end{equation} 
For $s\in [0,1]$ define the functions
\begin{equation}
\gs  =(1-\be(s)) \gl + \be(s) \gg .
\label{gs}
\end{equation}
Then $\gs \in \H_6(\gg)$ for each $s\in [0,1]$.
\\ 
\ni 
Recall that $\si \geq 1$. Choose $a\in\, ]n-2c,(n+1)-2c[$ and define the function $a(s):[0,1]\ra\re$ by 
$$a(s)=\frac{a}{1+\be(s)(\si-1)}.$$
Note that $a(s)$ is monotone decreasing with minimum $a(1)=a/\si$. Recall $\ep\in(0,\frac{1}{4})$ entering the definition of the smoothing function $f$. We assume from now on that it also fulfills
\begin{equation}
\ep^2 < \frac{1-2c}{2 \si^2}.
\label{epsilon}
\end{equation}  

\begin{lem}
If $[a-c_n , a+c_n]\cap \Sp(\gl) = \emptyset$, then $a(s)\notin \Sp(\gs)$ for $s\in[0,1]$.
\label{crossing}
\end{lem}

\begin{proof}
Take $\ga = (q(t),p(t)) \in \P(\gs)$. \\
Assume first that $|p(t_0)| > 3$ for some $t_0 \in S^1$. Then 
$$
(\gs - W_n) (\ga(t_0)) \geq nG(\ga(t_0)) > n\frac{9}{2}.
$$
By Lemma~\ref{lemma:savior} we therefore have
$$
n\si G(\ga) \geq (\gs - W_n) (\ga) > n\frac{9}{2}-c_n.
$$
Then $\ga \subset D^c(\ep)$ and $f\circ G(\ga)= G(\ga)$. Hence 
\begin{equation}
\label{e:gs}
\gs = n\bigl( 1-\be(s)\bigr)  \bigl((1-\tau_n(|p|) G+\tau_n(|p|)\si G\bigr) + n\be(s)\si G + W_n.
\end{equation} 
Doing a similar computation as in the proof of Proposition \ref{action2} where we replace $F$ by $G$, we find
$$\begin{aligned}
d(\gs)(\ga(t)) (Y) - \gs(\ga(t)) 
&\geq  n\bigl(1-\be(s)\bigr) \Bigl((1-\tau_n(|p|)) (G (\ga(t)))  + \tau_n(|p|) ( \si G (\ga(t)) \Bigr) \\
 & \quad \; + n\be(s) \bigl( \si G (\ga(t))\bigr) - W_n(\ga(t))\\
 &\geq n\bigl(1-\be(s)\bigr) \bigl(G (\ga(t))\bigr) + n\be(s) \bigl( \si G (\ga(t))\bigr) - W_n(\ga(t)) \\
 &\geq n G(\ga(t)) - W_n(\ga(t)) \\
 &\geq 2n
\end{aligned}$$
where the last inequality follows from \eqref{e:gs}, Lemma~\ref{lemma:savior} and the observation
$$
(\gs - W_n) (\ga(t)) = n \bigl( C_{n,s} \;G(\ga(t)) \bigr) \geq n \bigl( C_{n,s} \; \frac{9}{2} \bigr) -c_n
$$
where $C_{n,s}$ is a positive constant.
Thus 
$$\A_{\gs} \geq 2n \geq n+1 > a \geq a(s).$$
Assume next that $|p(t_0)| \leq \ep$ for somme $t_0 \in S^1$. Then
$$
(\gs - W_n) (\ga(t_0)) \leq n \si G(\ga(t_0)) \leq n \si \frac{\ep^2}{2}.
$$
By Lemma~\ref{lemma:savior} we therefore have
$$
nG(\ga) \leq (\gs-W_n) (\ga) \leq n \si \frac{\ep^2}{2} + c_n.
$$
Then $\tau_n (|p(t)|) = 0$ for all $t\in [0,1]$, which yields $\gl (\ga) = n(f\circ G) (\ga) + W_n(\ga)$ and
$$\begin{aligned}
\gs(\ga) & = n\Bigl(\bigl( 1-\be(s) \bigr) \bigl( f\circ G \bigr) (\ga) + \be(s) \si G (\ga) \Bigr) +W_n(\ga)\\
              & = n\Bigl( \bigl (1-\be(s) \bigr)  f\circ  + \be(s) \si \Bigr) G(\ga) + W_n(\ga).
\end{aligned}$$ 
By Lemma \ref{action} (i) we therefore have 
$$
\A_{\gs} (\ga) = \int_0^1 2n  \Bigl( \bigl (1-\be(s) \bigr)  f' \bigl(G(\ga) \bigr) \Bigr) G(\ga) - \gs (\ga) \,dt.
$$
Together with $f'\leq 2$ and the choice of $\ep$ we obtain that
$$\begin{aligned}
\A_{\gs} (\ga)& \leq \int_0^1 4 nG(\ga) -  W_n(\ga)\,dt \\
		& \leq 2n \si \ep^2 +  5 c_n \\
		& < \frac{n-2c}{\si} \\ 
		& \leq a(1) \leq a(s). 
\end{aligned}$$
Assume finally that $\ga$ lies in $D(3) \backslash D(\ep)$. Then $\tau_n(|p(\ga)|) = 0$ and $(f\circ G) (\ga) = G(\ga)$. Hence $\gl (\ga) = nG (\ga) + W_n(\ga)$ and 
$$\gs (\ga) =  \bigl(1+ \be(s) (\si-1) \bigr) nG (\ga) + W_n(\ga).$$
If $\A_{\gs} (\ga) = a(s) $, then, in view of Lemma \ref{action} (ii) and the definition of $a(s)$, we obtain that $[a-c_n, a+c_n] \cap \Sp (nG+W_n) \neq \emptyset$ which is a contradiction to our hypothesis.
\end{proof}

\section{Floer Homology for convex Hamiltonians}

\noindent
Floer Homology was invented by Floer in a series of seminal papers, see \cite{Flo881, Flo882, Flo889}.

\subsection{Definition of $HF^a_*(H;\F_p)$}
\label{section:definition of hf}

\subsubsection{The chain groups}
Let $H \in \H_6(\gg)$ such that all 1-periodic solutions of 
$$
\dot{x}(t) = X_H(x(t))
$$
are nondegenerate. For $a< (n+1)-2c$ define
$$
\P^a(H) := \{x\in \P(H) \mid \A_H(x) \leq a\}
$$
For $\ga(t) =(q(t),p(t)) \in \P(H)$ such that $|p(t_0)| \geq 6$ for some $t_0 \in S^1$ we have by Lemma~\ref{lemma:savior}
$$
nG (\ga) > 3n - c_n.
$$
Then by Lemma~\ref{action} (i) we have
$$
\A_H(\ga) = \A_\gg (\ga) = \int_0^1 n\si G(\ga) - W_n(\ga) dt \geq 3n-2c_n  \geq 2n \geq n+1
$$
whence

\begin{equation}
\P^a(H) \subset D(6)
\label{e:subset}
\end{equation}

\ni
By Lemma \ref{spectrum}, the set $\P^a(H)$ is finite. For each $x\in \P(H)$ the Maslov index $\mu(x)$ is a well-defined integer, see Section \ref{maslov index}. Define the $k^{th}$ Floer chain group $CF_k^a(H ; \F_p)$ as the finite-dimensional $\F_p$-vector space freely generated by the elements of $\P^a(H)$ of Maslov index $k$, and define the full Floer chain group as 
$$
CF^a_*(H;\F_p) = \bigoplus_{k\in \Z} CF^a_k (H;\F_p).
$$

\subsubsection{Almost complex structures}
\label{almost complex structures}
Recall that an almost complex structure $J$ on $T^*M$ is $\om$-compatible if
$$
\lg \cdot, \cdot \rg \equiv g_J(\cdot,\cdot) := \om (\cdot, J, \cdot)
$$
defines a Riemannian metric on $T^*M$. Consider the Liouville vector field $Y = \sum_i p_i\partial p_i$ and its semi-flow $\psi_t$, for $t \geq 0$, on $\Tm \backslash \Dcirc(6)$. Denote by
$$
\xi = \ker (\io_Y\om |_{\partial D(6)})
$$
the contact structure on $\partial D(6)$. \\

\ni
An $\om$-compatible almost complex structure $J$ on $T^*M$ is {\it convex} on $T^*M \backslash \Dcirc(6)$ if
$$ \begin{aligned}
J\xi = \xi,& \\
\om(Y(x),J(x)Y(x))=1,& \quad \mbox{for } x\in \partial D(6) \\
d\psi_t(x)J(x) = J(\psi_t(x))d\psi_t(x), & \quad \mbox{for } x\in \partial D(6) \mbox{ and } t\geq 0.
\end{aligned}$$
Following \cite{CFH95, BPS03} we consider the set $\Js$ of $t$-dependent smooth families $\J = \{J_t\}$, $t\in S^1$, of $\om$-compatible almost complex structures on $T^*M$ such that $J_t$ is convex and independent of $t$ on $T^*M \backslash \Dcirc(6)$. The set $\Js$ is non-empty and connected. 

\subsubsection{Compactness}
\label{section:compactness}
For $\J \in \Js$, for smooth maps $u$ from the cylinder $\re\times S^1$ to $T^*M$, and for $x^\pm\in \P^a(H)$ consider Floer's equation given by

\begin{equation}
\begin{cases}
\partial_s u + J_t (u) (\partial_t u - X_{H}(t,u))=0 \\
\lim_{s\ra±\infty} u(s,t)=x^± (t) \text{ uniformly in t.}
\end{cases}
\label{floereq}
\end{equation}

\begin{lem}
Solutions of Floer's equation \eqref{floereq} are contained in $D(6)$.
\label{l:compactness}
\end{lem}

\begin{proof}
By $\P^a (H) \subset D(6)$ we have $x^\pm \subset D(6)$, whence
\begin{equation}
\lim_{s\ra \pm\infty} u(s,t) = x^\pm (t) \subset D(6).
\label{e:lim}
\end{equation}
In view of the strong maximum principle, the lemma follows from the convexity of $J$ outside $D(6)$ and from \eqref{e:lim} together with the fact that $H=\gg$ outside $D(6)$ implies $\om(Y,JX_H) = 0$, see Appendix \ref{convexity} and \cite{CFHW96, FS07}.
\end{proof}

\ni
We denote the set of solutions of \eqref{floereq} by $\M(x^-, x^+, H; \J)$. The elements of $\P^a(H)$ are the stationary solutions of Floer's equation \eqref{floereq} and if $u\in \Mod$, then $\A_H(x^-) \geq \A_H(x^+)$, see the more general Lemma \ref{l:decreasing action} below. So $\M(x,x,H; \J)$ contains only the elements $x$, and
$$
\Mod = \emptyset \quad \mbox{if } \A_H(x^-) < \A_H(x^+) \mbox{ and } x^- \neq x^+.
$$

\ni
The compactness of the manifolds $\M(x^-,x^+,H;\J)$ follows from Lemma~\ref{l:compactness} and the fact that there is no bubbling-off of $\J$-holomorphic spheres. Indeed, $[\om]$ vanishes on $\pi_2 (T^*M)$ because $\om=d\la$ is exact. See ~\cite{Flo882,Sal90} for details. \\

\subsubsection{The boundary operators}
\label{floer boundary operator}

There exists a residual subset $\Jr(H)$ of $\Js$ such that for each $\J \in \Jr(H)$ the linearized operator to Floer's equation is surjective for each solutions of \eqref{floereq}. For such a regular $\J$ the moduli space $\Mod$ is a smooth manifold of dimension $\mu(x^-) - \mu(x^+)$ for all $x^\pm \in \P^a(H)$, see \cite{FHS95}. \\

\ni
Fix $\J\in \Jr(H)$. It is shown in \cite[Section 1.4]{AS06} that the manifold $\Mod$ can be oriented in a way which is coherent with gluing. In particular, when $\mu(x^-)-\mu(x^+) = 1$, $\Mod$ is an oriented one-dimensional manifold.\\ 

\ni
Note that the group $\re$ freely acts on $\M(x^-, x^+, H; \J)$ by time-shift. We will use the notation
$$
\Modti := \Mod / \re.
$$
Denoting by $[u]$ the equivalence class of $u$ in the zero-dimensional manifold $\Modti$, we define
$$
\eps([u]) \in \{-1,1\}
$$
to be $+1$ if the $\re$-action is orientation preserving on the connected component of $\Mod$ containing $u$ and $-1$ in the opposite case.\\

\ni
For $x^\pm \in \P^a(H)$ with $\mu(x^-) = \mu(x^+)+1$ let
$$
n(x^-,x^+,H;\J) := \sum_{[u] \in \Modti} \eps([u])
$$
be the oriented count of the finite set $\Modti$. For $k\in\Z$ one can define the Floer boundary operator
$$
\partial_k(\J) : CF^a_k(H;\F_p) \ra CF^a_{k-1}(H;\F_p)
$$
as the linear extension of
$$
\partial_k(\J) x^- = \sum n(x^-,x^+,H;\J)x^+
$$
where $x^- \in \P^a(H)$ has index $\mu(x^-)=k$ and the sum runs over all $x^+ \in \P^a(H)$ of index $\mu(x^+) = k-1$. Then $\partial_{k-1}(\J) \circ \partial_k(\J) = 0$ for each $k$. The proof makes use of the compactness of the 0- and the 1-dimensional components of $\Modti$, see \cite{Flo881,Schw93,AS06}. 

\subsubsection{The Floer homology group}
The $k^{th}$ Floer homology group is defined by
$$
HF^a_k(H;\F_p) := \frac{\ker \partial_k(\J)}{\im \partial_{k+1}(\J)}
$$
As our notations suggests, it does not depend on the choices involved in the construction. They neither depend on coherent orientations up to canonical isomorphisms, see \cite[Section 1.7]{AS06}, nor on $\J\in\Jr$ up to natural isomorphisms, as a continuation argument shows, see \cite{Flo881,Schw93}. The groups $HF^a_k(H;\F_p)$ do depend, however, on $a< 2(n+1)$ and $H\in \H_6(\gg)$. \\

\ni
In the sequel, the field $\F_p$ is fixed throughout. We shall therefore often write $CF^a_*(H)$ and $HF^a_*(H)$ instead of $CF^a_*(H;\F_p)$ and $HF^a_*(H;\F_p)$.

\subsection{Continuation homomorphisms} 
\label{section:continuation homomorphisms}

Let $\be:\re\ra[0,1]$ be the function from \eqref{beta}. Given two functions $\hl, \hg \in \H_6(\gg)$ with $\hl(t,x)\leq \hg(t,x)$ for all $(t,x)\in S^1\times\Tm$, we define the monotone homotopy
\begin{equation}
\hs= (1-\be(s)) \hl  + \be(s) \hg.
\label{hs} 
\end{equation}
Then $\hs\in\H_\rh(\gg)$ for each $s$, and 
$$
\hs = \begin{cases} \hl & \mbox{for }s\leq 0 \\
					    \hg & \mbox{for }s\geq 1
\end{cases}.$$
Consider the equation 
\begin{equation}
\begin{cases}
\partial_s u + J_{s,t} (u) (\partial_t u - X_{H_{s,t}} (u))=0 \\
\lim_{s\ra±\infty} u(s,t)=x^± (t) \text{ uniformly in t,}
\end{cases}
\label{floereq3}
\end{equation}
where the map $s\mapsto \{J_{s,t}\}$ for $s\in \re$ and $t\in [0,1]$, is a {\it regular homotopy} of families $\{J_t\}$ of almost complex structures on $T^*M$. This means that 
\begin{itemize}
\item $J_{s,t}$ is $\om$-compatible, convex and independent of $s$ and $t$ outside $D(6)$,
\item $J_{s,t}=J_t^- \in \J_{reg}(\hl)$ for $s\leq 0$, 
\item $J_{s,t}=J_t^+ \in \J_{reg}(\hg)$ for $s\geq 1$.
\end{itemize}

\ni
The following lemma is well-known. 

\begin{lem}
Assume that $u:\re\times S^1\ra T^*M$ is a solution of equation \eqref{floereq3}. Then
$$\A_{\hg} (x^+) \leq \A_{\hl} (x^-) - \int_0^1 \int^\infty_{-\infty} \be'(s) (H_1-H_0) (t,u(s,t))\;ds\;dt.$$
\label{l:decreasing action}
\end{lem}

\begin{proof}
Extend the map $u:Z=\re\times S^1\ra T^*M$ to a smooth map 
$$\widehat{u}:S^2=D_- \cup Z \cup D^+ \ra \Tm.$$
Since $\om=d\la$ is exact, we find, using Stoke's Theorem and taking orientations into account, that 
$$0=\int_{S^2} d(\widehat{u}^* \la) = \int_{\widehat{u}(S^2)} d\la = \int_{\widehat{u}(D_-)} d\la + \int_{u(Z)} d\la + \int_{\widehat{u}(D_+)} = \int_{x^-} \la + \int_{u(Z)} d\la -\int_{x^+} \la.$$
Moreover, by the definition of the homotopy \eqref{hs} we obtain that
$$\frac{d}{ds}\hs (t,u(s,t)) = d\hs(t,u(s,t)) (\partial_s u) + \be'(s) (H_1-H_0) (t,u(s,t)).$$
This and the asymptotic boundary condition in \eqref{floereq3} yield
$$\begin{aligned}
\int_0^1 \int^\infty_{-\infty} d\hs (t,u(s,t)) (\partial_s u)\;ds\;dt=& \int_0^1 \hg (t,x^+) \;dt - \int_0^1 \hl(t,x^-)\;dt\\
& - \int_0^1 \int^\infty_{-\infty} \be'(s) (H_1 - H_0) (t,u(s,t))\; ds\; dt.
\end{aligned}$$
Together with the compatibly $g_{s,t}(v,w)=\om(v,J_{s,t}w)$, Floer's equation in \eqref{floereq3}, Hamilton's equation \eqref{hameq} and the definition of the action functional we obtain that
$$\begin{aligned}
0& \leq \int_0^1 \int_{-\infty}^\infty g_{s,t}(\partial_s u, \partial_s u) \;ds\;dt \\
  & = \int_0^1 \int_{-\infty}^\infty g_{s,t}(\partial_s u, J_{s,t} (\partial_t u -X_{\hs}(u)))\;ds\;dt \\
  & = - \int_{Z} u^* \om - \int_0^1 \int_{-\infty}^\infty \om(X_{\hs} (u), \partial_t u)\;ds\;dt \\
  & = \int_{x^-} \la - \int_{x^+} \la + \int_0^1 \hg(t,x^+) \; dt - \int_0^1 \hl(t,x^-) \; dt \\ 
  & \phantom{=} -\int_0^1\int_{-\infty}^\infty \be'(s) (H_1 - H_0) (t,u(s,t)) \;ds\;dt \\
  & = \A_{\hl} (x^-) - \A_{\hg} (x^+) - \int_0^1\int_{-\infty}^\infty \be'(s) (H_1 - H_0) (t,u(s,t)) \;ds\;dt
\end{aligned}$$
as claimed.
\end{proof}

\ni
In view of this lemma, the action decreases along solutions $u$ of \eqref{floereq3}. By counting these solutions one can therefore define the Floer chain map
$$\phi_{\hg\hl}: CF^a_* (\hl) \ra CF^a_* (\hg),$$
see \cite{CFH95}. The induced {\it continuation homomorphism}
$$\Phi_{\hg\hl}: HF^a_* (\hl) \ra HF^a_* (\hg)$$
on Floer homology does not depends on the choice of the regular homotopy $\{J_{s,t}\}$ used in the definition. An important property of these homomorphisms is naturality with respect to concatenation,

\begin{equation}
\Phi_{H_3 H_2} \circ \Phi_{H_2 H_1} = \Phi_{H_3 H_1} \quad \mbox{for } H_1 \leq H_2 \leq H_3.
\label{concatenation}
\end{equation}

\ni
Another important fact is the following invariance property, which is proved in \cite{CFH95} and \cite[Section 4.5]{BPS03}.

\begin{lem}
If $a\not\in \Sp(\hs)$ for all $s\in [0,1]$, then $\Phi_{\hg\hl}: HF^a_* (\hg) \ra HF^a_* (\hl)$ is an isomorphism.
\label{isom}
\end{lem}

\section{From Floer homology to the homology of the free loop space}

\subsection{Continuation homomorphisms}
\label{section:Continuation homomorphisms2}
The goal of the section is to relate the groups 
$$
HF_*^a(\gl), \; HF_*^a(K_n) \mbox{ and } HF_*^a(\gg).
$$
By Proposition \ref{spectrum}, the set 
$$
\Sp (n)= \bigl(\Sp(\gl) \cup \Sp(K_n) \bigr) \cap\; ]n - 2c, (n+1) -2c[
$$
is finite. In particular, 
$$
\delta_n := \min \{s\in \Sp(K_n) \mid s> n-2c\} > n-2c.
$$

\ni
The set
$$
\Sp(\gl)\;\cap \;]n-2c,(n+1)-2c[\; \cap \; ]n-2c,\delta_n[
$$
is also finite. Define $d_n$ as the minimal distance between two elements of this set. Recall from (W2), we assume that the constant $c_n$ fulfills
\begin{equation}
0<c_n<\frac{d_n}{4}.
\label{d:cn}
\end{equation}
Choose 
\begin{equation}
a_n \in \bigl(\;]n-2c, (n+1)-2c[\; \cap \; ]n,\delta_n[\;\bigr) \backslash \Sp(n),
\label{def:a}
\end{equation} 
and
\begin{equation}
b_n\in \;]a_n, (n+1) -2c[
\end{equation}
such that
$$
]a_n-c_n, b_n+c_n[\; \cap \; \Sp(\gl) = \emptyset.
$$
Then $a_n$ is not in the action spectrum of $\gl$ and $K_n$, and by Proposition~\ref{action2}
$$
\A_{K_n} (x) \leq a_n \Longrightarrow  x\in\P(K_n) \cap D(\Si).
$$

\begin{rmk}
\label{rmk:action level} 
In the degenerate case when $c = 0$ we have that
$$\A_{K_n} (x) \leq a_n \Longleftrightarrow  x\in\P(K_n) \cap D(\Si).$$
\end{rmk}

\ni
Next we want to show that $HF^{a_n} (\gl)$ and $HF^{a_n/\si}(\gg)$ are naturally isomorphic when $n\geq 2$. This is a special case of a generalization of Lemma \ref{isom} stated  in \cite[Proposition 1.1]{Vit99}.\\

\ni
Set
$$
a_n(s):=\frac{a_n}{1+\be(s) (\si-1)}
$$
and
$$
b_n(s) := \frac{b_n}{1+\be(s) (\si-1)}.
$$

\ni
For $(s,t)$ in the band bounded by the graphs of $a_n(s)$ and $b_n(s)$ we have that $t\not\in\Sp(\gs)$ in view of the Non-crossing Lemma \ref{crossing}. Choose a partition $0=s_0 <s_1<\ldots <s_k<s_{k+1}=1$ so fine that
$$
b_n(s_{j+1}) >a_n(s_j) \mbox{ for } j= 1,\ldots,k.
$$
Abbreviate $\aa_j=a_n(s_j)$ and $G_j=G_{n,s_j}$. Then 
$$
\aa_j \not\in \Sp(\gs) \mbox{ for } s\in [s_j,s_{j+1}].
$$

\begin{figure}[h]
  \centering
  \psfrag{a}[][][1]{$s_0=0$}
  \psfrag{b}[][][1]{$s_{k+1}=1$}
  \psfrag{c}[][][1]{$c_n$}
  \psfrag{d}[][][1]{$b_n$}
  \psfrag{e}[][][1]{$a_n$}
  \psfrag{f}[][][1]{$\aa_1$}
  \psfrag{g}[][][1]{$\aa_2$}
  \psfrag{h}[][][1]{$\aa_k$}
  \psfrag{i}[][][1]{$s_1$}
  \psfrag{j}[][][1]{$s_2$}
  \psfrag{k}[][][1]{$s_k$}
  \includegraphics[width=10cm]{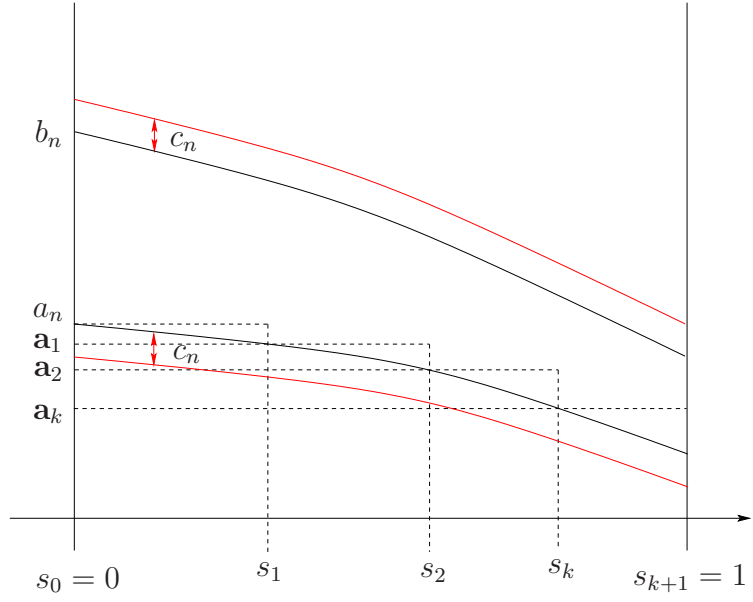}
  \caption{The curves $a_n(s)$ and $b_n(s)$.}
  \label{label2}
\end{figure}

\ni
Together with Lemma \ref{isom} we find that 
$$
\Phi_{G_{j+1}G_j} : HF_*^{\aa_j}(G_j) \ra HF_*^{\aa_{j}}(G_{j+1})
$$ 
is an isomorphism. Since $[\aa_{j+1}-c_n,\aa_j+c_n] \cap \Sp (G_{j+1}) = \emptyset $, we have $HF_*^{\aa_{j}}(G_{j+1})=HF_*^{\aa_{j+1}}(G_{j+1})$ and thus 
$$
\widehat{\Phi}_{G_{j+1}G_j} \equiv \Phi_{G_{j+1}G_j} : HF_*^{\aa_j}(G_j) \ra HF_*^{\aa_{j+1}}(G_{j+1})
$$
is an isomorphism. Recalling that $\aa_0=a_n$ and $\aa_{k+1}=a_n/\si$ we obtain that the composition
$$
\widehat{\Phi}_{\gg \gl} := \widehat{\Phi}_{G_{k+1}G_k} \circ \ldots \circ \widehat{\Phi}_{G_{2}G_1}\circ \widehat{\Phi}_{G_{1}G_0} : HF_*^{a_n}(\gl) \ra HF_*^{a_n/\si}(\gg)
$$
is an isomorphism. Let 
$$
 HF (\io_n) : HF_*^{a_n/\si}(\gg) \ra HF_*^{a_n}(\gg)
 $$
be the homomorphism induced by the inclusion $\io_n : CF_*^{a_n/\si}(\gg) \ra CF_*^{a_n}(\gg)$.

\begin{prop}
\label{p:diagram}
For each $k$ there is a commutative diagram of homomorphisms
\begin{equation}
\xymatrix{
  & HF^{a_n/\si}_k(\gg) \ar[rd]^{HF_k (\io_n)} & \\
  HF^{a_n}_k(\gl) \ar[ru]^{\widehat{\Phi}_{\gg \gl}} \ar[rr]^{\Phi_{\gg \gl}} \ar[rd]_{\Phi_{K_n \gl}} & & HF^{a_n}_k(\gg)   \\
&  HF^{a_n}_k(K_n) \ar[ru]_{\Phi_{\gg K_n}} & }
\label{diagram}
\end{equation}
Moreover $\widehat{\Phi}_{\gg \gl}$ is an isomorphism.
\end{prop}

\begin{proof}
By construction, the isomorphism $\widehat{\Phi}_{\gg \gl}$ is induced by the composition of Floer chain maps
$$\widehat{\phi}_{G_{j+1}G_j}:CF^{\aa_j}_* (G_j) \ra CF^{\aa_{j+1}}_* (G_{j+1}) \subset CF^{\aa_j}_* (G_{j+1}).$$
Therefore, $\io_n \circ \widehat{\Phi}_{\gg \gl}$ is induced by the composition of Floer chain maps
$$\phi_{G_{j+1}G_j} :CF^{\aa_{0}}_* (G_j) \ra CF^{\aa_0}_* (G_{j+1}).$$
By the concatenation property \eqref{concatenation} this composition induces the same map in Floer homology as
$$\phi_{\gg \gl} : CF^{\aa_{0}}_* (\gl) \ra CF^{\aa_0}_* (\gg).$$
The upper triangle therefore commutes. The lower triangle commutes in view of $\gl \leq K_n \leq \gg$ and according to \eqref{concatenation}.
\end{proof}

\begin{cor}
\label{corollary:bouded homology}
$\dim HF^{a_n}_k (K_n) \geq \rank \left( HF_k (\io): HF^{a_n / \si}_k (\gg) \ra HF^{a_n}_k (\gg) \right) $. 
\end{cor}

\subsection{To the homology of the free loop space}

The goal of this section is to prove that the rank of the map $HF_k (\io): HF^{a_n / \si}_k (\gg) \ra HF^{a_n}_k (\gg)$ is bounded below by the rank of $\io_k:  H_k (\La^{n a_n}) \ra H_k (\La)$. This will be done by first applying the Abbondandolo--Schwarz isomorphism, \cite{AS06},  between the Floer homology of $\gg$ and the Morse homology of its Legendre transform $L$, and by then applying the Abbondandolo--Mayer isomorphisms,\cite{AM06}, from the latter Morse homology to the homology of the free loop space $\La M$ of the base.

\begin{thm}
Let $(M,g)$ be a smooth, closed, orientable Riemannian manifold and $K_n$ be as above. It holds that
$$\dim HF^{a_n}_k (K_n; \F_p) \geq \rank \io_k H_k (\La^{n a_n} ; \F_p).$$
\label{t:exp growth}
\end{thm}

\ni
We start with proving

\begin{prop}
For each $k$ there is a commutative diagram of homomorphisms
$$
\xymatrix{
HF^{a_n /\si}_k(\gg;\F_p) \ar[r]^\cong \ar[d]_{HF_k (\io)} & H_k(\La^{n a_n } ;\F_p) \ar[d]^{H_k(\io)} \\
HF^{a_n}_k( \gg;\F_p) \ar[r]^\cong &  H_k(\La^{n \si a_n} ;\F_p)  } 
$$
where the horizontal maps are isomorphisms and the right map $H_k(\io)$ is induced by the inclusion $\La^{n a_n} \hookrightarrow \La^{n \si a_n}$.
\label{p:commut}
\end{prop}  

\begin{proof}
Let $L:S^1\times TM \ra \re$ be the Legendre transform of $\gg$, let
$$\E_L (q) := \int_0^1 L(t,q(t), \dot{q}(t))\,dt$$
be the corresponding functional on $\La M$, and let 
$$\La^b_L := \{ q \in \La M \mid \E_L (q) \leq b\}.$$
Applying Theorem 3.1 of Abbondandolo--Schwarz in \cite{AS06} we obtain for each $b>0$ the isomorphism
$$HF_k^b(\gg; \F_p) \xleftarrow{\Theta^b_k} HM^b_k (L ; \F_p),$$
where $HM^b (L ; \F_p)$ denotes the Morse homology below level $b$ of $\E_L $ constructed in \cite{AM06}, see also \cite[Section 2]{AS06}. The Abbondandolo--Schwarz chain isomorphisms 
$$CF^b_k (\gg;\F_p) \xleftarrow{\theta^b_k} CM^b_k (L;\F_p)$$
between the Floer and the Morse chain complexes commute with the inclusions
$$CF^b_k (\gg;\F_p) \hookrightarrow CF^{b'}_k (\gg;\F_p)$$
and
$$ CM^b_k (L;\F_p)  \hookrightarrow CM^{b'}_k (L;\F_p)$$
for $b'>b$, see \cite[p. 298]{AS06}. Therefore, the induced diagram of homology groups commutes, 

$$
\xymatrix{
HF^{b}_k(\gg;\F_p)  \ar[d]_{HF_k (\io)} & HM^b_k(L,\F_p) \ar[l]_{\Theta^b_k}\ar[d]^{HM_k(\io)} \\
HF^{b'}_k( \gg;\F_p)  &  HM^{b'}_k(L;\F_p)\ar[l]_{\Theta^{b'}_k}  } 
$$

\ni
Moreover, Abbondandolo and Mayer constructed chain isomorphisms
$$CM^b_k (L;\F_p) \xrightarrow{\tau^b_k} C_k(\La^{b}_L;\F_p)$$
between the Morse and the singular chain complexes which commute with the inclusions
$$ CM^b_k (L;\F_p)  \hookrightarrow CM^{b'}_k (L;\F_p)$$
and 
$$ C_k (\La^b_L;\F_p)  \hookrightarrow C_k (\La^{b'}_L;\F_p)$$
for $b'>b$ (see \cite{AM06} and \cite[Section 2.3]{AS06}). Thus the induced diagram of homology groups commutes, i.e.

$$
\xymatrix{
HF^{b}_k(\gg;\F_p)\ar[r]^{T^b_k}  \ar[d]_{HF_k (\io)} & H_k(\La^{b}_L,\F_p) \ar[d]^{HM_k(\io)} \\
HF^{b'}_k( \gg;\F_p) \ar[r]_{T^{b'}_k}   &  H_k(\La^{b'}_L;\F_p) } 
$$
Notice now that $L(t,q,v) = \frac{1}{n\si} \demi |v|^2 - W_n(t,q)$, whence by Lemma \ref{lem:legendre} $\La^{a_n}_L$ retracts on $\La^{n\si a_n}$ as $[a_n - c_n, a_n +c_n]$ does not belong to the spectrum of $\gl$. Proposition~\ref{p:commut} follows.
\end{proof}

\ni
Consider now the commutative diagram 

\begin{equation} 
\label{e:diagram loop space}
\xymatrix{
H_k(\La^{na_n};\F_p) \ar[d] \ar[dr]^{\io_k} \\
H_k(\La^{n\si a_n};\F_p) \ar[r] & H_k(\La M;\F_p)   }
\end{equation}
induced by the inclusion $\La^{na_n} \subset \La^{n\si a_n} \subset \La M$. In view of Proposition \ref{p:commut} and \eqref{e:diagram loop space} we have that 
$$
\rank \left( HF_k (\io): HF^{a_n / \si}_k (\gg;\F_p) \ra HF^{a_n}_k (\gg;\F_p) \right)
$$
is bounded below by 
$$
\rank \left( \io_k: H_k(\La^{na_n};\F_p) \ra H_k(\La M;\F_p) \right).
$$
Together with Corollary \ref{corollary:bouded homology} this yields
$$
\dim HF^{a_n}_k (K_n; \F_p) \geq \rank \left( \io_k: H_k(\La^{na_n};\F_p) \ra H_k(\La M;\F_p)\right).
$$
Which conclude the proof of Theorem \ref{t:exp growth}.\\

\ni
Suppose $(M,g)$ is energy hyperbolic, $h:= C(M,g)>0$. By definition of $C(M,g)$, there exist $p\in \Pr$ and $N_0 \in \N$ such that for all $m\geq N_0$,
$$
\sum_{k\geq 0} \dim \io_k H_k (\La^{\demi m^2}; \F_p) \geq e^{\frac{1}{\sqrt{2}}hm}.
$$ 
Therefore there exists $N\in \N$ such that for all $n\geq N$, 
$$
\sum_{k\geq0} \rank \io_k \geq e^{hn}.
$$
Together with Proposition \ref{p:commut} we find that 
\begin{equation}
\sum_{k\geq0} \rank HF_k(\io) \geq \sum_{k\geq0} \rank \io_k \geq e^{hn}.
\label{e:energy hyp ineq}
\end{equation}

\ni
Similarly, if $c(M,g)>0$ we find that there exists $p\in \Pr$ and $N\in \N$ such that for all $n\geq N$,
\begin{equation}
\sum_{k\geq0} \rank HF_k(\io) \geq \sum_{k\geq0} \rank \io_k \geq n^{h}.
\label{e:energy ell ineq}
\end{equation}

 \clearemptydoublepage

\chapter{Morse-Bott homology}
\label{chapter:morse-bott homology}
In order to prove Theorem B, we use Morse--Bott homology and its correspondence to Floer homology. This chapter is devoted to the definition of Morse--Bott homology. In the first section we give a definition of a generic hypersurface in order to achieve a Morse--Bott situation. In the second section, we associate to a relevant Hamiltonian an additional perturbation in order to obtain an isomorphism between its Floer homology and Morse--Bott homology. The main tool of this isomorphism is the {\it Correspondence Theorem} due to Bourgeois--Oancea, \cite{BO09}, which will be discussed in the last section.

\section{A Morse-Bott situation}
Let $(X,g)$ be a Riemannian manifold. A smooth function $f\in C^\infty (M, \re)$ is called {\it Morse-Bott} if 
$$
\crit (f) := \{x\in X \mid df(x) = 0\}
$$
is a submanifold of $X$ and for each $x\in \crit (f)$ we have
$$
T_x\crit(f) = \ker (\Hess(f)(x)).
$$
Let $F: \Tm \ra\re$ be a Hamiltonian such that $F|_\Si \equiv 1$. The action functional $\A_F$ is invariant under the $S^1$-action on $\La M$ given by $\ga (t) \mapsto \ga(t+\cdot)$. In order that it's critical points are Morse--Bott manifolds we make the following nondegeneracy assumption on the Reeb flow on $\Si$.

\begin{itemize}
\item[(A)] The closed Reeb orbits of $\Si$ are of Morse-Bott type, i.e. for each $\tau$ the set $\Od_R(\tau)$ of $\tau$-periodic Reeb orbits is a closed submanifold and every closed Reeb orbit is transversally nondegenerate, i.e.
$$\det(1-d\vp_R^\tau (\ga(0))|\xi) \neq 0.$$  
\label{assumption}
\end{itemize}
Assumption (A) is generically satisfied. We will then say that $\Si$ is generic. If (A) is satisfied, then the action functional $\A_F$ is Morse-Bott. \\

\ni
There are several ways to deal with Morse--Bott situations. The first possibility we are going to use, is to choose an additional small perturbation to get a Morse situation. Our perturbation is the one introduced in \cite{CFHW96}. The second possibility is to choose an additional Morse function on the critical manifold. The chain complex is then generated by the critical points of this Morse function while the boundary operator is defined by counting trajectories with cascades. This approach was carried out by Frauenfelder in \cite[Appendix A]{Frauen04} and by Bourgeois and Oancea in \cite{BO09}. \\

\ni
In \cite{BO09}, Bourgeois and Oancea studied a particular class of {\it admissible Hamiltonians} corresponding in our setting to Hamiltonians $H:\Tm \ra\re$ such that
\begin{itemize}
\item[(i)] $H|_{D(\Si)}$ is a $C^2$-small Morse function and $H<0$ on $D(\Si)$;
\item[(ii)]$H(q,p)=h(p)$ outside $D(\Si)$, where $h$ is a strictly increasing function, convex at infinity and such that the 1-periodic orbits of $X_h$ are in one-to-one correspondence with closed Reeb orbits.    
\end{itemize}
The 1-periodic orbits of $X_H$ then fall into two classes:
\begin{itemize}
\item[(1)] critical points of $H$ in $D(\Si)$;
\item[(2)] nonconstant 1-periodic orbits of $X_h$.
\end{itemize}
For such Hamiltonians, they obtained an isomorphism of the homology of the Morse-Bott complex with the Floer homology with respect to the same Hamiltonian with an additional perturbation.\\ 

\ni
In the following we will follow their approach. However in view of our particular choice of Hamiltonian, we will deal with a slightly simpler situation. Instead of dealing with two classes of periodic orbits as in \cite{BO09}, we only have nonconstant 1-periodic orbits. Moreover, as we are working below an energy level $a$, we are able to apply their {\it Correspondence Theorem} to our situation.

\section{An additional perturbation}
\label{section:an additional perturbation}
Let $\Si$ be a fiberwise starshaped hypersurface and assume $\Si$ is generic in the sense of~(A).
As in section \ref{section:relevant hamiltonians}, we can define a Hamiltonian function $F:\Tm\ra\re$ by the two conditions
$$
F|_\Si \equiv 1, \quad F(q,sp)=s^2F(q,p), \quad s\geq 0 \mbox{ and }(q,p)\in \Tm,
$$
and smoothen $F$ by composing it with the smooth function $f:\re\ra\re$.\\

\ni
By the assumption of genericity, all the elements of $\P (f\circ F)$ are transversally nondegenerate. Thus every orbit $\ga \in \P (f\circ F)$ gives rise to a whole circle of nonconstant 1-periodic orbits of $X_{f\circ F}$ whose parametrizations differ by a shift $t\in S^1$. We denote by $\Sg$ the set of such orbits, so that $S_\ga = S_{\ga(\cdot + t)}$ for all $t\in S^1$. The Maslov index is still well-defined for these orbits. But as the eigenspace to the eigenvalue 1 is two dimensional, their index will be a half integer, see section \ref{maslov index}.\\

\ni
The following construction is the one described in \cite[Proposition 2.2]{CFHW96} and \cite{BO09}. For each $\ga \in \P (f\circ F)$, we choose a Morse function $\fg : \Sg \ra \re$ with exactly one maximum $Max$ at $t_1$ and one minimum $min$ at $t_2$. We denote by $\lga \in \N$ the maximal natural number so that $\ga (t+1/\lga) = \ga(t)$ for all $t\in S^1$. We choose a symplectic trivialization $\psi := (\psi_1, \psi_2): U_\ga \ra V \subset S^1 \times \re^{2n-1}$ between open neighborhoods $U_\ga \subset T^*M$ of $\ga(S^1)$ and $V$ of $S^1 \times \{0\}$ such that $\psi_1 (\ga(t)) = \lga t$. Here $S^1\times\re^{2n-1}$ is endowed with the symplectic form $\om_0 := \sum_{i=1}^{2n} dp_i \wedge dq_i$, $q_1 \in S^1$, $(p_1, q_2, p_2, \ldots, p_n) \in \re^{2n-1}$. Let $\rho: S^1\times \re^{2n-1} \ra [0,1]$ be a smooth cutoff function supported in a small neighborhood of $S^1 \times \{0\}$ so that $\rho|_{S^1 \times \{0\}} \equiv 1$. For $\delta>0$ and $(t,q,p) \in S^1\times U_\ga$, we define
$$
\Fd (t,q,p) := f\circ F(q,p) + \delta \rho(\psi(q,p)) \fg (\psi_1(q,p)-\lga t).
$$
We will denote the perturbation added to $f\circ F$ by $h_\delta$.\\

\ni
This perturbation destroys the critical circles and gives rise to two obvious solutions of $\dot{x} = X_\Fd (t,x)$, namely 
$$
\gaM(t) = \ga(t+ t_1/\lga)
$$
and
$$
\gam(t) = \ga(t+ t_2/\lga).
$$

\ni
By construction these orbits are nondegenerate, their indexes are then integer. It is shown in \cite[Proposition 2.2]{CFHW96} that for $\delta$ sufficiently small $\gam$ and $\gaM$ are the only elements of $\P(\Fd)$ in $U_\delta$ and their indexes are given by
$$
\mu(\gam) = \mu(\ga)-\demi \quad \mbox{and} \quad \mu(\gaM) = \mu(\ga)+\demi.
$$
Recall from section \ref{maslov index} that $\mu(\ga)$ is a half integer. Therefore $\mu(\gam)$ and $\mu(\gaM)$ are integers.

\section{Morse-Bott homology}
Consider the space of autonomous smooth Hamiltonian functions 
$$
\H'_6 (\gg) =\{H: T^*M \ra \re \mid H=\gg \mbox{ on } S^1\times T^*M \backslash D(6) \}.
$$
Let $H \in \H'_6(\gg)$. 
Fix $a < n+1$, and suppose that every orbit $\ga \in \P^a (H)$ gives rise to a whole circle of non-constant periodic orbits $\ga$ of $X_{H}$, which are transversally non-degenerate. We denote by $S_\ga$ the set of such orbits. 

\subsubsection{The chain groups}

For each $S_\ga$, $\ga \in \P^a(H)$ choose a Morse function $f_\ga : S_\ga \ra \re$ with exactly one maximum and one minimum. We denote by $\gaM$, $\gam$ the orbits in $S_\ga$ starting at the maximum and the minimum of $f_\ga$ respectively. \\

\ni
The $k^{th}$ Morse-Bott chain group is defined as the finite dimensional $\F_p$-vector space freely generated by the the set of $\gam$, $\gaM$ of Maslov index $k$ so that $\ga$ is an element of $\P^a(H)$, and the full chain group is defined as
$$
BC^a_* (H; \F_p) := \bigoplus_{\ga \in \P^a(H)} \F_p \lg \ga_m , \ga_M \rg
$$
where the grading is given by the Maslov index. 

\subsubsection{Almost complex structures}

Following section \ref{almost complex structures} in the definition of Floer homology, we are considering the set $\Js$ of $t$-dependent $\om$-compatible almost complex structures $\J$ on $T^*M$ such that $\J$ is convex and independent of $t$ on $T^*M \backslash \Dcirc(6)$. Recall that an $\om$-compatible almost complex structure $J$ on $T^*M$ is {\it convex} on $T^*M \backslash \Dcirc(6)$ if
$$\begin{aligned}
J\xi = \xi, & \\
\om(Y(x),J(x)Y(x))=1,&\quad \mbox{for } x\in \partial D(6)\\
d\psi_t(x)J(x) = J(\psi_t(x))d\psi_t(x),& \quad \mbox{for } x\in \partial D(6) \mbox{ and } t\geq 0,
\end{aligned}$$
where $Y$ denote the Liouville vector field $\sum_i p_i\partial p_i$. 

\subsubsection{Trajectories with cascades}

Fix $\J \in\Js$, $\oga, \uga \in \P^a(H)$. We denote by
$$
\widehat{\M}(S_\oga, S_\uga; H, \J)
$$
the space of solutions $u:\re\times S^1 \ra T^*M$ of Floer's equation 
\begin{equation}
\partial_s u + J_t (u) (\partial_t u - X_{H}(t,u))=0 
\end{equation}
subject to the asymptotic conditions 
$$
\begin{aligned}
\lim_{s\ra -\infty} u(s,t) &= \oga(t), \\
\quad \lim_{s\ra +\infty} u(s,t)& = \uga(t),\\
\quad \lim_{s\ra \pm\infty} \partial_s u (s,t) &= 0 
\end{aligned}
$$
uniformly in $t$. The Morse-Bott moduli spaces of Floer trajectories are defined by 
$$
\Modgg :=  \widehat{\M}(S_\oga, S_\uga; H, \J) /\re.
$$

\ni
It is shown in \cite[Proposition 3.5]{BO09} that there exists a dense subset $\Jr(H) \subset \Js$ such that given $\J \in \Jr(H)$ the Morse-Bott moduli spaces of Floer trajectories are smooth manifolds, and their dimensions are 
$$
\dim \Modgg = \mu(\oga) - \mu(\uga).
$$

\ni
We have natural evaluation maps
$$
\evo : \Modgg \ra S_\oga \quad
$$
and
$$
\evu : \Modgg \ra S_\uga
$$
defined by 
$$
\evo([u]) := \lim_{s\ra-\infty} u(s,\cdot), \quad \evo([u]) := \lim_{s\ra \infty} u(s,\cdot).
$$

\bigskip
\ni
For each $S_\ga$, $\ga \in \P^a(H)$, consider the Morse function $f_\ga : S_\ga \ra \re$ and its negative gradient flow $\phi^s_\ga$ with respect to $-\nabla f_\ga$. Define the stable an unstable manifold by
$$
\begin{aligned}
W^s (p;-\nabla f_\ga ) &:= \{z\in S^1 \mid \lim_{s\ra+\infty} \phi^s_\ga (z) = p\}, \\
W^u (p; -\nabla f_\ga) &:= \{z\in S^1 \mid \lim_{s\ra-\infty} \phi^s_\ga (z) = p\}.
\end{aligned}
$$ 
Then $W^u(\max)= S_\ga \backslash \{\min\}$, $W^s(\max) = \{\max\}$, $W^u(\min) = \{\min\}$, $W^s(\min)= S_\ga \backslash \{\max\}$.\\

\ni
It is shown in \cite{BO09} that for a generic choice of these Morse functions, all the maps $\evo$ are transverse to the unstable manifolds $W^u (p)$, $p\in \crit(f_\ga)$, all the maps $\evu$ are transverse to the stable manifolds $W^s (p)$, $p\in \crit(f_\ga)$, and all pairs
$$
(\evo, \evu) : \Modgg \ra S_\oga \times S_\uga,
$$
$$
(\evo, \evu) : \M(S_\oga, S_{\ga_1}; H, \J)\; { }_\evu\! \times_\evo \; \M(S_{\ga_1}, S_{\uga}; H, \J)\ra S_\oga \times S_\uga
$$
are transverse to products $W^u(p)\times W^s(q)$, $p\in \crit(f_\oga)$, $q\in \crit(f_\uga)$. We denote by $\Fr(H, \J)$ the set consisting of collections $\{f_\ga\}$ of Morse functions which satisfy the above transversality conditions.\\

\ni
Let now $\J \in \Jr(H)$ and $\{f_\ga\} \in \Fr(H,\J)$. For $p\in \crit(f_\ga)$ we denote the Morse index by 
$$
\ind (p) := \dim W^u(p; -\nabla f_\ga).
$$
Let $\oga, \uga \in \P^a(H)$ and $p\in \crit(f_\oga)$, $q\in \crit(f_\uga)$. For $m\geq 0$ denote by 
$$\Modmpq$$
the union for $\tilde{\ga}_1, \ldots, \tilde{\ga}_{m-1} \in \P^a(H)$ of the fiber products
\begin{equation}
W^u(p) \times_\evo \; (\M(S_\oga, S_{\tilde{\ga}_1}; H, \J) \times \re^+ ) \; { }_{\vp_{f_{\tilde{\ga}_1}}\circ\evu}\! \times_\evo \; (\M(S_{\tilde{\ga}_1}, S_{\tilde{\ga}_2}; H, \J)\times \re^+)
\label{e:brokentraj}
\end{equation}
$$
\; { }_{\vp_{f_{\tilde{\ga}_2}}\circ \evu}\! \times_\evo \; \ldots \; { }_{\vp_{f_{\tilde{\ga}_{m-1}}} \circ \evu}\! \times_\evo \; \M(S_{\tilde{\ga}_{m-1}}, S_{\uga}; H, \J) \; { }_{\evu} \! \times W^s(q).
$$
This is a smooth manifold of dimension 
$$
\begin{aligned}
\dim \Modmpq &= (\mu(\oga) + \ind(p)) - (\mu(\uga) + \ind (q)) -1 \\
 &= \mu(\oga_p) - \mu(\uga_q) -1 .
\end{aligned}
$$
We denote 
$$
\Modpq = \bigcup_{m\geq0} \Modmpq
$$ 
and we call this the moduli space of {\it Morse-Bott trajectories with cascades}, whereas the space $ \M_m (p, q; H, \{f_\ga \}, \J) $ is called the {\it moduli space of trajectories with cascades with $m$ sublevels}. \\

\begin{figure}[h]
  \centering
  \psfrag{a}[][][1]{$p$}
  \psfrag{b}[][][1]{$q$}
  \psfrag{c}[][][1]{$-\nabla f$}
  \psfrag{d}[][][1]{$u$}
  \includegraphics[width=10cm]{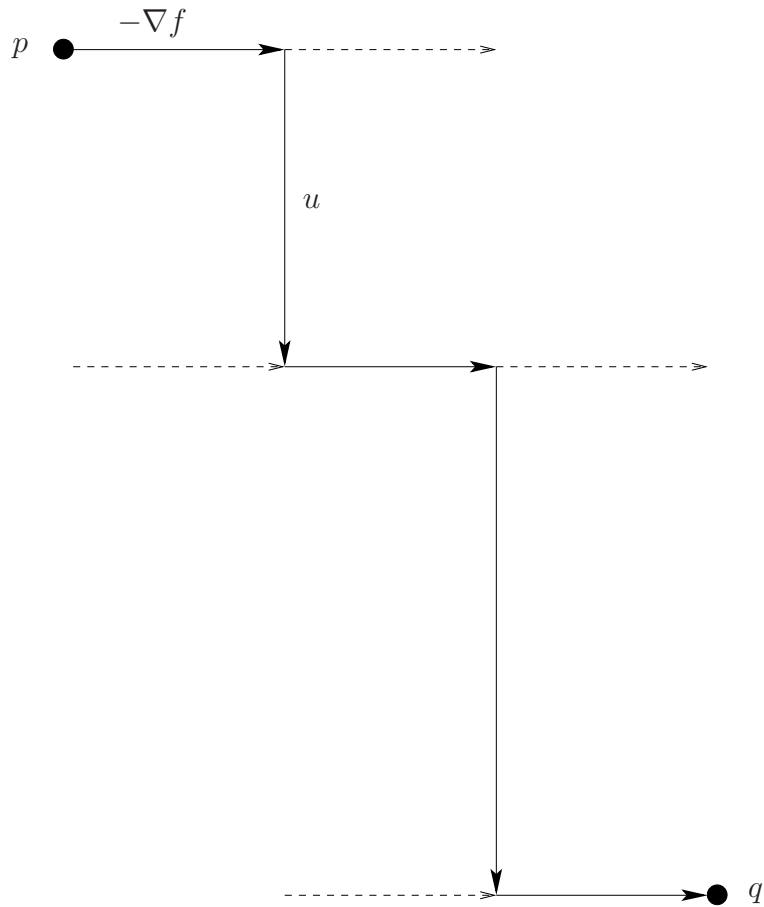}
  \caption{Trajectory with cascades.}
\end{figure}


\subsubsection{The Morse-Bott differential and the Correspondence Theorem}

Let $p \in \crit (f_\ga)$, then $\ga_p \in \P(H_\de)$ for all $\de \in \,]0,\de_0]$ if $\de_0$ is small enough. Consider
$$
\M_{]0,\de_0[}(\oga_p, \uga_q; H, \{f_\ga\}, \J):= \bigcup_{0<\de<\de_0} \{\de\} \times \M(\oga_p, \uga_q; H_\de,  \J),
$$
with 
$$
\mu(\oga_p)-\mu(\uga_q) = 1,
$$
where
$$
\oga, \uga \in \P^a(H), \quad p\in \crit(f_\oga), \quad q\in \crit(f_\uga), \quad \J \in \Js.
$$

\begin{thm}{\it (Correspondence Theorem).}
Let $H \in \H'_6 (\gg)$, $\J \in \Jr(H)$ and $\{f_\ga\} \in \Fr(H, \J)$. There exists
$$
\de_1 := \de_1 (H,\J) \in \,]0,\de_0[
$$
such that for any
$$\oga, \uga \in \P^a(H), \quad p\in \crit(f_\oga), \quad q\in \crit(f_\uga)$$
with
$$
\mu(\oga_p) - \mu(\uga_q) = 1,
$$
the following hold:
\begin{itemize}
\item[(i)] $\J$ is regular for $\M(\oga_p, \uga_q; H_\de,  \J)$ for all $\de\in ]0,\de_1[$; 
\item[(ii)] the space $\M_{]0,\de_1[}(\oga_p, \uga_q; H, \{f_\ga\}, \J)$ is a one-dimensional manifold having a finite number of components that are graphs over $]0,\de_1[$;
\item[(iii)] there is a bijective correspondence between points
$$[u] \in \M(\oga_p, \uga_q; H, \{f_\ga\}, \J)$$
and connected components of $\M_{]0,\de_1[}(\oga_p, \uga_q; H, \{f_\ga\}, \J)$.
\end{itemize}
\end{thm}

\ni
For each $[u] \in \Modpq$, the sign $\eps([u_\de])$, defined in section \ref{floer boundary operator}, is constant on the corresponding connected component $C_{[u]}$ for continuity reasons. We define a sign $\overline{\eps} ([u])$ by
\begin{equation}
\label{e:morse-bott sign}
\overline{\eps}([u]) := \eps([u_\de]), \quad \de\in ]0,\de_1[, \, (\de, [u_\de]) \in C_{[u]}.
\end{equation}

\bigskip 
\ni  
We define the Morse-Bott differential
$$
\partial : BC^a_* (H; \F_p) \ra BC^a_{*-1} (H; \F_p)
$$
by 
\begin{equation}
\label{e:mb differential}
\partial\ga_p := \sum_{\substack{\uga \in \P^a(H), \,q\in\crit (f_\uga),\\  \mu(\ga_p) - \mu(\ga_q)=1}}\, \sum_{[u]\in\M (\ga_p, \uga_q ; H, \{f_\ga \}, \J)} \overline{\eps} (u) \uga_q\, , \quad p\in \crit(f_\ga).
\end{equation}

\ni
For $\delta$ sufficiently small, the definitions imply an isomorphism of free modules
\begin{equation}
\label{e:morsebott gp isom}
CF^a_*(H_\delta; \F_p) \simeq BC^a_* (H; \F_p).
\end{equation}
Moreover the Correspondence Theorem and definition \eqref{e:morse-bott sign} of signs in the Morse-Bott complex implies that the corresponding differentials also coincide. As a consequence, it holds that
$$
H_*(BC_*^a (H; \F_p), \partial) = HF_*^a(H_\de,\J).
$$

 \clearemptydoublepage

\chapter{Proofs of theorem A and theorem B}

\section{Proof of theorem A}

\begin{thm1}
Let $M$ be a closed, connected, orientable, smooth manifold and let $\Si \subset T^*M$ be a fiberwise starshaped hypersurface. Let $\vp_R$ be the Reeb flow on $\Si$, and let $N_R$, $n_R$, $E(M)$ and $e(M)$ be defined as in section \ref{section:free loop space}. Then

\begin{itemize}
\item[(i)]
$N_R \geq E(M)$;
\item[(ii)]
$n_R \geq e(M) - 1$.
\end{itemize}

\end{thm1}

\begin{proof}
The starting point of the proof is the crude estimate 
$$
\sum_{k\geq 0} \dim \io_k H_k (\La^a ; \F_p) \geq \dim \io_0 H_0(\La^a ;\F_p).
$$
Denote by $\cc(M)$ the set of conjugacy classes in $\pi_1(M)$. Then
$$
\La M \,=\, \coprod_{\al\,\in\,\cc (M)} \La_\al M .
$$
For each element $\al \in \cc(M)$ denote by $e(\al)$ the infimum of the energy of a closed curve representing $\al$. Consider the energy sublevels
$$
\cc^a (M) := \left\{ \al \in \cc(M) \mid e(\al) \leq a \right\}.
$$
Then
$$
\dim \io_0 H_0(\La^a ;\F_p) = \# \cc^a (M).
$$

\ni
Consider the Hamiltonian function $F:\Tm \ra\re$ defined by
$$
F|_\Si \equiv 1, \quad F(q,sp)=s^2F(q,p), \quad s\geq 0 \mbox{ and }(q,p)\in \Tm.
$$
Use the cut-off function $f:\re\ra\re$ defined in section \ref{section:relevant hamiltonians} to obtain the smooth function $f\circ F$.\\

\ni
Fix $n\in \N$. Choose a sequence of perturbations 
$$
\left(V_{n,l} (t,q)\right)_{l\in \N}
$$
such that all 1-periodic solutions of $\dot{x}(t) = X_{n(f\circ F)+V_{n,l}} $ are nondegenerate, for all $l\in \N$, and 
$$
\| V_{n,l}(t,q) \|_{C^\infty} \, \ra \, 0 \mbox{ for } l\ra \infty.
$$
In order to apply the results of chapter \ref{chapter:convex to starshaped} we can, without loss of generality, suppose that $\| V_{n,l} (t,q) \|_{C^1}=: c_l < c < \frac{1}{4}$ for all $l \in \N$, where $c$ is the constant defined in section \ref{section:relevant hamiltonians}. In the following, we will use the notation
$$
F_{n,l} := n (f\circ F) + V_{n,l}.
$$

\ni
For each $l\in \N$ and an action level $a_{n,l} >n-2c_l$, we estimate
\begin{equation}
\begin{aligned}
\#\P^{a_{n,l}}(F_{n,l}) &= \dim CF^{a_{n,l}}_*(F_{n,l}; \F_p) \\
 & \geq \dim HF^{a_{n,l}}_* (F_{n,l};\F_p) \\
 & \geq \dim \io_0 H_0(\La^{na_{n,l}} ;\F_p) \\
 & = \# \cc^{na_{n,l}} (M),
\end{aligned}
\end{equation}
where the second inequality follows from Theorem \ref{t:exp growth}.
Moreover, if $\al\in\cc^{na_{n,l}}(M)$, then the set
$$
\P^{a_{n,l}} (F_{n,l};\al) := \{x =(q,p) \in \P^{a_{n,l}}(F_{n,l}) \mid q \in \al\}
$$
is not empty. This follows from the Abbondandolo--Schwarz chain isomorphisms
$$CF^{a_{n,l}}_k (\gg;\F_p) \xleftarrow{\theta^{a_n}_k} CM^{a_{n,l}}_k (L;\F_p)$$
between the Floer and the Morse chain complexes.\\

\ni
Fix $a_n$ such that $a_n \in  [a_{n,l}, n+1]$ for all $l\in \N$, and fix a non-zero conjugacy class $\al\in \cc^{na_n} (M)$. Consider a sequence $(x_l)_{l\in \N}$ such that $x_l \in \P^{a_{n,l}}(F_{n,l}; \al)$. By Proposition~\ref{action2}, 
$$
x_l \subset D(\Si) \quad \mbox{for all }l \in \N.
$$
The set $D(\Si)$ is compact. Since $V_{n,l} \,\ra\, 0$ in $C^\infty$, $X_{F_{n,l}} \ra X_{F_n}$ in $C^\infty$ on $D(\Si)$. By the Ascoli--Arzel\`a theorem, there exists a subsequence $x_{l_j}$ of $x_l$ such that $x_{l_j} \ra x$
as $j\ra\infty$, with convergence in $C^0$. Then, by making use of the Hamilton equation
$$
\dot{x}_{l_j} (t) = X_{F_{n,l_j}} (x_{l_j}(t)),
$$
we obtain $x_{l_j} \ra x$ when $j\ra\infty$ in $C^\infty$, and $x\in \P^{a_n}(F_n)$. Since $[x_{l_j}] = c$ for all $l_j$ and $x_{l_j} \ra x$, we see that $[x] =\al$. Hence 
$$
x\in \P^{a_n}(F_n;\al) \quad \mbox{and} \quad x\in D(\Si).
$$
We have thus shown that
$$\# \P^{a_n} (nF) \geq \# \cc^{n a_n} (M) - 1$$
as $\al$ is a non-trivial conjugacy class.\\

\ni
For a closed orbit $\ga$, denote by $\ga^k$ its $k$th iterate $\ga^k(t)= \ga(kt)$. The iterates of a closed orbit are geometrically the same. We want to control the contribution of the iterates of a closed orbit to $\P^{a_n}(nF)$. We are considering orbits with action $\A_{nF} (\ga) \leq a_n$, so there are finitely many of them. There exists a lower bound $\aa>0$ of the action of these closed orbits. We have 
$$
\A_{nF}(\ga^k) = k\aa.
$$
Hence at most $a_n /\aa$ iterates of such orbits will have action $\leq a_n$. Notice that if we increase the action level, $\aa$ will still be a lower bound.\\

\ni
Denote by $\tilde{\P}(nF)$ the set of geometrically different closed 1-periodic orbits. Then
$$
\# \tilde{\P}^{a_n} (nF) \geq \frac{\aa}{a_n} \# \cc^{n a_n} (M) -1
$$
which implies the estimates $(i)$ and $(ii)$ of Theorem A.
\end{proof}
\newpage

\section{Proof of theorem B}

\begin{thm2}
Let $M$ be a closed, connected, orientable smooth manifold and let $\Si \subset \Tm$ be a generic fiberwise starshaped hypersurface. Let $\vp_R$ be the Reeb flow on $\Si$, and let $N_R$, $n_R$, $C(M,g)$ and $c(M,g)$ be defined as in section \ref{section:free loop space}. Then

\begin{itemize}
\item[(i)]
$N_R \geq C(M,g)$.

\item[(ii)]
$n_R \geq c(M,g)-1$.
\end{itemize}
\end{thm2}

\subsubsection{Relevant Hamiltonians}

Let $\Si \subset \Tm$ be a fiberwise starshaped hypersurface, generic in the sense of (A), i.e. all closed Reeb orbits of $\Si$ are of Morse-Bott type.\\

\ni
Consider the Hamiltonian function $F:\Tm \ra\re$ defined by
$$
F|_\Si \equiv 1, \quad F(q,sp)=s^2F(q,p), \quad s\geq 0 \mbox{ and }(q,p)\in \Tm.
$$
Use the cut-off function $f:\re\ra\re$ defined in section \ref{section:relevant hamiltonians} to obtain the smooth function~$f\circ F$.\\

\ni
The Hamiltonian vectore field $X_F$ is a nonvanishing smooth vector field on $\Si$ which is compact. Thus there exists $\tau>0$ such that every closed orbit $x$ of $X_F$ restricted to $\Si$ has period at least $\tau$. Consider a $T$-periodic orbit $x(t) =(q(t), p(t))$ on $\Si$ and define for $s>0$
$$
x_s(t) := (q(st), sp(st)),
$$ 
then $x_s$ is a periodic orbit of $X_F$ of period $\frac{T}{s}$. Moreover all the nonconstant periodic orbits of $X_F$ are of this form. Recall the definition of the cut-off function $f:\re\ra\re$,
$$\begin{cases}
f(r) = 0 & \mbox{if } r\leq \ep^2,\\
f(r) = r  & \mbox{if } r\geq \ep  \\
f'(r) >0 & \mbox{if } r > \ep^2, \\
0\leq f'(r)\leq2 & \mbox{for all } r,
\end{cases}$$
and choose $\ep$ such that $\ep < \frac{\tau^2}{4}$. Then the closed orbits $x$ of $X_{f\circ F}$ with $f\circ F (x) \geq \ep$ agree with those of $X_F$, while the nonconstant periodic orbits $\tilde{x}$ of $X_{f\circ F}$ with $f\circ F (\tilde{x}) < \ep$ have period at least 
$$
\frac{1}{f'(F (\tilde{x}))} \cdot  \frac{\tau}{\sqrt{\ep}} >1.
$$
Hence the elements of $\P (f\circ F)$ are the same as the elements of $\P(F)$. \\

\ni
Fix $n\in \N$. In order to apply the results of chapter~\ref{chapter:convex to starshaped} and chapter~\ref{chapter:morse-bott homology}, we need to determine a suitable perturbation $V_n$ of $n(f\circ F)$. It will be the sum of a pertubation $h_{n,\delta}$ following section~\ref{section:an additional perturbation} with support in $\Tm \backslash \Dcirc(\ep - d)$ and a positive Morse function $U: \Tm \ra\re$ with support in $D(\ep^2+d)$, where $d< \frac{\ep-\ep^2}{3}$. \\

\ni
By the assumption of genericity, all the nonconstant elements of $\P(n(f\circ F))$ are transversally nondegenerate. Thus every nonconstant orbit $\ga \in \P(n(f\circ F))$ gives rise to a whole circle of nonconstant 1-periodic orbits of $X_{n(f\circ F)}$ whose parametrizations differ by a shift $t\in S^1$. We can thus apply the construction of the perturbation described in section \ref{section:an additional perturbation} and consider the perturbed Hamiltonian
$$
\Fnd (t,q,p) := n(f\circ F(q,p)) + \delta \rho(\psi(q,p)) \fg (\psi_1(q,p)-\lga t).
$$
We will denote the perturbation added to $n(f\circ F)$ by $h_{n,\delta}$ and assume that the support of $h_{n,\delta}$ is contained in $\Tm \backslash \Dcirc(\ep-d)$, where $d< \frac{\ep-\ep^2}{3}$.\\

\ni
Recall that this perturbation destroys the critical circles and gives rise to two solutions of $\dot{x} = X_\Fnd (t,x)$, namely 
$\gam(t)$ and $\gaM(t)$. \\

\ni
Let $U: \Tm \ra\re$ be a positive Morse function with support in $D(\ep^2+d)$. We define the perturbation $V_n: S^1 \times \Tm \ra\re$ by 
\begin{equation}
\label{e:perturbation morse}
V_n (t,q,p) =  h_{n,\delta} (t,q,p) + U(q,p).
\end{equation}  
Without loss of generality, we can choose $\delta$ and $U$ so that 
$$
\| V_n (t,q,p) \|_{C^1} <c <\frac{1}{4},
$$
where $c$ is the constant defined in section \ref{section:relevant hamiltonians}.\\

\ni
The 1-periodic orbits of $X_{n(f\circ F) + V_n}$ fall in two classes:
\begin{enumerate}
\item critical points of $n(f\circ F) + U$ in $D(\ep^2 +d)$,
\item nonconstant 1-periodic orbits of $n(f\circ F) +  h_{n,\delta}$ in $\Tm \backslash \Dcirc(\ep-d)$.
\end{enumerate}

\ni
Set 
$$\begin{aligned}
\gg(t,q,p) & := n \si (f\circ G)(q,p) +W_n (t,q) + U(q,p), \\
K_n(t,q,p) & := (1-\tau_n(|p|) \left(n (f\circ F) + V_n\right) (t,q,p)+\tau_n(|p|)\;  \gg(t,q,p), \\
\gl(t,q,p) & := (1-\tau_n(|p|) \left(n (f\circ G) + W_n \right)(t,q,p) + U(q,p) +\tau_n(|p|)\;  \gg(t, q,p) \\
 	      & \phantom{i}= n \bigl( (1-\tau_n(|p|)  (f\circ G )(q,p)+\tau_n(|p|)\;  \si G(q,p) \bigr) + W_n(t,q) + U(q,p),
\end{aligned}$$
where $\tau$ and $W_n$ are defined as in section~\ref{section:relevant hamiltonians}. Moreover we assume that 
\begin{itemize}
\item[(B1)] $\ep^2 < \frac{\tilde{\tau}^2}{4}$ , where $\tilde{\tau}$ is the minimal period of closed orbits of $X_F$, $X_G$ and $X_{\si G}$ restricted to $\Si$.
\item[(B2)] $f'(\ep - d) > \demi$.
\item[(B3)] $W_n$ as support in $\Tm \backslash \Dcirc(\ep - d)$.    
\item[(B4)] $c_n < \frac{n}{2} (\ep-d)^2$. 
\end{itemize}
Assumptions (B1) and (B3) imply that the 1-periodic orbits of $X_{\gg}$ fall in two classes:
\begin{enumerate}
\item critical points of $n\si (f\circ G) + U$ in $D(\ep^2 +d)$,
\item nonconstant 1-periodic orbits of $n\si (f\circ G) + W_n$ in $\Tm \backslash \Dcirc(\ep-d)$
\end{enumerate}
and similarly for $\gl$. Assumptions (B2) and (B4) will be used in the next section.

\subsubsection{The Non-crossing Lemma}

\ni
Following section~\ref{section:non-crossing lemma}, 
for $s\in [0,1]$ define the functions
\begin{equation}
\gs  =(1-\be(s)) \gl + \be(s) \gg 
\end{equation}
and
$$
a(s)=\frac{a}{1+\be(s)(\si-1)}
$$
where $a\in\, ]n-2c,(n+1)-2c[$.

\begin{lem}
If $[a-c_n , a+c_n]\cap \Sp(\gl) = \emptyset$, then $a(s)\notin \Sp(\gs)$ for $s\in[0,1]$. Moreover $0\notin \Sp(\gs)$ for $s\in[0,1]$.
\label{crossing2}
\end{lem}

\begin{proof}
Take $\ga = (q(t),p(t)) \in \P(\gs)$.\\
Assume first that $|p(t_0)| > 3$ for some $t_0 \in S^1$, then $U(\ga)=0$. Then the statement follows from the proof of Lemma~\ref{crossing}.\\
Assume next that $\ga \in D(3)\backslash \Dcirc(\ep-d)$. Then $U(\ga)=0$ and $\tau_n (|p(t)|) =0$ for all $t\in [0,1]$, which yields $\gl (\ga) = n(f\circ G) (\ga) + W_n(\ga)$ and
$$\begin{aligned}
\gs(\ga) & = n\bigl( 1-\be(s) \bigr) \bigl( f\circ G \bigr) (\ga) + n \be(s) \si (f \circ G) (\ga) +W_n(\ga)\\
              & = n\Bigl( \bigl ((1-\be(s) \bigr) + \be(s) \si )  f\circ  \Bigr) G(\ga) + W_n(\ga).
\end{aligned}$$
Hence the first part of the statement follows from the proof of Lemma~\ref{crossing}. We need to show that $\A_{\gs} (\ga) \neq 0$. By Lemma~\ref{action} (i) we have
$$
\A_{\gs} (\ga) = \int_0^1 2n  \Bigl( \bigl ((1-\be(s)) + \be(s) \si \bigr)  f' \bigl(G(\ga) \bigr) \Bigr) G(\ga) - \gs (\ga) \,dt.
$$
Together with $f'(\ep-d) > \demi$ and the choice of $c_n$ we obtain that
$$\begin{aligned}
\A_{\gs} (\ga) & \geq \int_0^1 \bigl( 2n f'(\ep-d) -n \bigr) G(\ga) - W_n (\ga)) \, dt \\
                     & \geq \int_0^1 \frac{n}{2} (\ep-d)^2 - c_n \,dt \\
                     & > 0. 
\end{aligned}$$
Assume finally that $\ga$ lies in $\Dcirc (\ep-d)$. Then $\ga(t) = \ga(0)$ is a constant orbit, $W_n(\ga) =0$, $\tau(\ga) = 0$ and $U(\ga) >0$. Thus
$$\begin{aligned}
\A_{\gs} (\ga) &= \int_0^1 - n  \bigl ( (1-\be(s) ) + \be(s) \si \bigr) \; f(G(\ga)) - U(\ga) \, dt \\
                      &\leq -U(\ga) < 0.  
\end{aligned}$$
This conclude the proof of the lemma.                      
\end{proof}

\subsubsection{Floer homology groups}

For $H \in \H_6(\gg)$ and $a< (n+1)-2c$ we define
$$
\P^{(0,a]} (H) := \{x\in \P(H) \mid 0< \A_H(x) \leq a\} = \P^a (H) \backslash \P^0 (H).
$$
Suppose that $0, a \notin \Sp(H)$ and every 1-periodic orbits $x\in \P^{(0,a]}(H)$ is nondegenerate. The Floer chain group $CF^{(0,a]}_*(H;\F_p)$ is defined as the quotient 
$$
CF^{(0,a]}_*(H;\F_p) := CF^a_*(H;\F_p) / CF^0_*(H;\F_p).
$$
It is usefull to think of this chain complex as the $\F_p$-vector space freely generated by the elements of $\P^ {(0,a]}(H)$ graded by the Maslov index. By Lemma~\ref{l:decreasing action} the subcomplex $CF^0_*(H;\F_p)$ is invariant under the boundary operator defined in section~\ref{floer boundary operator}. We thus get an induced boundary operator  on the quotient. We denote the $k^{th}$ Floer homology group of the quotient complex by 
$$
HF^{(0,a]}_k(H;\F_p) := \frac{\ker \partial_k(\J) :  CF^{(0,a]}_k (H;\F_p)\ra CF^{(0,a]}_{k-1} (H;\F_p) }{\im \partial_{k+1}(\J) : CF^{(0,a]}_{k+1} (H;\F_p)\ra CF^{(0,a]}_k (H;\F_p)}.
$$

\subsubsection{Continuation homomorphisms}

Let $a_n$ be the action level \eqref{def:a} of section~\ref{section:Continuation homomorphisms2}. 

\begin{prop}
\label{p:diagram2}
$\dim HF^{(0,a_n]}_k (K_n) \geq \rank \left( HF_k (\io): HF^{(0,a_n / \si]}_k (\gg) \ra HF^{(0,a_n]}_k (\gg) \right) $. 
\end{prop} 

\begin{proof}
In view of the Non-crossing Lemma~\ref{crossing2}, the isomomorphism $
\widehat{\Phi}_{\gg \gl} : HF_*^{a_n}(\gl) \ra HF_*^{a_n/\si}(\gg)
$ of section~\ref{section:Continuation homomorphisms2} induces an isomorphism 
$$
\widetilde{\Phi}_{\gg \gl} : HF_*^{(0,a_n]}(\gl) \ra HF_*^{(0,a_n / \si]}(\gg). 
$$ 
Recalling the proof of Proposition~\ref{p:diagram}, we thus have for each $k$ a commutative diagram of homomorphisms
\begin{equation}
\xymatrix{
  & HF^{(0,a_n / \si]}_k(\gg) \ar[rd]^{HF_k (\io_n)} & \\
  HF^{(0,a_n]}_k(\gl) \ar[ru]^{\widetilde{\Phi}_{\gg \gl}} \ar[rr]^{\Phi_{\gg \gl}} \ar[rd]_{\Phi_{K_n \gl}} & & HF^{(0,a_n]}_k(\gg)   \\
&  HF^{(0,a_n]}_k(K_n) \ar[ru]_{\Phi_{\gg K_n}} & }
\label{diagram2}
\end{equation}
Moreover $\widetilde{\Phi}_{\gg \gl}$ is an isomorphism. This concludes the proof of the proposition.
\end{proof}

\subsubsection{To the homology of the free loop space}

\begin{prop}
\label{p:lower bound}
Let $(M,g)$ be a closed, orientable Riemannian manifold and $K_n$ be as above. It holds that for $k\geq 1$
$$\dim HF^{(0,a_n]}_k (K_n; \F_p) \geq \rank \io_k H_k (\La^{n a_n} ; \F_p).$$
\end{prop}

\begin{proof}
Let $L:S^1\times TM \ra \re$ be the Legendre transform of $\gg$, let
$$\E_L (q) := \int_0^1 L(t,q(t), \dot{q}(t))\,dt$$
be the corresponding functional on $\La M$, and let 
$$\La^{(0,b]}_L := \La^b_L \backslash \La^0_L , \quad \La^b_L:=\{ q \in \La M \mid \E_L (q) \leq b\}.$$
Recall that both the Abbondandolo--Schwarz isomorphism, \cite{AS06}, and Abbondandolo--Mayer isomorphism, \cite{AM06}, are chain complex isomorphisms. Following the proof of Proposition~\ref{p:commut}, we have for each $k$ a commutative diagram of homomorphisms
$$
\xymatrix{
HF^{(0,a_n / \si]}_k(\gg;\F_p) \ar[r]^\cong \ar[d]_{HF_k (\io)} & H_k(\La^{(0,n a_n] } ;\F_p) \ar[d]^{H_k(\io)} \\
HF^{(0,a_n]}_k( \gg;\F_p) \ar[r]^\cong &  H_k(\La^{(0,n\si a_n]} ;\F_p)  } 
$$
where the horizontal maps are isomorphisms and the right map $H_k(\io)$ is induced by the inclusion $\La^{(0,n a_n] } \hookrightarrow \La^{(0,n\si a_n]}$.\\
By definition $\La^0$ is the set of constant loops. Hence every critical point $q \in \La^0$ of $\E$ has Morse index zero. This yields for $k\geq 1$
$$
H_k(\La^{(0, b] } ;\F_p) = H_k(\La^{b } ;\F_p).
$$
Consider now the commutative diagram 
\begin{equation} 
\xymatrix{
H_k(\La^{na_n};\F_p) \ar[d] \ar[dr]^{\io_k} \\
H_k(\La^{n\si a_n};\F_p) \ar[r] & H_k(\La M;\F_p)   }
\end{equation}
induced by the inclusion $\La^{na_n} \subset \La^{n\si a_n} \subset \La M$. Together with Proposition~\ref{p:diagram2} this yields for $k\geq 1$
$$
\dim HF^{(0,a_n]}_k (K_n; \F_p) \geq \rank \left( \io_k: H_k(\La^{na_n};\F_p) \ra H_k(\La M;\F_p)\right),
$$
which concludes the proof of Proposition~\ref{p:lower bound}.
\end{proof}

\medskip

\ni
We can now proove Theorem B.

\begin{proof}
Let $\Si \subset \Tm$ be a fiberwise starshaped hypersurface, generic in the sense of~(A), i.e. all closed Reeb orbits of $\Si$ are of Morse-Bott type.\\

\ni
Consider the Hamiltonian function $F:\Tm \ra\re$ defined by
$$
F|_\Si \equiv 1, \quad F(q,sp)=s^2F(q,p), \quad s\geq 0 \mbox{ and }(q,p)\in \Tm.
$$
We use the cut-off function $f:\re\ra\re$ and the positive Morse function $U$ defined previously to obtain the smooth function $f\circ F + U$.\\

\ni
By the assumption of genericity, we can consider the perturbed Hamiltonian
$$
\Fnd (t,q,p) := n(f\circ F(q,p)) + V_n ,
$$
where $V_n$ is given by \eqref{e:perturbation morse}, i.e. the sum of the Morse--Bott pertubation $h_{n,\delta}$ and the Morse function $U$. Without loss of generality, we can choose $\delta$ and $U$ so that 
$$
\| V_n (t,q,p) \|_{C^1} <c <\frac{1}{4},
$$
where $c$ is the constant defined in section \ref{section:relevant hamiltonians} and such that the isomorphism \eqref{e:morsebott gp isom} between the Morse-Bott complex $BC^{(0,a_n]}_* (n(f\circ F) +U;\F_p)$ and the Floer complex $CF^{(0,a_n]}_*(\Fnd; \F_p)$ holds.\\

\ni
For an action level $a_{n} >n - 2c$ we have that 
\begin{equation}
\label{e:ineqbz}
\begin{aligned}
\#\P^{(0,a_n]}(n(f\circ F) +U) &= \demi \dim BC^{(0,a_n]}_* (n(f\circ F)+U;\F_p) \\
&= \demi \dim CF^{(0,a_n]}_*(\Fnd; \F_p) \\
 & \geq \demi \dim HF^{(0,a_n]}_* (\Fnd;\F_p).
\end{aligned}
\end{equation} 
Together with Proposition \ref{p:lower bound} this yields 
$$
\#\P^{(0,a_n]}(n(f\circ F) +U) \geq \demi \sum_{k\geq 1} \rank (\io_k: H_k(\La^{na_n}; \F_p) \ra H_k (\La M; \F_p)).
$$

\ni
Suppose $(M,g)$ is energy hyperbolic, $h:= C(M,g)>0$. By definition of $C(M,g)$, there exist $p\in \Pr$ and $N_0 \in \N$ such that for all $n\geq N_0$,
$$
\sum_{k\geq 0} \dim \io_k H_k (\La^{\demi n^2}; \F_p) \geq e^{\frac{1}{\sqrt{2}}hn}.
$$ 
Therefore there exists $N\in \N$ such that for all $n\geq N$, 
$$
\sum_{k\geq1} \rank \io_k \geq e^{hn} - m,
$$
where $m$ is the dimension of the base space $M$. Together with Proposition \ref{p:commut} we find that 
$$
\sum_{k\geq1} \rank HF_k(\io) \geq \sum_{k\geq1} \rank \io_k \geq e^{hn} - m.
$$

\ni
Similarly, if $c(M,g)>0$ we find that there exists $p\in \Pr$ and $N\in \N$ such that for all $n\geq N$,
$$
\sum_{k\geq1} \rank HF_k(\io) \geq \sum_{k\geq1} \rank \io_k \geq n^{c(M,g)} - m.
$$

\ni
We again need to take care of the iterates $\ga^k$ of an element $\ga \in \P^{(0,a_n]}(n(f\circ F) +U)$. Let $\aa>0$ be the minimal action of the elements of $\P^{(0,a_n]}(n(f\circ F) +U)$ . We have 
$$
\A_{n(f\circ F) + U}(\ga^k) \geq k\aa.
$$
Hence at most $a_n /\aa$ iterates of such orbits will have action less or equal to $a_n$.\\

\ni
Denote by $\tilde{\P}^{a_n} (n F)$ the set of geometrically different, non-constant, closed 1-periodic orbits of $X_{nF}$. We get for an energy hyperbolic manifold 
$$
\#\tilde{\P}^{a_n}(n F) \geq \frac{\aa}{2a_n} (e^{C(M,g) a_n} -m),
$$
and if $c(M,g)>0$
$$
\#\tilde{\P}^{a_n}(n F) \geq \frac{\aa}{2a_n} (a_n^{c(M,g)} - m),
$$
which yields Theorem~B in view of the definitions.
\end{proof}

\subsection{The simply connected case: generalizing Ballmann--Ziller}
\label{section:the simply connected case}
In \cite{BZ82}, Ballmann and Ziller proved that the number of closed geodesics of any bumpy Riemannian metric on a compact, simply connected manifold $M$ grows like the maximal Betti number of the free loop space. In this section, we will prove a similar result for the number of closed Reeb orbits on $\Si$.  \\

\ni
Let $\La M$ be the free loop space of $M$ and denote by $\La^a$ the sublevel set of closed loops of energy $\leq a$. Denote by $b_i(\La M)$ the rank of $H_i (\La M; \F_p)$.

\begin{thm3}
Suppose that $M$ is a compact and simply connected $m$-dimensional manifold. Let $\Si$ be a generic fiberwise starshaped hypersurface of $T^*M$ and $R$ its associated Reeb vector field. Then there exist constants $\al = \al(R) >0$ and $\be = \be(R)>0$ such that
$$
\# \Od_R (\tau) \geq \al \max_{1 \leq i\leq \be \tau} b_i (\La M)
$$
for all $\tau$ sufficiently large.
\end{thm3}

\begin{proof}
Recall the following result of Gromov
\begin{thm}{\rm(Gromov).}
\label{thm:gromov}
There exists a constant $\kappa = \kappa (g) >0$, such that $H_i(\La^{\kappa t^2}; \F_p) \ra H_i (\La M; \F_p)$ is surjective for $i\leq t$.
\end{thm}
\ni
A proof can be found in Appendix~\ref{appendix:gromov's work}. Renormalize $g$ as in section~\ref{section:relevant hamiltonians}, i.e. such that $\Si \subset \{G^{-1}(1)\}$ where $G(q,p) := \demi g^*(q)(p,p)$. Then there exists $\be := \be(R)$ such that $b_i (\La^{n^2}) = \rank H_i(\La^{n^2}; \F_p) \geq b_i (\La M)$ for $i \leq \be n$.\\

\ni
Let $H: \Tm \ra \re$ be a fiberwise homogeneous Hamiltonian of degree 2 such that $H|_\Si = 1$. Consider its Morse--Bott perturbation $H_\delta$. Denote by $b_i (nH)$ the rank of $HF^{(0,n]}_i(nH_\de; \F_p)$. We already proved that 
$$
\# \P^{(0,n]} (nH) = \demi \dim CF^{(0,n]}_*(nH_\de ; \F_p).
$$

\ni
For a non-constant closed orbit $\ga$, denote by $\ga^k$ its $k$th iterate $\ga^k (t) = \ga(kt)$. The iterates of a closed 1-periodic orbit give rise to different critical circles of $\A_{nH}$ of index $i$ or $i-1$. Following \cite{BZ82} we are going to control the contribution of a closed 1-periodic orbit and its iterates to $b_i (nH)$.\\

\ni
By construction of $CF^{(0,n]}_* (nH_\de; \F_p)$, a non-constant element $\ga \in \P^{(0,n]} (nH)$ gives rise to two generators $\gam$ and $\gaM$ with their Maslov index related in the following way
$$
\mu (\gam) = \mu (\ga) - \demi
$$
and
$$
\mu(\gaM)= \mu (\ga) + \demi.
$$
Moreover, by construction of $\gam$, 
$$(\gam)^k (t) = \gam (kt) = \ga (kt + t_1) = \ga^k (t+t_1) = (\ga^k)_{\min} (t),
$$
and similarly $(\gaM)^k = (\ga^k)_{\max}$.\\

\ni
The index iteration formula of Salamon and Zehnder, see Lemma \ref{lemma:iteration formula}, implies the following 

\begin{lem}
Let $\ga$ be a nondegenerate closed orbit. Then there exist constants $\al_\ga$ and $\be_\ga$, such that 
\begin{equation}
k \al_\ga - \be_\ga \leq \mu(\ga^k) \leq k\al_\ga + \be_\ga.
\label{e:iteration formula}
\end{equation}
Moreover, $| \be_\ga | < m:= \dim M$.
\end{lem}

\ni
Notice that if $\al_\ga = 0$, then $|\mu (\ga)|<m$. We consider three cases\\

1. Assume first that $\A_{nH} (\ga) \leq 6m/ \be(R)$. There exist only finitely many such closed orbits. Let $\al>0$ be a lower bound of the average index $\al_\ga$ of the closed orbits $\ga$ for which $\al_\ga >0$. Using \eqref{e:iteration formula} we get
$$
\mu(\ga^{k+l}) - \mu(\ga^k) \geq l \al_\ga - 2\be_\ga \geq l \al - 2m.
$$
Hence $\mu(\ga^{k+l}) >\mu(\ga^k)$, if $l > 2m/\al$. Therefore at most $N_1 := 4m/\al$ iterates of $\ga$ can have index $i$ or $i-1$ for $i>m$. \\

2. Assume next that $\A_{nH} (\ga) \geq 6m/\be(R)$ and $\mu(\ga) > m$. In view of \eqref{e:iteration formula} we have $\al_\ga >1$. Thus
$$
\mu(\ga^{k+l}) - \mu(\ga^k) \geq l \al_\ga - 2\be_\ga \geq l - 2m.
$$
Hence at most $N_2 := 4m$ iterates of $\ga$ can have index $i$ or $i-1$.\\

3. Finally assume that $\A_{nH} (\ga) \geq 6m / \be(R)$ and $\mu(\ga) \leq m$. In view of \eqref{e:iteration formula} and $\mu(\ga) \leq m$ we have $\al_\ga \leq 2m$. This and again \eqref{e:iteration formula} yields
\begin{equation}
\mu (\ga^k) \leq k \al_\ga + \be_\ga \leq k (2m) + m < 3 k m \quad \mbox{for } k >1.
\label{e:muk}
\end{equation}
If $k> \be(R) n/6m$, then, using the assumption $\A_{nH} (\ga) \geq 6m / \be(R)$,
$$
\A_{nH} (\ga^k) = k \A_{nH}(\ga) \geq k 6m / \be(R) > n.
$$
Recall now that we are counting only closed orbit below the action level $n$. Hence $k \leq \be(R) n/6m$, and then with \eqref{e:muk},
$$
\mu(\ga^k) < 3km \leq  \be(R)n/2 \quad \mbox{for } k >1.
$$   
Therefore at most $N_3:=2$ iterates of $\ga$ can have index $i$ or $i-1$ of action $\A_{nH} (\ga) \leq n$ if $i> \be(R)n/2$. \\

\ni
We obtain that for $n$ large enough there exists a constant $N_0 := \max\{N_1, N_2, N_3\}$, such that at most $N_0$ iterates of any $\ga \in \P^{(0,n]} (nH)$ will give rise to generators of index $i$ or $i-1$, for $i> \be(R)n/2$. Summarizing, if we denote by $\tilde{\P}^n (nH)$ the set of geometrically different, non-constant, closed 1-periodic orbits, we get  
\begin{equation}
\# \tilde{\P}^n(nH) \geq \demi \;\max_{\frac{\be(R)n}{2} < i <\infty} \frac{b_i (nH)}{N_0}.
\end{equation}
In view of Proposition~\ref{p:lower bound} and Gromov's result, we thus have 
$$
\# \Od_R (n) \geq \al(R) \max_{\be n/2 < i\leq \be n} b_i (\La M).
$$
Since $\# \Od_R(n) \geq \# \Od_R(n/2)$ we get
$$
\# \Od_R (n) \geq \al(R) \max_{1 \leq i\leq \be n} b_i (\La M).
$$

\end{proof}

 \clearemptydoublepage

\chapter{Evaluation}
\label{Evaluation}

In this chapter we evaluate the results of the previous chapter on the examples introduced in section~\ref{section:examples}. 

\section{Lie groups}

Let $M$ be a Lie group. As its fundamental group is abelian, the free loop space $\La M$ admits the decomposition 
$$
\La M  = \coprod_{\al \in \pi_1(M)} \La_\al M.
$$
Moreover, all the components $\La_c$ of the free loop space are homotopy equivalent. 

\begin{prop}
Let $M$ be a closed, connected, Lie group and let $\Si \subset \Tm$ be a generic, fiberwise starshaped hypersurface. Let $\vp_R$ be the Reeb flow on $\Si$, and let $N_R$, $n_R$, $C_1(M,g)$ and $c_1(M,g)$ be defined as in section \ref{section:free loop space}. Then

\begin{itemize}
\item[(i)] $N_R \geq E(M) + C_1(M,g)$
\item[(ii)] $n_R \geq e(m) + c_1(M,g)-1$.
\end{itemize}
\end{prop}

\begin{proof}
For $c \in \pi_1(M)$ let $e(c)$ be the infimum of the energy of a closed curve representing~$c$. Then
\begin{equation}
\label{e:lie1}
\cc^a = \{ c \in \pi_1(M) \mid e(c) \leq a \}.
\end{equation}
Moreover, recalling the argument in section \ref{section:lie groups} on Lie groups, we have that
\begin{equation}
\label{e:lie2}
\dim \io_k H_k (\La_1^a; \F_p) \leq \dim \io_k H_k (\La_c^{2a+2e(c)}; \F_p),
\end{equation}
where $\La_i^a = \La_i \cap \La^a$ is the sublevel set of loops in the component $\La_i$ of energy $\leq a$. Using the notation from the proof of Theorem B, it holds that 
$$
\begin{aligned}
\#\P^{2n^2}(F) & \geq \demi \sum_{k\geq 1} \dim \io_k H_k(\La^{2n^2}; \F_p) \\
		       & = \demi \sum_{k\geq 1} \bigl( \sum_{c \in \pi_1(G)} \dim \io_k H_k (\La_c^{2n^2}; \F_p) \bigr).
\end{aligned}
$$
Hence by \eqref{e:lie2}  and \eqref{e:lie1}
$$
\begin{aligned}
\phantom{\#\P^{2n^2}(F)} & \geq \demi \sum_{k\geq 1} \bigl( \sum_{c \in \pi_1(M)} \dim \io_k H_k (\La_1^{n^2-e(c)}; \F_p) \bigr) \\
		       & \geq \demi \sum_{k\geq 1} \bigl( \sum_{c \in \cc^n} \dim \io_k H_k (\La_1^{\demi n^2}; \F_p) \bigr) \\
		       & \geq \demi \#\cc^{n}(M) \cdot \sum_{k\geq 1} \dim \io_k H_k (\La_1^{\demi n^2}; \F_p).
\end{aligned}
$$
Arguing as in the proof of Theorem B, the Proposition follows in view of the definitions.

\end{proof}

\section{$\pi_1(M)$ finite: the case of the sphere}
\label{section:p1 finite}
In this section we will show that the product of two spheres $S^l\times S^n$ has $c(M,g)\geq 2$, so that $n_R \geq 1$.. Our main tool will be the cohomology classes of the free loop space discovered by Svarc and Sullivan as an application of Sullivan's theory of minimal models, see \cite{Sv60,Su75}. We begin by recalling some basics properties of the minimal model following \cite{Kli78} and \cite{FOT08}. \\

\ni
The {\it minimal model} $\M_M$ for the rational homotopy type of a simply connected countable $CW$ complex $M$ is a differential graded algebra over $\qe$ with product denoted by $\wedge$, having the following properties:
\begin{enumerate}
\item $\M_M$ is free-commutative, i.e. is free as an algebra except for the relations imposed by the associativity and graded commutativity. The vector space spanned by the generators of any given degree $k$ is finite; its dual is isomorphic to $\pi_k(M) \otimes_\Z\qe$;
\item the differential $d$ applied to any generator is either zero, or raises the degree by one and is a polynomial in generators of strictly lower degree;
\item $H^*(\M_M;\qe) = H^*(M;\qe)$
\end{enumerate}  

\ni
The rational cohomology of the sphere $S^n$ is an exterior algebra on one generator in degree $n$. A minimal model for $S^n$ is given by
$$
(\bigwedge x, 0) \quad \mbox{with } |x|=n
$$
when $n$ is odd, and
$$
(\bigwedge(x,y),d) \quad \mbox{with } |x|=n,\;  |y|=2n-1, \;dx=0 \mbox{ and } dy=x^2
$$
when $n$ is even, see \cite[Example 2.43]{FOT08}.\\

\ni
Given $\M_M$, a minimal model $\M_{\La M}$ for the free loop space $\La M$ can be constructed as follows. Each generator $x$ of $\M_M$ is also a generator of $\M_{\La M}$ with the same differential. The remaining generators are obtained by associating to each generator $x$ of $\M_M$ a generator $\overline{x}$ for $\M_{\La M}$ of one degree less. In order to define their differential, extend $\overline{\phantom{abc}}$ to all of $\M_M$ as a derivation acting from the right. Then define $d\overline{x} := -\overline{dx}$. \\

\ni
A minimal model for the free loop space of the sphere $\La S^n$ is given by
$$
(\bigwedge (x,\overline{x}),0) \quad \mbox{with } |x|=n \mbox{ and } |\overline{x}|=n-1 
$$
when $n$ is odd and 
$$
(\bigwedge(x,y,\overline{x},\overline{y}),d) \quad \mbox{with } 
\begin{aligned}
|x|&=n,\;  |y|=2n-1, \; |\overline{x}| = n-1,\; |\overline{y}|=2n-2,  \\
dx&=0,\; dy=x^2,\; d\overline{x}=0 \mbox{ and } d\overline{y}=-2x\overline{x}
\end{aligned}
$$
when $n$ is even. 

\subsubsection{$n$ odd}
When $n \geq 3$ is odd, the rational cohomology of $\La S^n$ has one generator in each degree
$$
k(n-1) \mbox{ and } k(n-1)+1, \quad \mbox{for } k \in \N.
$$
Ideed, $x^k = 0$ for $k\geq 2$. For example when $n=5$

\begin{table}[h]
\begin{tabular}{c |*{15}{c}l}
$i$                 & 0 & 1 & 2 & 3 & 4 & 5 & 6 & 7 & 8 & 9 & 10 & 11 & 12 & 13 & 14 & 15 \\
   \hline
$b_i(\La S^5)$& 1 & 0 & 0 & 0 & 1 & 1 & 0 & 0 & 1 & 1 & 0 & 0 & 1 & 1 & 0 & 0  \\
\end{tabular}
\end{table}

\ni
This yields
$$
\sum_{i=1}^k b_i(\La S^n) \geq 2 \Big\lfloor \frac{k}{n-1} \Big\rfloor.
$$
Consider Gromov's constant $\kappa := \kappa(g)$, see Appendix \ref{appendix:gromov's work}. Then it holds that
$$
\dim \io_j H_j(\La^{\kappa t^2}) = b_j(\La S^n) \quad \mbox{for } j\leq t.
$$
This yields
$$
\lim_{t\ra \infty}\frac{1}{t} \sum_{j\geq 1} \dim \io_j H_j(\La^{\kappa t^2}) \geq \frac{2}{n-1}.
$$

\subsubsection{$n$ even}
When $n$ is even, the Sullivan cohomology classes are given by
$$
w^*(s) := \overline{x}\,\overline{y}^s, \quad s\in \N.
$$
For every $s\in \N$, its differential is zero and it is not a boundary since $d\M_{\La S^n}$ is contained in the ideal generated by the subalgebra $\M_{S^n} \subset \M_{\La S^n}$. The rational cohomology of $\La S^n$ has then one generator in each degree
$$
|w^*(s)| = (1+2s) (n-1) \quad \mbox{for } s\in \N.
$$
This yields
$$
\sum_{i=1}^k b_i(\La S^n) \geq \frac{k}{2(n-1)}.
$$
And it holds that
$$
\lim_{t\ra \infty}\frac{1}{t} \sum_{k\geq 1} \dim \io_k H_k(\La^{\kappa t^2}) = \frac{1}{2(n-1)}.
$$ 

\subsubsection{Product of spheres $S^l \times S^n$}
Consider the product of two spheres $S^l\times S^n$ of dimensions $l,n \geq 2$.
\begin{lem}
There exists a constant $\al:= \al(l,n)>0$ depending on $l$ and $n$ such that 
$$
\sum_{i=1}^k b_i(\La(S^l \times S^l)) \geq \al \, k^2.
$$
\end{lem} 

\ni
Recall that for smooth manifolds $M$ and $N$,
$$
\La(M\times N) = \La M \times \La N.
$$
The Knneth formula tells us that 
$$
b_i(\La(S^l \times S^n)) = \sum_{j+k=i} b_j(\La S^l) \cdot b_k(\La S^n).
$$ 
If $l$ and $n$ are odd, we have that 
$$
\begin{cases}
b_j(\La S^l) \cdot b_k(\La S^n) \neq 0, \\
i=j+k.
\end{cases}
$$
if
$$
\begin{cases}
i =j(l-1) + k(n-1),  \\
i =j(l-1) + k(n-1)+1, \\
i =j(l-1) + k(n-1)+2.  \\
\end{cases}
$$
\ms

\ni
For example when $l=5$ and $n=7$
\begin{table}[h]
\begin{tabular}{c |*{15}{c}l}
$i$                 & 0 & 1 & 2 & 3 & 4 & 5 & 6 & 7 & 8 & 9 & 10 & 11 & 12 & 13 & 14 & 15 \\
   \hline
$b_i(\La S^5)$& 1 & 0 & 0 & 0 & 1 & 1 & 0 & 0 & 1 & 1 & 0 & 0 & 1 & 1 & 0 & 0  \\
$b_i(\La S^7)$& 1 & 0 & 0 & 0 & 0 & 0 & 1 & 1 & 0 & 0 & 0 & 0 & 1 & 1 & 0 & 0 \\
\hline
$b_i(\La (S^5\times S^7))$& 1 & 0 & 0 & 0 & 1 & 1 & 1 & 1 & 1 & 1 & 1 & 2 & 3 & 2 & 1 & 2
\end{tabular}
\end{table}

\ni
This yields
$$
b_i(\La(S^l \times S^n)) \geq s \quad\mbox{if } i=s(l-1)(n-1).
$$
Thus 
$$
\sum_{i=1}^k b_i(\La(S^l \times S^l)) \geq \frac{k^2}{(l-1)^2(n-1)^2}.
$$
Arguing similarly when $l$ or $n$ is even, we obtain the lemma.\\

\ni
Again, using Gromov's work, we obtain
$$
c(M,g) = \limsup_{t\ra \infty} \frac{1}{\log (\sqrt{2\kappa}t)} \log \sum_{k\geq 1} \dim \io_k H_k(\La^{\kappa t^2}) \geq 2.
$$
This yields
$$
n_R \geq 1.
$$
Similarly, if $M= S^{n_1}\times \ldots \times S^{n_k}$ is a product of $k$ spheres of dimensions $n_j \geq 2$, then $$n_R \geq k-1.$$ 

\section{Negative curvature}

\begin{prop}
Let $M$ be a closed connected orientable manifold endowed with a metric of negative curvature and let $\Si \subset \Tm$ be a fiberwise starshaped hypersurface. Let $\vp_R$ be the Reeb flow on $\Si$, and let $N_R$ be defined as in section \ref{section:free loop space}. Then

$$N_R \geq h_{top}(g) >0$$
where $h_{top}(g)$ denotes the topological entropy of the geodesic flow.
\end{prop}

\ni
Suppose that $M$ possesses a Riemannian metric with negative curvature. Then by Proposition~\ref{prop:negative curvature}, it holds that
$$
\La M \simeq M \coprod_{\cc(M)} S^1.
$$
Since all Morse indices vanish, 
$$
\sum_{k\geq0} \dim H_k(\La^a M;\F_p) = \dim H_0(\La^a M; \F_p).
$$
Moreover the generators appearing while increasing the energy will not kill the previous ones. Hence for $c>1$ the map 
$$
H_0(\io): H_0(\La^{a}M;\F_p) \ra H_0(\La M;\F_p)
$$
is injective. Using the notation from the proof of Theorem A, this yields
$$\begin{aligned}
\# \P^{a_n}(nF) & \geq \# \cc^{na_n} (M).
\end{aligned}$$
Recalling the argument in section~\ref{section:negative curvature}, we have that
$$
\# \P^{a_n}(nF) \geq \frac{e^{h_{top} (g)  \sqrt{2 n a_n}}}{2 \sqrt{2na_n}}
$$
where $h_{top}(g)$ denotes the topological entropy of the geodesic flow. 
The proof of the Proposition follows in view of the definitions.

 \clearemptydoublepage

\appendix

\chapter{Convexity}
\label{convexity}
\ni 
In this chapter we follow the work of Biran, Polterovich and Salomon (see \cite{BPS03}) and Frauenfelder and Schlenk (see \cite{FS07}) in order to introduce useful tools for the compactness of moduli spaces introduced in section \ref{section:definition of hf}.\\ 

\ni 
Consider a $2n$-dimensional compact symplectic manifold $(N,\om)$ with non empty boundary $\partial N$. The boundary $\partial N$ is said to be {\it convex} if there exists a Liouville vector field $Y$ i.e. $\L_Y\om = d\io_Y\om = \om$, which is defined near $\partial N$ and is everywhere transverse to $\partial N$, pointing outwards. 

\begin{definition} {\rm (cf \cite{EG91})}
(i) A compact symplectic manifold $(N,\om)$ is convex if it has non-empty convex boundary.

(ii) A non-compact symplectic manifold $(N,\om)$is convex, if there exists an increasing sequence of compact, convex submanifolds $N_i \subset N$ exhausting $N$, that is
$$N_1\subset N_2 \subset \ldots \subset N_i \subset \ldots \subset N \quad \mbox{and} \quad \bigcup_i N_i = N.$$
\end{definition}

\ni
Cotangent bundles over a smooth manifold $M$ are examples of exact convex symplectic manifolds. In fact, the $r$-disc bundle $D(r)$
$$D(r) = \{(q,p) \in T^*M \mid |p| \leq r \}$$
is a compact, convex submanifold and $T^*M = \cup_{k\in\N} D(k)$.\\

\ni
Let $(\overline{N}, \om)$ be a compact, convex symplectic manifold and denote $N=\overline{N}\backslash \partial \overline{N}$. Choose a smooth vector field $Y$ and a neighborhood $U$ of  $\partial \overline{N}$ such that $\L_Y \om = \om$ on $U$. Denote by $\vp^t$ the flow of $Y$, suppose that $U= \{ \vp^t(x) \in \partial\overline{N} \mid -\ve < t \leq 0 \}$, and denote by $\xi := \ker (\io (Y)\om|_{T\partial\overline{N}}$ the contact structure on the boundary determined by $Y$ and $\om$. Then there exists an $\om$-compatible almost complex structure $J$ on $\overline{N}$ such that
\begin{equation}
 J\xi = \xi,
 \label{convex acs1}
\end{equation}
\begin{equation}
\om(Y(x), J(x) Y(x)) = 1, \quad \mbox{for } x\in \partial\overline{N} \mbox{ and}
\label{convex acs2}
\end{equation}
\begin{equation}
D\vp^t(x) J(x) = J(\vp^t(x))D\vp^t (x), \quad \mbox{for } x\in \partial\overline{N} \mbox{ and } t\in (-\ve, 0].
\label{convex acs3}
\end{equation}
Such an almost complex structure is called {\it convex near the boundary}. We recall that an almost complex structure $J$ on $\overline{N}$ is called $\om$-compatible if 
$$ \langle\cdot,\cdot\rangle \equiv g_J (\cdot,\cdot) := \om(\cdot, J \cdot)$$ defines a Riemannian metric on $\overline{N}$. \\

\ni
Consider the function $f :U\ra \re$ given by 
$$ f(\vp^t(x)) := e^t,$$
for $x\in \partial \overline{N}$ and $t\in (-\ve, 0]$. Since $\L_Y\om = \om$ on $U$, we have $(\vp^{t})^*\om = e^t \om$ on $U$ for all $t \in (-\ve, 0]$. Hence, \eqref{convex acs2} and \eqref{convex acs3} yield 
\begin{equation}
\langle Y(v), Y(v)\rangle = f(v), \quad \mbox{for } v\in U.
\label{grad f1}
\end{equation}
Together with \eqref{convex acs1} this implies that 
\begin{equation}
\nabla f(v) = Y(v), \quad v\in U
\label{grad f2}
\end{equation}
where $\nabla$ is the gradient with respect to the metric $\langle \cdot,\cdot \rangle$. With these properties, we can now give the following theorem due to Viterbo, (see \cite{Vit92}).

\begin{thm}
For $h\in C^\infty (\re)$ define $H\in C^\infty(U)$ by 
$$H(v) = h(f(v)), \quad \mbox{for } v\in U$$
Let $\Om$ be a domain in $\C$ and let $J\in \Ga (\overline{N}\times\Om, {\rm End}(T\overline{N}))$ be a smooth section such that $J_z := J(\cdot, z)$ is an $\om$-compatible convex almost complex structure. If $u\in C^\infty (\Om, U)$ is a solution of Floer's equation 
\begin{equation}
\partial_z u(z) + J (u(z),z) \partial_t u(z) = \nabla H(u(z)), \quad z = s+it \in \Om,
\label{floereq conv}
\end{equation}
 then
 \begin{equation}
 \nabla(f(u)) = \langle \partial_s(u), \partial_s(u) \rangle + h'' (f(u)) \cdot \partial_s (f(u)) \cdot f(u).
 \end{equation}
 \label{convex1}
\end{thm} 

\begin{proof}
We abbreviate $d^c (f(u)) := d(f(u))\circ i = \partial_t(f(u)) ds - \partial_s(f(u)) dt$. Then 
\begin{equation}
-dd^c(f(u)) = \nabla (f(u)) \, ds\wedge dt.
\label{ddc}
\end{equation}
In view of the identities \eqref{grad f1}, \eqref{grad f2} and \eqref{floereq conv} we compute
$$\begin{aligned}
-dd^c (f(u))  =  &\; - (df(u)) \partial_s u)\, ds + (df(u)\partial_s u )\,dt \\
                   = &\; -\bigl( df(u) \left(J(u,z) \partial_tu \right) \bigr) \, dt - \bigl( df(u) \left(J(u,z) \partial_s u\right) \bigr) \, ds \\
                      &\; + \bigl( df(u) \left( \partial_s u + J(u,z) \partial_t u \right) \bigr) \, dt + \bigl( df(u) \left( J(u,z)\partial_s u -  \partial_t u \right) \bigr) \, ds \\
                     = &\; \om \left( Y(u),\partial_t u \right) \, dt + \om \left( Y(u),\partial_s u \right) \, ds \\
                      &\; + \langle \nabla f(u) , \nabla H(u) \rangle \, dt +  \langle \nabla f(u) , J(u,z) \nabla H(u) \rangle \, ds \\
                      = &\; u^* \io_Y \om +  \langle Y(u) , h'(f(u))
                     Y(u) \rangle \, dt +0 \\
                      = &\; u^* \io_Y \om + h'(f(u)) f(u) \, dt.
\end{aligned}$$
As $d\io_Y\om = \L_Y\om = \om$ we obtain with \eqref{floereq conv} that
$$\begin{aligned}
du^*\io_Y\om = u^*\om = &\; \om \bigl(\partial_s u , J(u,z) \partial_s u - J(u,z) \nabla H(u) \bigr) \, ds\wedge dt \\
 = & \; \bigl( \langle \partial_s u, \partial_s u \rangle - dH (u) \partial_s u \bigr) \, ds \wedge dt \\
 = & \; \bigl( \langle \partial_s u, \partial_s u \rangle - \partial_s (h(f(u))) \bigr) \, ds \wedge dt.
\end{aligned}$$
Together with the previous equality it follows that
$$\begin{aligned}
-dd^c (f(u)) = &\; \bigl( \langle \partial_s u, \partial_s u \rangle - \partial_s (h(f(u)))  + \partial_s (h'(f(u)) f(u)) \bigr) \, ds \wedge dt \\
 = & \; \bigl( \langle \partial_s u, \partial_s u \rangle + h''(f(u)) \cdot \partial_s f(u) \cdot f(u) \bigr) \, ds\wedge dt,
\end{aligned}$$
and hence \eqref{ddc} yields the statement of the theorem.
\end{proof}

\begin{rmk}
\label{time dependent ham}
{\rm (Time dependent Hamiltonian). Repeating the calculation in the proof of Theorem~\ref{convex1}, one shows the following more general result.} Let $h\in C^\infty (\re^2,\re)$ and define $H \in C^\infty(U \times \re)$ by 
$$ H(v,s) = h(f(v),s), \quad v\in U, s\in \re.$$
If $\Om$ is a domain in $\C$ and if $u\in C^\infty (\Om, U)$ is a solution of the time-dependent Floer equation
\begin{equation}
\partial_s u(z) + J(u(z),z) \partial_t u(z) = \nabla H(u(z),s), \quad z=s+it \in \Om,
\label{floereqt conv}
\end{equation}
then
\begin{equation}
\nabla f(u) = \langle \partial_s u , \partial_s u \rangle + \partial^2_1 h(f(u),s) \cdot \partial_s f(u) \cdot f(u) + \partial_1 \partial_2 h(f(u),s) \cdot f(u).
\end{equation}
\end{rmk}

\ni
Theorem~\ref{convex1} implies

\begin{cor}
{\rm (Maximum Principle).} Assume that $u\in C^\infty (\Om, U)$ and that either $u$ is a solution of Floer's equation \eqref{floereq conv} or $u$ is a solution of the time-dependent Floer equation \eqref{floereqt conv} and $\partial_1 \partial_2 h \geq 0$. Then if $f\circ u$ attains its maximum on $\Om$, we have that $f\circ u$ is constant.
\label{compactness}
\end{cor}

\begin{proof}
Assume that $u$ solves equation \eqref{floereq conv}. We set 
$$ b(z) := -h''(f(u(z))) \cdot f(u(z)).$$
The operator $L$ on $C^\infty(\Om,\re)$ defined by $L(w) = \nabla w + b(z) \partial_s w$ is uniformly elliptic on relatively compact domains in $\Om$, and according to Theorem~\ref{convex1}, $L(f\circ u) \geq 0$. If $f\circ u$ attains its maximum on $\Om$, the strong Maximum Principle (see \cite[Theorem 3.5]{GT83}) thus implies that $f\circ u$ is constant. The other claim follows similarly from the second part of the strong Maximum Principle and Remark \ref{time dependent ham}.
\end{proof}

 \clearemptydoublepage

\chapter{Legendre transform}
Let $M$ be a smooth closed manifold of dimension $m$. Let $T^*M$ be the corresponding cotangent bundle. We will denote local coordinates on $M$ by $q=(q_1,\ldots,q_m)$, and on $T^*M$ by $x=(q,p)=(q_1,\ldots,q_m,p_1,\ldots,p_m)$. \\

\ni
Consider a Hamiltonian $H: S^1 \times \Tm \ra \re$ such that 
$$
\det \left( \frac{\partial^2 H}{\partial p_j \partial p_k} \right) \neq 0.
$$
We can introduce the variable $v$ defined by
\begin{equation}
\label{e:v}
v:= \frac{\partial H}{\partial p}.
\end{equation}
The {\it Legendre transform} of $H$ is given by the {\it Lagrangian} $L:S^1\times TM \ra \re$,
$$
L (t,q,v) := \sum_{j=1}^m p_j v_j - H (t,q,p).
$$
The Legendre transform $(t,q,p) \mapsto (t,q,v)$ establishes a one-to-one correspondence between the solutions of the first order Hamiltonian system on $\Tm$
$$
\begin{cases}
\dot{q}=\phantom{-} \partial_p H(t,q,p) \\
\dot{p}= -\partial_q H(t,q,p) 
\end{cases}
$$
and the second order Lagrangian system on $M$
\begin{equation}
\frac{d}{dt} \frac{\partial L}{\partial v} = \frac{\partial L}{\partial q}.
\label{e:lagrange}
\end{equation}
The set of $1$-periodic solutions of the {\it Euler-Lagrange equation} \eqref{e:lagrange} is the set of critical points of the Lagrangian action functional $\E_L : \La M \ra \re$ given by
$$
\E_L (q):= \int_0^1 L(t,q(t),\dot{q}(t)) \, dt.
$$
Consider the following Hamiltonian
$$
H(t,q,p) = \be \demi |p|^2 + W(t,q).
$$
Then its Legendre transform is 
$$
L(t,q,v) = \frac{1}{\be} \demi |v|^2 - W(t,q).
$$
Assume $x(t) = (q(t), p(t)) \in \P(H)$. Then using Lemma \ref{action} and \eqref{e:v}
$$\begin{aligned}
\A_H(x) & = \int_0^1 \be \demi |p(t)|^2 - W(t,q(t)) \, dt\\
	     & = \int_0^1 \frac{1}{\be} \demi |\dot{q}(t)|^2 - W(t,q(t)) \, dt \\
	     & = \E_L (q(t)) .
\end{aligned}$$
Thus $\A_H$ and $\E_L$ have the same critical values.
We introduce the notation
$$
\La^a_L :=\{q\in \La M \mid \E_L(q) \leq a \}. 
$$

\begin{lem}
\label{lem:legendre}
Let $H = \be \demi |p|^2 + W(t,q)$ with $\| W\|_{C^1} < c$. Assume that $[a-c, a+c]$ does not belong to the action spectrum of $H$. Then $\La^a_L$ retracts on $\La^{\be a}$. 
\end{lem}
 
\begin{proof}
Let $q \in \La^{\be a}$, then
$$\begin{aligned}
\E_L (q) & = \int_0^1 \frac{1}{\be} \demi |\dot{q}|^2 - W(t,q) \,dt \\
	& \leq a - \int_0^1 W(t,q) \, dt \\
	& < a+c.
\end{aligned}$$
Similarly if $\E (q) > \be a$ then $\E_L(q) > a - c$. Thus
$$
\La^{a-c}_L \subset \La^{\be a} \subset \La^{a+c}_L.
$$

\ni
Now note that $\E_L$ is smooth on $\La M$ and satisfies the Palais--Smale condition, see \cite{Benci86}. We can thus apply Lemma 2 of \cite{GM169} which tells us that, as $[a-c, a+c]$ does not contain any critical value of $\E_L$, $\La^{a-c}_L$ and $\La^{a+c}_L$ are homotopy equivalent to $\La^{a}_L$.   
\end{proof}

 \clearemptydoublepage

\chapter{Gromov's theorem}
\label{appendix:gromov's work}

\ni
Let $M$ be a compact simply connected manifold and consider its free loop space $\La M = W^{1,2}(S^1,M)$. Let $g$ be a Riemannian metric on $M$. Recall that the energy functional $\E:= \E_g: \La M \ra\re$ is defined by 
$$
\E(q) = \demi \int_0^1 |\dot{q}(t)|^2 \, dt
$$
where $|\dot{q}(t)| = g_{q(t)} (\dot{q}(t),\dot{q}(t))$. The length functional $\L := \L_g: \La M \ra\re$ is defined by 
$$
\L(q) = \int_0^1 |\dot{q}(t)| dt.
$$
For $a>0$ we consider the sublevel sets 
$$
\La^a := \{q\in \La M \mid \E(q) \leq a\}
$$
and
$$
\L^a := \{q\in \La M \mid \L(q) \leq a\}.
$$

\begin{thm}{\rm(Gromov).}
Given a Riemannian metric $g$ on $M$ there exists a constant $\kappa = \kappa (g) >0$, such that $H_j(\La^{\kappa t^2}; \F_p) \ra H_j (\La M; \F_p)$ is surjective for $j\leq t$.
\end{thm}

\ni
The proof of this theorem follows from the following lemmata.

\begin{lem}
\label{lem:gromov 1}
Given a Riemannian metric $g$ on $M$ there exists a constant $\kappa = \kappa (g)>0$, such that every element in $H_j(\La M; \F_p)$ can be represented by a cycle in $\L^{\kappa j}$.
\end{lem}

\ni
Gromov's original proof of this result is very short. A detailed proof for the based loop space can be found in \cite{Pa97}, we will follow this work and the proof by Gromov in \cite[Chapter 7A]{Gr07}.  

\begin{proof}
Let $\{V_\al\}_{\al \in I}$ be a finite covering of $M$ by geodesically convex open sets. Consider a triangulation $T$ of $M$ such that each closed simplex lies in one of the $V_\al$. We assume that the 1-skeleton of $T$ consist of geodesic segments. For each $p\in M$ we define $T(p)$ as the closed face of $T$ of minimal dimension that contains $p$. For example, if $p$ is a vertex then $T(p) = \{p\}$. We also define $O(p)$ as the union of all the maximal simplices of $T$ that contain $p$.\\

\ni
Given $k\in\N$, we define open subsets $\La_k \subset \La M$ as follow. A loop $q$ belongs to $\La_k$ if for each $j=1, 2,\ldots,2^k$, 
$$
q \left( \left[\frac{j-1}{2^k},\frac{j}{2^k} \right]\right) \subset V_\al
$$
for some $\al \in I$ and
$$
O \left(q \left(\frac{j-1}{2^k} \right) \right) \cup O\left(q \left(\frac{j}{2^k}\right)\right) \subset V_\al
$$
for the same $\al$.

\ni
Let $B_k$ be the subset of loops $\ga\in\La_k$ such that $\ga$ is a broken geodesic and $\ga$ restricted to each subinterval $[(j-1) / 2^k,j/2^k]$ is a constant speed parametrized geodesic. Then each $\ga\in B_k$ determines a sequence
$$
\{p_j = \ga ( \,j/ 2^k) \}
$$
with the following properties:
\begin{itemize}
\item[(i)] $p_0 = p_{2^k}$ ;
\item[(ii)] $O(p_{j-1}) \cup O(p_j)$ lies in a single $V_\al$ for each $j=1,2,\ldots, 2^k$ .
\end{itemize}
Conversely each sequence $\{p_j\}$ with properties (i) and (ii) determines a broken geodesic loop in $B_k$. Moreover this correspondence is bijective. This yields a cell decomposition on $B_k$ as follow: a cell that contains $\ga$ is given by
$$
T(p_1)\times T(p_2) \times \ldots \times T(p_{2^k}).
$$
Hence $B_k$ is a finite cell complex. Moreover, using the methods of Milnor in \cite[Section 16]{Mi63}, we have that $B_k$ is a deformation retract of $\La_k$.\\

\ni
Since $M$ is simply connected, there exists a smooth map $f:M\ra M$ such that $f$ collapses the 1-skeleton of $T$ to a point and $f$ is smoothly homotopic to the identity. This map $f$ naturally induces a map $\hat{f} : \La M \ra \La M$. We need the following lemma.

\begin{lem}
\label{lem:gromov 2}
There exist a constant $\kappa >0$ such that for any integer $k >0$, we have 
$$
\hat{f} (\mbox{$i$-skeleton of } B_k) \subset \L^{\kappa i}
$$
for all $i \leq \dim B_k$.
\end{lem}

\begin{proof}
Consider a cell
$$
C = T(p_1)\times T(p_2) \times \ldots \times T(p_{2^k})
$$
of dimension $i \leq \dim B_k$. Take a path $\ga \subset C$, then $\ga$ is a broken geodesic with each of its legs lying in one $V_\al$, $\al \in I$. Consider the following constants:
$$
\begin{aligned}
K &:= \max_{x\in M} \parallel d_x f \parallel \\
d &:= \max_{\al \in I} \{g\mbox{-diameter of } V_\al\} \\
N(\ga) &:= \# \{\mbox{legs of $\ga \nsubseteq$ 1-skeleton} \}.
\end{aligned}
$$ 
Since $f$ collapses the 1-skeleton of $T$ to one point it holds that 
$$
\L \left(\hat{f}(\ga)\right) \leq K \cdot d\cdot N(\ga).
$$
By assumption the 1-skeleton consists of geodesic segments, thus the legs of $\ga$ from $T(p_j)$ to $T(p_{j+1})$ belong to the 1-skeleton if $1\leq j < 2^k -1$ and $\dim T(p_j) = \dim T(p_{j-1}) =0$. Equivalently, the only legs that do not belong to the 1-skeleton are the legs which begin or end in a $T(p_j)$ of nonzero dimension. Thus 
$$
N(\ga) \leq 2i.
$$
If we set $\kappa = 2Kd$ we obtain
$$
\L \left(\hat{f}(\ga) \right) \leq \kappa i
$$
which concludes the proof of the lemma.   
\end{proof}

\ni
We shall show that any $\eta \in H_i(\La M)$ can be represented by a cycle whose image lies in $\L^{\kappa i}$, where $\kappa$ is the constant given by Lemma \ref{lem:gromov 2}. Since $f$ is a surjective map, $f$ has degree one since it is homotopic to the identity, this implies Lemma~\ref{lem:gromov 1}. \\

\ni
Observe that for all $k \in \N$
$$
\La_k \subset \La_{k+1}   
$$
and 
$$
\La M = \bigcup_{k=1}^\infty \La_k.
$$
Since $f$ is homotopic to the identity, there exists $\mu \in H_i(\La M)$ such that $\hat{f}_*(\mu) = \eta$. Let $C$ be a cycle that represents $\mu$. Then its image will lie in $\La_k$ for some $k$.  Retract $\La_k$ onto $B_k$. Then we can move $C$ by a homotopy into the $i$-skeleton of $B_k$. By Lemma~\ref{lem:gromov 2} $\hat{f}$ maps all points of the $i$-skeleton of $B_k$ to points in $\L^{\kappa i}$ and in particular the image of $C$  will lies in $\L^{\kappa i}$. Hence $\eta = \hat{f}_*(\mu)$ can be represented by a cycle whose image lies in $\L^{\kappa i}$. 
\end{proof}

\begin{lem}
There exist a constant $\kappa := \kappa(g)>0$ depending only on $g$ such that each element of $H_k(\La M; \F_p)$ can be represented by a cycle in $\La^{\demi (\kappa k)^2}$.  
\end{lem}

\begin{proof}
Let $\Delta^k$ be the standard $k$-simplex, and let 
$$
\psi= \sum_i n_i \psi_i : \Delta^k \ra\L^{(\kappa -2)k}
$$
be an integral cycle, where $\kappa$ is the constant given by Lemma \ref{lem:gromov 1}. For convenience of notation we pretend that $\psi$ consists only on one simplex. As $W^{1,2}(S^1,M)$ is a completion of $C^\infty(S^1,M)$, see \cite{Kli95}, 
we can replace $\psi$ by a homotopic and hence homologous cycle 
$$
\psi_1: \Delta^k \ra \L^{(\kappa -1)k}
$$
consisting of smooth loops. We identify $\psi_1$ with the map 
$$\Delta^k \times S^1 \ra M, \quad (s,t)\mapsto \psi_1(s,t)=(\psi_1(s))(t).$$
Endow the manifold $M\times S^1$ with the product Riemannian metric. We lift $\psi_1$ to the cycle $\widetilde{\psi}_1 : \Delta^k \ra \La(M \times S^1)$ defined by 
$$
\widetilde{\psi}_1(s,t) = (\psi_1(s,t),t).
$$
This cycle consists of smooth loops whose tangent vectors do not vanish. For each $s$ let $\widetilde{\psi} (\si(s))$ be the reparametrization of $\widetilde{\psi}(s)$ proportional to arc length. The homotopy $\Psi : [0,1] \times \Delta^k \ra \La (M\times S^1)$ defined by 
$$
(\Psi(\tau,s))(t) = \widetilde{\psi}_1 (s,(1-\tau) t +\tau \si(s)) 
$$
shows that $\widetilde{\psi}_1$ is homologous to the cycle $\widetilde{\psi}_2(s) := \Psi (1,s)$. Its projection $\psi_2$ to $\La M$ is homologous to $\psi_1$ and lies in $\L^{(\kappa -1)k}$. Since for each $s$ the loop $\widetilde{\psi}_2(s)$ is parametrized proportional to the arc length, we conclude that 
$$\begin{aligned}
\E(\psi_2 (s)) \leq \E(\widetilde{\psi}_2(s)) & = \demi \left(\L\big (\widetilde{\psi}_2(s)\big)\right)^2 \\
 & = \demi \left( \big(\L (\psi_2(s)) \big)^2 +1 \right) \\
 & \leq \demi (\kappa -1)^2 k^2 +\demi \\
 & \leq \demi (\kappa k )^2
\end{aligned}$$
for each $s$, so that indeed $\psi_2 \in \La^{\demi (\kappa k)^2}$. This concludes the proof.
\end{proof}

 \clearemptydoublepage

\addcontentsline{toc}{chapter}{\protect\numberline{}{Bibliography}}

\pagestyle{fancy}
\fancyhf{} 
\fancyhead[LE,RO]{\bfseries\thepage} \fancyhead[LO]{\bfseries
Bibliography} \fancyhead[RE]{\bfseries Bibliography}
\renewcommand{\headrulewidth}{0.5pt}
\renewcommand{\footrulewidth}{0pt}
\addtolength{\headheight}{0.5pt} 
\fancypagestyle{plain}{ 
\fancyhead{} 
\renewcommand{\headrulewidth}{0pt} 
}
\bibliographystyle{plain}
\bibliography{mybiblio}

\end{document}